\documentclass[11pt,a4paper]{article}
\usepackage[a4paper,left=2.5cm,right=2.5cm,top=2.5cm,bottom=2.5cm]{geometry} 
 
 \usepackage{amsmath,amssymb,amsthm,amsfonts,subcaption,caption}
\usepackage[colorlinks=True,
            linkcolor=blue,
            anchorcolor=green,
            citecolor=blue
            ]{hyperref} 
\usepackage[noabbrev,capitalise,nameinlink]{cleveref}
\usepackage{float}
\usepackage{graphicx}
\usepackage{epstopdf}
\usepackage{multirow} 
\usepackage{enumitem,url}
\usepackage{threeparttable}  
\usepackage{booktabs,longtable, pdflscape}

\usepackage[ruled,vlined]{algorithm2e}
\usepackage{etoolbox}

\usepackage{xcolor}

 \newtheorem{theorem}{Theorem}[section]

\newtheorem{lemma}{Lemma}[section]

\newtheorem{corollary}{Corollary}[section]
\newtheorem{definition}{Definition}[section]
\newtheorem{remark}{Remark}[section]
\newtheorem{assumption}{Assumption}[section]

\newtheorem{example}{Example}[section]

\usepackage{hyperref}
\newcommand{\footremember}[2]{%
    \footnote{#2}
    \newcounter{#1}
    \setcounter{#1}{\value{footnote}}%
}

\def\B{{\mathbb B}}

\def\N{{\mathbb N}}

\def\R{{\mathbb  R}}

\def\C{\mathbb{C}}
\def\H{{\bf  H}}
\def\I{{\bf  I}}

\def\Q{{\bf  Q}}

\def\v{{\bf v}} 
\def\w{{\bf w}}
\def\z{{\bf z}}
\def\x{{\bf x}}

\def\y{{\bf y}}
\def\m{{\bf m}}

\def\bze{{\boldsymbol  \nu}}

\author{
  Shenglong Zhou\footremember{bjtu0}{Corresponding author: Shenglong Zhou}\footremember{bjtu1}{School of Mathematics and Statistics, Beijing Jiaotong University, Beijing 100044, China.}~~
  Shuai Li\footnotemark[2] ~~
  Hui Zhang\footremember{qufu}{School of Management Science, Qufu Normal University, Rizhao 276826, China.}~~
  Ziyan Luo\footnotemark[2]
  \footremember{bjtu2}{Email: shlzhou@bjtu.edu.cn, 24110488@bjtu.edu.cn, zhanghuiopt@163.com, luozy@bjtu.edu.cn}
}

\title{\vspace{-1.25cm}
 Sharp-Peak Functions for Exactly Penalizing \\Binary Integer Programming
\vspace{-0.25cm}}

\date{}

\begin{document}
\flushbottom
 
\maketitle
 
\vspace{-1.3cm}

\begin{abstract}  
\noindent \textbf{Abstract:} Unconstrained binary integer programming (UBIP) is a challenging optimization problem due to the presence of binary variables. To address the challenge, we introduce a novel class of functions named sharp-peak functions (SPFs), which equivalently reformulate the binary constraints as equality constraints, giving rise to an SPF-constrained optimization. Rather than solving this constrained reformulation directly, we focus on its associated penalty model.  The established exact penalty theory shows that the global minimizers of UBIP and the penalty model coincide when the penalty parameter exceeds a threshold, a constant independent of the solution set of UBIP. To analyze the penalty model, we introduce Karush-Kuhn-Tucker (KKT) points and a new type of stationarity, referred to as P-stationarity, and provide a comprehensive characterization of its optimality conditions.  We then develop an efficient algorithm called ShaPeak based on the inexact  alternating direction method of multipliers. 
  It converges to a P-stationary point at a linear rate or terminates at such a point within finitely many steps. These results are established under appropriate parameter choices and a single mild assumption, namely, the local Lipschitz continuity of the gradient over a bounded box. Finally, numerical experiments demonstrate its nice performance in comparison to several established solvers.

\vspace{0.3cm} 
 
\end{abstract}

\numberwithin{equation}{section}

\section{Introduction}\label{Section-Introduction}
Binary integer programming (BIP) is a fundamental optimization problem with extensive applications across various domains. These include classical combinatorial problems such as maximum satisfiability \cite{bacchus2021maximum,hansen1990algorithms}, as well as graph-based optimization problems such as graph coloring \cite{tabi2020quantum} and Max-Cut \cite{bur02,kochenberger2013solving}. It has also attracted significant attention in modern machine learning, including adversarial attacks \cite{esm19,wee21}, transductive inference \cite{joa99}, and binary neural networks \cite{las16,qin20}. For additional applications, we refer to \cite{liu23,pan23}. The unconstrained version of BIP is given by
\begin{equation}
\label{UBIP}\min\limits_{\x\in\R^n}~f(\x),~~{\rm s.t.}~~\x\in\{0,1\}^n,\tag{UBIP}
\end{equation}
where function ${f:\R^n\to\R\cup\{\infty\}}$ is continuously differentiable.  However,  UBIP is well-recognized as an NP-hard problem due to the presence of binary variables \cite{gar79,koc14}. Instead of directly solving \eqref{UBIP},  we aim to address the following  sharp-peak function constrained optimization (SPCO), 
\begin{equation}\label{SCO}
\begin{aligned}
\min\limits_{x\in{\R^n}}~ f(\x),~~{\rm s.t.}~~& g(x_i) =0,~0 \leqslant x_i \leqslant 1,~i\in[n], 
\end{aligned} \tag{SPCO}
\end{equation}
where $[n]:=\{1,2,\cdots,n\}$ and $g:\R\to\R\cup\{\infty\}$ is a sharp-peak function as outlined in Definition \ref{def-SPF}. In general, for any $x\in[0,1]$, it satisfies $g(x)\geqslant 0$ and $g(x)=0$ if and only if $x\in\{0,1\}$.

 \begin{table}[!t]
\renewcommand{\arraystretch}{1.0}\addtolength{\tabcolsep}{0pt}
\centering
\caption{Main techniques to process the binary constraints. \label{setupdiff}}
\begin{tabular}{lll}\hline
 Methods & References &Reformulations \\
\hline
\multicolumn{3}{c}{Relaxation methods}\\\hline
 LP relaxation &\cite{Hsieh15,Kom07}& $\{0,1\}^n \approx\{\x\mid 0 \leqslant \x \leqslant 1\}$ \\

 Spectral relaxation &\cite{Cour07,Shi07} & $\{0,1\}^n \approx\left\{\x \mid\|2\x-1\|_2^2=n\right\}$\\

 SDP relaxation &\cite{buc19,Wang16}  & $\{0,1\}^n \approx\left\{\x \mid \textbf{X} \succeq \x \x^\top, \operatorname{diag}( \textbf{X} )=\x\right\}$ \\

Doubly positive relaxation &\cite{Huang14,Wen10} & $\{0,1\}^n \approx\left\{\x \mid  \textbf{X}  \succeq (2\x-1) (2\x-1)^\top, \operatorname{diag}( \textbf{X} )=1\right\}$ \\\hline 
\multicolumn{3}{c}{Equivalent reformulations}\\\hline
 Piecewise &\cite{Zhang07,lucidi2010exact,murray2010algorithm} & $\{0,1\}^n \Leftrightarrow\{\x \mid\x\odot(1-\x)=0\}$ \\ 
 
  $\ell_2$ box &\cite{Murray10,Rag69} & $\{0,1\}^n \Leftrightarrow\left\{\x \mid 0 \leqslant \x \leqslant 1,\|2\x-1\|_2^2=n\right\}$ \\

  $\ell_2$ box MPEC  &\cite{Yuan17} & $\{0,1\}^n \Leftrightarrow\left\{\x \mid 0 \leqslant \x \leqslant 1, \|\v\|^2 \leqslant n,\langle 2\x-1,\v\rangle = n\right\}$ \\


  $\ell_p$ box ($p>0$) &\cite{Wu19} & $\{0,1\}^n \Leftrightarrow\left\{\x \mid0 \leqslant \x \leqslant 1,\|2\x-1\|_p^p=n, p>0\right\}$ \\

$\ell_0$ norm &\cite{yuan2016sparsity} & $\{0,1\}^n \Leftrightarrow\left\{\x \mid\|\x\|_0+\|\x-1\|_0 \leqslant n\right\}$ \\
 SPF & This paper & $\{0,1\}^n \Leftrightarrow\left\{\x \mid  0 \leqslant \x \leqslant 1, g_i(\x)=0, i\in[n]\right\}$ \\
\hline
\end{tabular}
\end{table}
\subsection{Related work}  

Strategies for handling binary constraints can broadly be classified into two categories: relaxation methods and equivalent reformulations, see Table \ref{setupdiff}. We provide a brief overview below.

 Relaxation methods aim to convert binary constraints into continuous optimization models, making them more tractable. One standard approach was linear programming relaxation \cite{Hsieh15, Kom07}, where the binary constraint was replaced by a box constraint. This transformation turned the original NP-hard problem into a convex optimization problem with box constraints, which can be solved efficiently using modern optimization algorithms. Further advancements include spectral relaxation \cite{Cour07, Shi07}, which substituted the binary constraint with a spherical constraint. The resulting problem is then solved through eigenvalue decomposition.  Another notable contribution relaxed the binary constraint by a convex cone to derive the semi-definite programming (SDP) formulations. Then these SDP can be solved by some customized algorithms, such as interior point methods \cite{wolkowicz2012handbook, anjos2011handbook},  quasi-Newton and smoothing Newton methods \cite{Wang16}, a spectral bundle method \cite{helmberg2000spectral}, branch-and-bound algorithms \cite{buchheim2013semidefinite, buc19}.   To refine the relaxation,  a doubly positive relaxation method \cite{Huang14,Wen10} enforced non-negativity constraints on both the eigenvalues and the elements of the SDP solution. Numerical results have demonstrated that this approach can significantly improve the solution quality compared to traditional SDP relaxation methods.
 
 Equivalent reformulations provide another approach to addressing binary constraints in optimization problems. For example, piecewise constraint $\{\x \mid \x \odot (1-\x)=0\}$ has been widely used in the literature for characterizing binary variables and reformulating binary optimization problems. In particular, it was employed in \cite{Zhang07} to transform the UBIP into an equality-constrained optimization problem that can be treated by penalty methods. See also \cite{lucidi2010exact,murray2010algorithm} for related formulations and algorithms. Similarly, the authors in \cite{Rag69} reformulated the UBIP as a continuous optimization problem using the $l_2$ norm, allowing it to be solved with second-order interior-point methods \cite{Mar12, Murray10}. Then this approach has been extended by \cite{Yuan17}, where the authors incorporated $l_2$ norm within the framework of mathematical programming with equilibrium constraints. This reformulation led to an augmented biconvex optimization problem with a bilinear equality constraint, which was effectively solved by a penalty method.  In a broader context, \cite{Wu19} proposed a continuous $l_p$-norm (for  $p>0$) boxed reformulation and used the alternating direction method of multipliers to solve it.

\subsection{Contribution and organization} 
In this paper, we introduce a sharp-peak function (SPF) to derive  model \eqref{SCO}. Consequently, our approach serves as an equivalent reformulation, as shown in Table \ref{setupdiff}. Most importantly, the new model provides strong theoretical guarantees and enables the development of an efficient numerical algorithm. In other words, it offers a practical and effective continuous optimization approach for solving the discrete programming problem, \eqref{UBIP}. The main contributions are fourfold.

\textit{1) The invention of SPFs}. We introduce SPFs to equivalently transform binary constraints into equality constraints. These functions are highly general, capable of accommodating a wide range of fundamental functions.  Importantly,  when the penalty method is applied to solve problem \eqref{SCO}, SPFs facilitate the development of a new exact penalty theorem, a result that other existing functions, such as the piecewise function \cite{Zhang07},  $\ell_2$ norm\cite{Murray10,Rag69}, and $\ell_p$ norm with $p>1$ \cite{Wu19}, may fail to achieve. A counterexample is given in \eqref{ex-pi-penalty-1} to illustrate this.

\textit{2) New exact penalty theorems}.  To address  \eqref{SCO}, an equivalent model of \eqref{UBIP}, we investigate its penalty model by showing that the global minimizers to \eqref{SCO} and the penalty model coincide when the penalty parameter exceeds a threshold, a constant that is independent of the solution set of \eqref{SCO}. This contrasts with traditional exact penalty theory, where the penalty parameter typically depends on the solution set of the original problem.  Furthermore, we introduce the concepts of Karush-Kuhn-Tucker (KKT) points and P-stationary points of the penalty model, which allow for a more comprehensive optimality analysis by revealing their relationships with local and global minimizers for \eqref{SCO}, as illustrated in Figure \ref{fig:relation}.  These relationships suggest that targeting a P-stationary point for the penalty problem is a promising way to solve \eqref{UBIP}, leading to an effective algorithm based on the scheme of the P-stationary point.
  
\textit{3) A simple algorithm with convergence guarantees}.  
To solve the penalized problem, we propose an algorithm named ShaPeak, built upon the inexact alternating direction method of multipliers (iADMM) with an adaptively updated penalty parameter. The inexactness in solving subproblems and the adaptive penalty update pose challenges for convergence analysis. Nevertheless, we establish that the whole sequence converges to a P-stationary point  at a linear rate or terminate at this point within finitely many steps, provided that parameters are chosen properly and $\nabla f$ is locally Lipschitz continuous on a bounded box.

\textit{4) Superior numerical performance}. We benchmark ShaPeak against several state-of-the-art solvers, including the commercial solver GUROBI, which is designed for quadratic and linear problems. In addition to these problems, ShaPeak is also capable of handling non-quadratic applications effectively. The numerical experiments  demonstrate that ShaPeak achieves superior accuracy and computational efficiency.

The paper is organized as follows. In the next subsection, we introduce the notation and functions used throughout the paper. Section \ref{sec:spfs} defines SPFs and presents two special examples. Section \ref{Section-model} presents the penalty model for \eqref{SCO} and establishes the exact penalty theory. In Section \ref{Section-algorithm}, we develop the algorithm and analyze its convergence properties. The final two sections are dedicated to extensive numerical experiments and concluding remarks.

 \subsection{Preliminaries} \label{Section-pre}
The $1$-dimensional and  $n$-dimensional unit boxes are denoted by
$$B:= \{x\in\R: 0 \leqslant x \leqslant   1\},\qquad \B:=\{\x\in\R^n:  x_i\in B\},$$
where  `$:=$' means `define'. For vector $\x\in \R^n$, we use $\|\x\|$ and $\|\x\|_\infty$  to denote its Euclidean and infinity norms. For a symmetric matrix $\Q\in \R^{n\times n}$, we use $\|\Q\|$ to denote its spectral norm. The indicator function of a given set $\Omega$ is defined by  $\delta_\Omega(\x)=0$ if $\x\in\Omega$ and $\delta_\Omega(\x)=\infty$ otherwise. Let $N_{\B}(\x)$ be the normal cone \cite{rock98} of $\B$ at $\x$. Then any  $\v\in N_{\B}(\x)$ satisfies
\begin{equation}\label{normal-cone-B}
v_i \in \begin{cases} ~(-\infty,0],& {\rm if}~x_i=0,\\
~\{0\},& {\rm if}~x_i\in(0,1),\\
~[0,+\infty),& {\rm if}~x_i=1.\end{cases}
\end{equation}
For a lower semi-continuous function {$f:{\mathbb R}^n\rightarrow \R$}, the definition of the (limiting) subdifferential denoted by {$\partial f$ can be found in \cite[Definition 8.3]{rock98}. 
The proximal operator of a function $\varphi:\R^n\to\R$, associated with a parameter $\tau>0$, is defined by
\begin{equation}\label{proximal-varphi}
{\rm Prox}_{\tau \varphi}(\z):=\operatorname*{argmin}\limits_{\x\in \R^n}~ \varphi(\x)+\frac{1}{2\tau}\|\x-\z\|^2.
\end{equation}
In particular, we write
\begin{equation}\label{proximal-varphi-B}
{\rm Prox}_{\tau \varphi}^\B(\z):={\rm Prox}_{\tau \varphi+\delta_\B}(\z).
\end{equation}
 We say function $f$ is  $L$-strong smooth on $\Omega\subseteq\R^n$ if 
$$ 
f(\x) \leqslant   f(\w)+\langle \nabla f(\w), \x-\w\rangle +  \frac{L}{2}\|\x-\w\|^2, ~~\forall~\x,\w\in \Omega,
$$
where $L>0$. We say function $h:\Omega\to\R^m$ is locally Lipschitz continuous at point $\x\in\Omega$ if 
$$ 
\|h(\w) - h(\v)\|\leqslant   L(\x) \|\w-\v\| $$
holds for any $\w$ and $\v$ around $\x$, where $L(\x)>0$ is the Lipschitz constant relying on $\x$. Then function $h$ is said to be locally Lipschitz continuous on $\Omega$ if it is locally Lipschitz continuous at every point in $\Omega$. By Heine-Borel theorem, one can show that if function $h$ is locally Lipschitz continuous on a compact set $\Omega$, then $h$ is Lipschitz continuous on $\Omega$.  Moreover, if $\nabla f$ is  Lipschitz continuous on $\Omega\subseteq\R^n$ with a constant $L>0$, then $f$ is $L$-strong smooth on $\Omega\subseteq\R^n$.  A function $f$ is said to be $\ell$-strong convex on $\Omega\subseteq\R^n$ if 
$$ 
f(\x) \geqslant   f(\w)+\langle \nabla f(\w), \x-\w\rangle +  \frac{\ell}{2}\|\x-\w\|^2, ~~\forall~\x,\w\in \Omega,
$$
where $\ell>0$. Finally, we define a useful operator by
\begin{equation}\label{operator-Pi}
\Pi(r,s,u,v)= 
\begin{cases}
\{r\},& u<v,\\
\{r,s\},& u=v,\\
\{s\},&u>v.
\end{cases}
\end{equation}

 \section{Sharp-Peak Functions}\label{sec:spfs}
 We now introduce the sharp-peak function (SPF) that can be employed to equivalently reformulate the binary constraint as an equality constraint. These functions have three properties. First, they give an equivalent reformulation of the binary constraint. Second, their subdifferentials on B have a uniform lower bound, which is key to our exact penalty theory. Third, the subdifferential sum rule on the $\{0,1\}$ helps with the later optimality analysis. Specifically, it is defined as follows.
 \begin{definition}[Sharp-Peak Functions] \label{def-SPF}Function ${g:\R\to\R\cup\{\infty\}}$ is called a sharp-peak function if it satisfies the following conditions:
\begin{itemize}[leftmargin=22pt]
\item[c1)] For any $x\in B$, $g(x) \geqslant 0$ and $g(x)=0$ if and only if $x\in\{0,1\}$.
\item[c2)] $g$ is lower semi-continuous on $B$ and there exists a  $c>0$ such that
\begin{equation}\label{lower-bd-subdiff}
 |\nu|\geqslant c, ~~\forall~\nu\in\partial g(x),~\forall~x\in B.\end{equation} 
\item[c3)] $\partial (g(x)+\delta_{B}(x))=\partial g(x)+N_{B}(x)$ for any $x\in\{0,1\}$. 
\end{itemize}
 \end{definition}
According to c1) and c2), function $g$  has a sharp peak within  interval $(0,1)$, see functions in Figures \ref{fig:spf-g} and \ref{ex-spf-g}. Therefore, we call it the sharp-peak function. For c2), it allows $\partial g(x)=\emptyset$ for some $x\in B$.  Moreover, from \eqref{lower-bd-subdiff}, there is a uniform positive lower bound $c$, which is not a strict condition. In fact, if ${c_0:=\inf \{|\nu|: \nu\in\partial g(x), ~x\in B\}>0}$, then $c$ can be set as ${c=c_0}$.  For example, $g(x)=1-|2x-1|$ and $c=2$. For c3), its sufficient conditions can be mild, such as $g(x)$ being continuously differentiable or being regular at $x\in\{0,1\}$.

The definition of SPFs can encompass a broad range of basic functions. For example,  $g(x)=1-|2x-1|^p$ with $p\in(0,1]$ is a SPF. Furthermore, its composite function $\psi(g(x))$ is still a SPF if $\psi$ satisfies properties: $$a)  \text{ $\psi(0)=0$,\qquad b) 
 $\psi(t)$ is continuous differentiable on $t\in B$ and $\kappa:=\inf_{t\in B}\psi'(t)>0$}.$$ 
 Indeed, $\kappa>0$ means that $\psi$ is increasing on $B$, thereby $\psi(t)\geq \psi(0)=0$ for any $t\in B$. Moreover,  $g(x)=1-|2x-1|^p\in B$ for any $x\in B$. These two facts and condition a) imply that $\psi(g(x))\geqslant 0$ for any $x\in B$ and $\psi(g(0))=\psi(g(1))=\psi(0)=0$. Therefore, c1) holds.  For for any $x(\neq 1/2)\in B$ if $p\in(0,1)$ and $x \in B$ if $p=1$, by the (limiting) subdifferential chain rule, $\partial \psi(g(x)) \subseteq \psi'(g(x)) \partial g(x)$, which leads to  $|\psi'(g(x)) \nu|\geqslant \kappa c$ for any $\nu\in \partial g(x)$ and hence $|\zeta|\geqslant \kappa c$ for any $\zeta\in \partial \psi(g(x))$. This verifies c2). Finally, since $g$ is continuously differentiable at $\{0,1\}$, so is $\psi(g(x))$ from condition b), which results in c3) immediately. 

Based on the above claim, a class of SPFs can be given as follows,
\begin{align}\label{spf-g-abs}
\psi(1-|2x-1|^p),~p\in(0,1],
\end{align}
where $\psi$ can be various basic functions, such as, $\psi(t)=t$, $|t|$, $|t+1|^r-1$  with ${r\geq  1}$, $e^t-1$, $\sin(t)$, $\tan(t)$, and  $\log(1+t)$. See their plots in Figure \ref{fig:spf-g}. We point out that $g(x)=1-|2x-1|^p$ with $p>0$ has been studied in \cite{Wu19}. It is a SPF if $p\in(0,1]$ and is not a SPF if $p>1$ due to $g'(1/2)=0$. Moreover, piecewise function $g(x)=x(1-x)$ \cite{Zhang07, lucidi2010exact,murray2010algorithm} is not a SPF as well due to $g'(1/2)=0$.

 \begin{figure}[!t]
 	\centering 	
\includegraphics[scale=0.75]{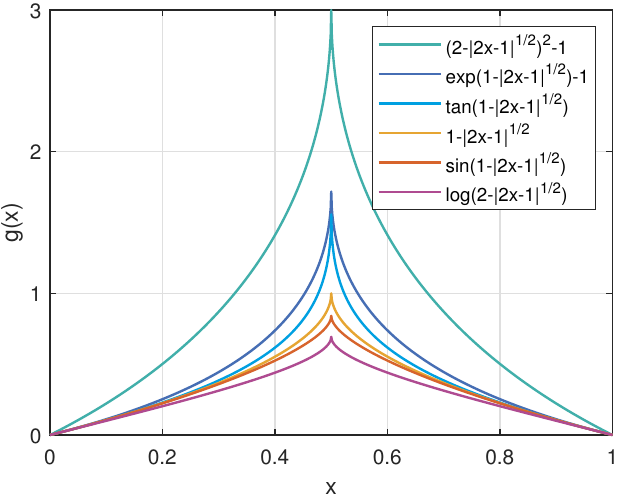}~~~
  \includegraphics[scale=0.75]{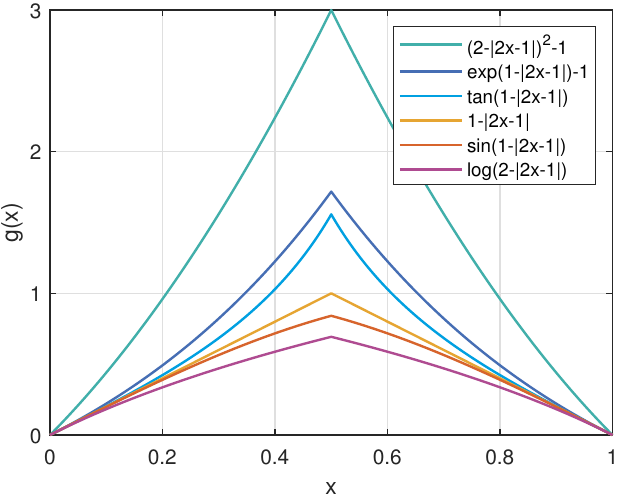}	
 	\caption{Plots of $\psi(1-|2x-1|^p)$. $\psi(t)=(1+t)^2-1$,~ $e^t-1$, $\tan(t)$, $t$, $\sin(t)$, and $\log(1+t)$. $p=1/2$ (left subfigure) and $p=1$ (right subfigure).} 
 	\label{fig:spf-g}
 \end{figure}
 
In what follows, we introduce two useful SPFs that cover a broad class of basic functions and facilitate the design of an efficient numerical algorithm.
\begin{lemma} Given $a \geqslant 1, b \geqslant 1, p> 0, q> 0$,  and $ \omega \in[0,1]$, let $g$ be defined by
\begin{equation}\label{def-g-h} 
g(x;\omega,a,b,p,q)=
\begin{cases}\dfrac{ |x+a|^p}{p}-\dfrac{a^p}{p}~~=:g_1(x,a,p),&x< \omega,\\[1.5ex]
\min\Big\{g_1(\omega,a,p),~ g_2(\omega,b,q)\Big\}, &x= \omega,\\[1.5ex]
\dfrac{ |x-1-b|^q}{q}-\dfrac{b^q}{q}~=:g_2(x,b,q),&x> \omega,\\
\end{cases}  
\end{equation} 
and $h$ be defined by
\begin{equation}\label{def-h-h}
h(x;\omega,a,b,p,q)=
\begin{cases}\dfrac{a^p}{p}-\dfrac{|x-a|^p}{p}~~=:h_1(x,a,p),&x < \omega,\\[1ex]
\min\Big\{h_1(\omega,a,p),~ h_2(\omega,b,q)\Big\}, &x= \omega,\\[1.5ex]
\dfrac{b^q}{q}-\dfrac{|x-1+b|^q}{q}~=:h_2(x,b,q),&x> \omega.
\end{cases}
\end{equation} 
Then both functions are SPFs.
\end{lemma}
\begin{proof} We prove the results for the function $g$ only, as the proof for the function $h$ follows similarly. It is easy to check that both functions are lower semi-continuous on $B$ and satisfy c1) in Definition c1). The remaining parts of c2) and c3) are proved in three cases. 

\underline{Case 1: $\omega\in(0,1)$.} Under such a case, we have
\begin{itemize}[leftmargin=15pt]
\item when  $x\in[0,\omega)$, $\partial g(x;\omega,a,b,p,q)=\{(x+a)^{p-1}\}$;
\item when  $x\in( \omega,1]$, $\partial g(x;\omega,a,b,p,q)=\{-(b+1-x)^{q-1}\}$;
\item when  $x=\omega$, by letting  
$$g_1:=\dfrac{|\omega+a|^p}{p}-\dfrac{a^p}{p}, \qquad g_2:=\dfrac{ |\omega-1-b|^q}{q}-\dfrac{b^q}{q},$$
one can check that
\begin{equation}
\partial g(x;\omega,a,b,p,q)=
\begin{cases}
[(\omega+a)^{p-1},\infty),&g_1< g_2,\\
\{(\omega+a)^{p-1},-(b+1-\omega)^{q-1}\},& g_1 = g_2,\\
(-\infty,-(b+1-\omega)^{q-1}],&g_1 > g_2.
\end{cases}
\end{equation} 
\end{itemize}
Therefore, for any $\nu\in \partial g(x;\omega,a,b,p,q)$ and any $t\in(0,1)$, it follows
\begin{equation*}
 |\nu|\geqslant \min\left\{(\omega+a)^{p-1},a^{p-1},(b+1-\omega)^{q-1},b^{q-1}\right\},
\end{equation*} 
which proves c2). We now verify c3) by the following two cases: If $x=0$,  then 
\begin{eqnarray*}
\begin{aligned}\partial \left(g(0;\omega,a,b,p,q)+\delta_{B}(0)\right)&=(-\infty, a^{p-1}]=  a^{p-1}+(-\infty, 0]\\
&=\partial g(0;\omega,a,b,p,q) + N_{B}(0). 
 \end{aligned}\end{eqnarray*}
 If $x=1$,  then
\begin{eqnarray}\label{diff-decom-1}
\begin{aligned}
\partial \left(g(1;\omega,a,b,p,q) +\delta_{B}(1)\right)&=[b^{q-1},\infty) = b^{q-1}+[0,\infty)\\
  &=\partial g(1;\omega,a,b,p,q) + N_{B}(1).
  \end{aligned}\end{eqnarray}
\underline{Case 2: $\omega=0$.} Under such a case, it holds
\begin{equation*}
\begin{aligned}
g(x;0,a,b,p,q) &=
\begin{cases}0,&~x=0,\\
\dfrac{|x-1-b|^q}{q}-\dfrac{b^q}{q},&~x\in(0,1], 
\end{cases}\\
 \partial g(x;0,a,b,p,q)&=
\begin{cases}[a^{p-1},\infty),&x=0,\\
\{-(b+1-x)^{q-1}\},&x\in(0,1]. 
\end{cases}
  \end{aligned} \end{equation*} 
Therefore, for any $\nu\in\partial g(x;0,a,b,p,q)$ and any $ x\in(0,1),$ the above condition results in
\begin{equation*}
 |\nu|\geqslant \min\left\{a^{p-1},(b+1)^{q-1},b^{q-1}\right\}.
\end{equation*} 
Moreover, one can check that  \eqref{diff-decom-1} holds for $x=1$. For $x=0$, 
\begin{equation*}
\begin{aligned}
\partial \left(g(0;0,a,b,p,q)+\delta_{B}(0)\right)&=(-\infty, \infty)=  [a^{p-1},\infty)+(-\infty, 0]\\
&=\partial g(0;0,a,b,p,q) + N_{B}(0).
  \end{aligned}\end{equation*}
Hence, c2) and c3) are satisfied. 

\underline{Case 3: $\omega=1$.} The proof follows similarly to the case where $\omega=0$.
\end{proof}

Examples of $g(x;\omega,a,b,p,q)$ and $h(x;\omega,a,b,p,q)$ can be found in Figure \ref{ex-spf-g}. When $a=b$, $p=q$, and $\omega\in(0,1)$, both functions are continuous on $B$. In many cases, they are discontinuous at $t=u$. Note that for when $\omega=1$ or $\omega=0$, functions $g$ and  $h$ on $B$ can be expressed as
\begin{equation*}
\begin{aligned}
g(x;0,a,b,p,q)&=
\dfrac{|x-(x)_{0/1}-b|^q}{q}-\dfrac{b^q}{q}, \qquad &&g(x;1,a,b,p,q)=
\dfrac{|x(1-x)_{0/1}+a|^p}{p}-\dfrac{a^p}{p},\\
h(x;0,a,b,p,q)&=
\frac{b^q}{q}-\dfrac{|x-(x)_{0/1}+b|^q}{q}, \qquad &&h(x;1,a,b,p,q)=
\frac{a^p}{p}-\dfrac{|x(1-x)_{0/1}-a|^p}{p},
\end{aligned}\end{equation*}
where $(x)_{0/1}$ is the step function or 0/1 loss function \cite{wang21, zhou21} defined by  
$$(x)_{0/1}=
\begin{cases} 
1,&x>0,\\
0,&x \leqslant 0.
\end{cases}$$ 
To end this section, we derive the proximal operators of $g(x;\omega,a,b,p,q)$ and $h(x;\omega,a,b,p,q)$.

\begin{figure}[!th]
\centering
\includegraphics[scale=0.64]{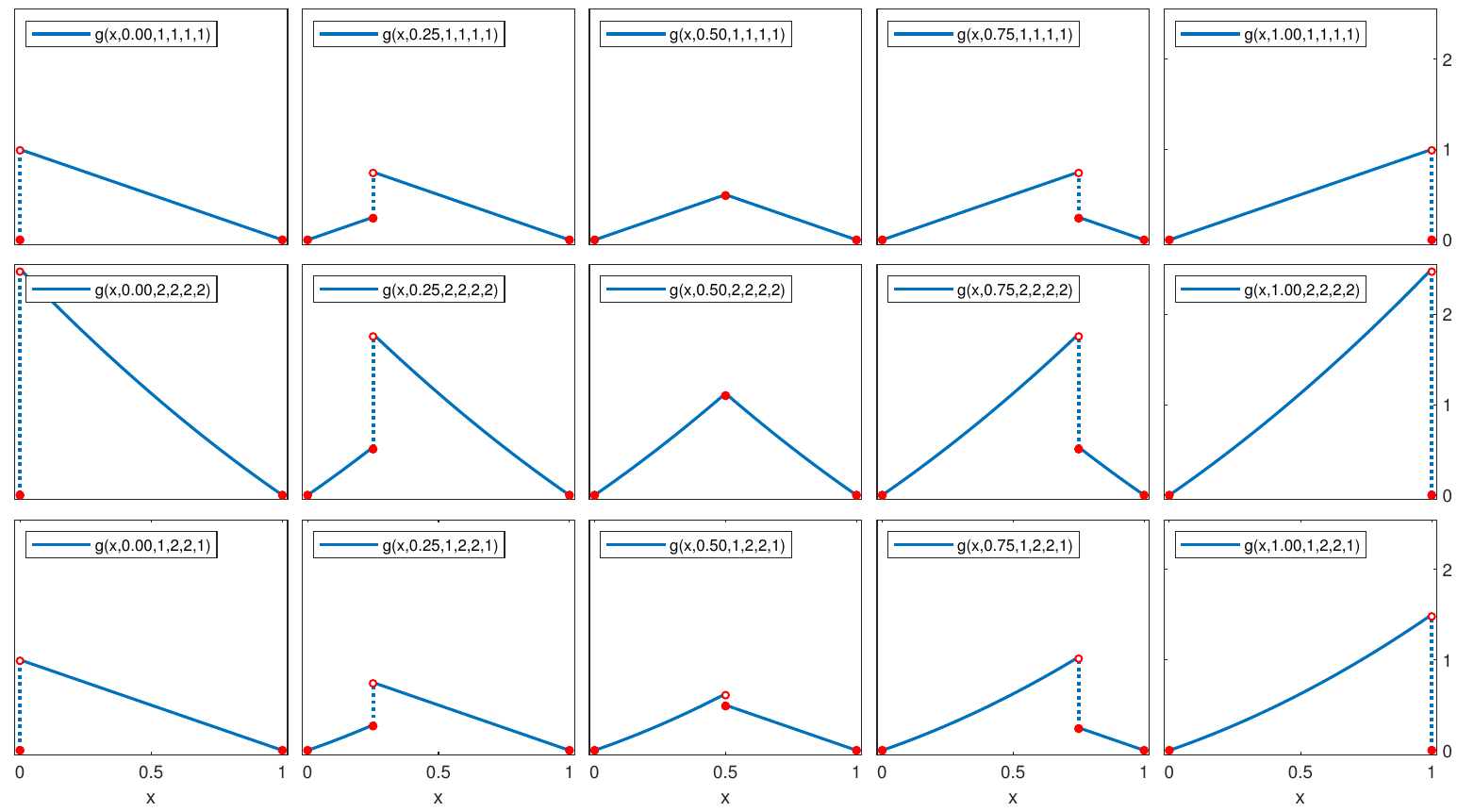}\\[3ex]
\includegraphics[scale=0.64]{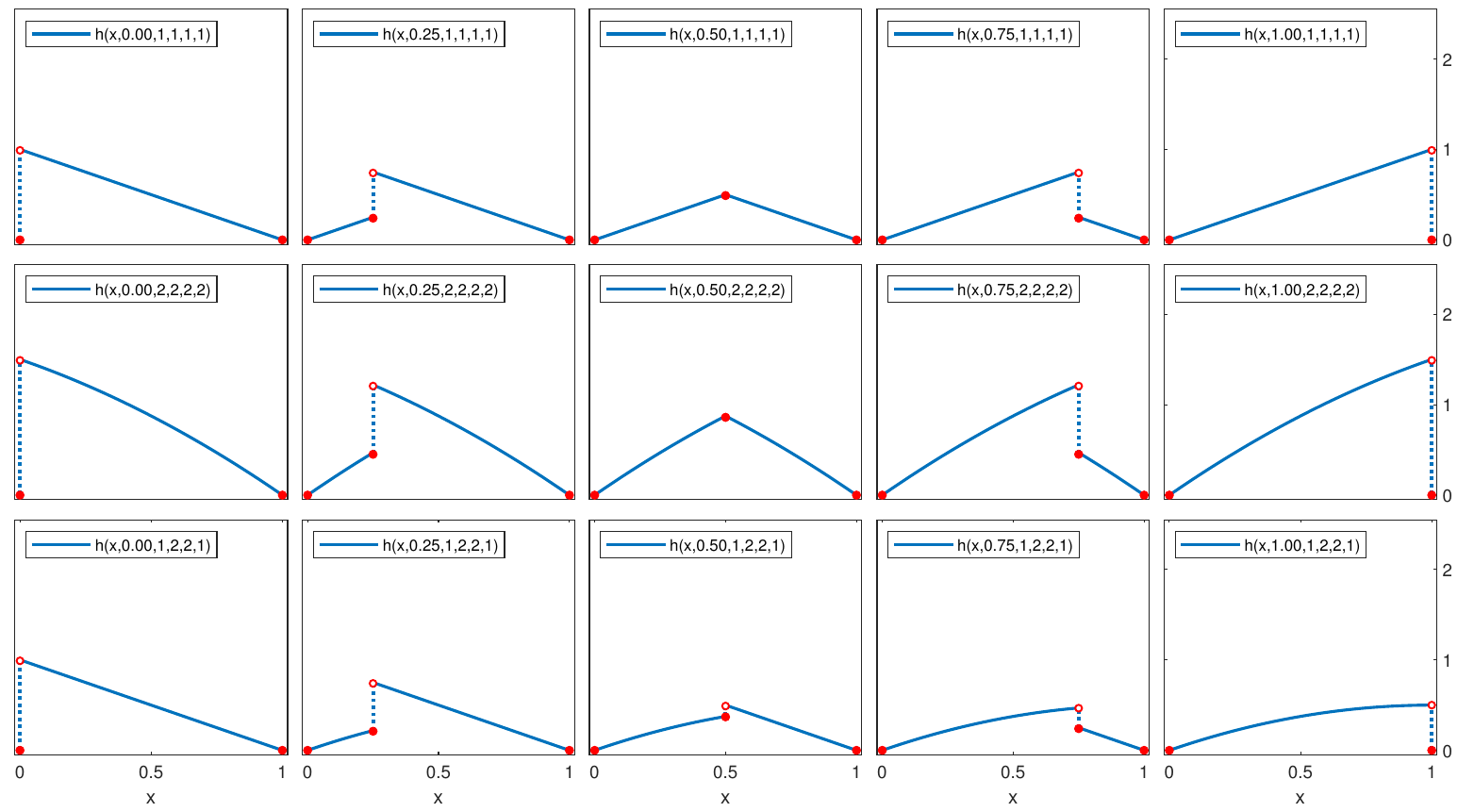}
\caption{Examples of $g(x;\omega,a,b,p,q)$ and $h(x;\omega,a,b,p,q)$. \label{ex-spf-g}}
\end{figure}

\begin{lemma}\label{lemma-prox-gh}Let $g(x;\omega,a,b,p,q)$ and $h(x;\omega,a,b,p,q)$ be defined by (\ref{def-g-h}), and
\begin{equation} 
\begin{aligned}
x_1^*&={\rm argmin}_{x\in[0, \omega]}~\phi_1(x):=\dfrac{(x-z)^2}{2\tau}+\dfrac{(x+a)^p}{p}-\dfrac{a^p}{p}, \\
x_2^*&={\rm argmin}_{x\in[\omega, 1]}~\phi_2(x):=\dfrac{(x-z)^2}{2\tau}+\dfrac{(1+b-x)^q}{q}-\dfrac{b^q}{q},\\
x_3^*&={\rm argmin}_{x\in[0, \omega]}~\phi_3(x):=\dfrac{(x-z)^2}{2\tau}-\dfrac{(a-x)^p}{p}+\dfrac{a^p}{p},\\
x_4^*&={\rm argmin}_{x\in[\omega, 1]}~\phi_4(x):=\dfrac{(x-z)^2}{2\tau}-\dfrac{(x-1+b)^q}{q}+\dfrac{b^q}{q}.
\end{aligned}
\end{equation} 
Then two proximal operators are computed by
\begin{equation*}
\begin{aligned}
{\rm Prox}_{\tau g(\cdot;\omega,a,b,p,q)}^B(z) = \Pi\left(x_1^*,x_2^*, \phi_1(x_1^*),\phi_2(x_2^*)\right),\\[1ex]
{\rm Prox}_{\tau h(\cdot;\omega,a,b,p,q)}^B(z) = \Pi\left(x_3^*,x_4^*, \phi_3(x_3^*),\phi_4(x_4^*)\right),
\end{aligned}
\end{equation*}
where $\Pi(\cdot,\cdot,\cdot,\cdot)$ is the set-valued operator defined by \eqref{operator-Pi}.
\end{lemma}

\begin{proof}  By decomposing $B=[0,\omega]\cup[\omega,1],$ one can just compare $\phi_1(x_1^*)$ and $\phi_2(x_2^*)$ (resp. $\phi_3(x_3^*)$ and $\phi_4(x_4^*)$) to derive ${\rm Prox}_{\tau g}^B$ (resp. ${\rm Prox}_{\tau h}^B$). 
\end{proof}

It can be verified that $x_1^*, x_2^*, x_3^*,$ and $x_4^*$ admit closed-form solutions when $p, q\in\{1/2,2/3,1,2\}$, leading to explicit expressions for the two proximal operators. Table \ref{tab:case-abpq} lists several closed-form solutions for case $p, q\in\{1,2\}$, where $\kappa_1:=\sqrt{(\tau+\tau^2)(2a+1)}-\tau-a\tau$ and $\kappa_2:=\sqrt{(\tau-\tau^2)(2a-1)}+\tau-a\tau$. Moreover, when $\tau$ exceeds a certain threshold, the proximal operators of $g$ and $h$ yield only binary elements, highlighting the effectiveness of SPFs in binary integer programming.

  \begin{table}[p] 
 \renewcommand{\arraystretch}{.98}\addtolength{\tabcolsep}{4pt}
\centering
\caption{Closed forms of proximal operators  ${\rm Prox}_{\tau g }^B(z)$ and  ${\rm Prox}_{\tau h}^B(z)$.} \label{tab:case-abpq}
\begin{tabular}{l llll}\hline
 $(\omega,a,b,p,q)$  & $\qquad \tau\geq {1/2}$&&  $\qquad\qquad 0<\tau \leqslant  {1/2} $\\ \hline
 $g(\cdot,0,1,1,1,1)$   &
 $\left\{\begin{array}{ll}
  0, &z< {1/2}\\
  \{0, 1\}, &z= {1/2}\\ 
   1,&z> {1/2} 
 \end{array}\right.$ &&
 $\left\{\begin{array}{ll}
  0, &z< \sqrt{2\tau}-\tau\\
  \{0, z+\tau\}, \hspace{12mm}&z= \sqrt{2\tau}-\tau\\ 
  z+\tau, &z\in( \sqrt{2\tau}-\tau, 1-\tau)\\ 
   1,&z \geqslant 1-\tau
 \end{array}\right.$ \\ \hline

  $g(\cdot,1,1,1,1,1)$   &
 $\left\{\begin{array}{ll}
  0, &z< {1/2}\\
  \{0, 1\}, &z= {1/2}\\ 
   1,&z> {1/2}
 \end{array}\right.$ &&
 $\left\{\begin{array}{ll}
  0, &z \leqslant  \tau\\
  z-\tau, &z\in( \tau, 1+\tau-\sqrt{2\tau})\\ 
    \{1, z-\tau\}, \hspace{12mm}&z= 1+\tau-\sqrt{2\tau}\\ 
    1,&z> 1+\tau-\sqrt{2\tau}
 \end{array}\right.$ \\ \hline
 
 $g(\cdot,1/2,1,1,1,1)$   &
 $\left\{\begin{array}{ll}
  0, &z<{1/2}\\
  \{0, 1\}, &z= {1/2}\\ 
   1,&z> {1/2}
 \end{array}\right.$ &&
 $\left\{\begin{array}{ll}
  0, &z \leqslant   \tau\\
  z-\tau, &z\in( \tau, {1/2})\\ 
   \{z-\tau, z+\tau\}, \hspace{5mm}&z= {1/2}\\ 
     z+\tau , &z= ({1/2}, 1-\tau)\\ 
    1,&z \geqslant 1-\tau
 \end{array}\right.$ \\\hline

  & $\qquad\tau\geq {{1}/{(2a)}}$&&  $\qquad\qquad 0<\tau \leqslant {{1}/{(2a)}}$\\ \hline
  $g(\cdot,0,a,a,2,2)$    &
 $\left\{\begin{array}{ll}
  0, &z<{1/2}\\
  \{0, 1\}, &z= {1/2}\\ 
   1,&z> {1/2}
 \end{array}\right.$ &&
 $\left\{\begin{array}{ll}
  0, &z<  \kappa_1\\
   \{0 ,  \frac{z+(1+a)\tau}{1+\tau}  \},\hspace{6mm}  &z= \kappa_1\\
     \frac{z+(1+a)\tau}{1+\tau}  , &z\in( \kappa_1, 1-a\tau)\\  
    1,&z \geqslant 1-a\tau
 \end{array}\right.$ \\ \hline
 
  $g(\cdot,1,a,a,2,2)$   &
 $\left\{\begin{array}{ll}
  0, &z<{1/2}\\
  \{0, 1\}, &z= {1/2}\\ 
   1,&z> {1/2}
 \end{array}\right.$ &&
 $\left\{\begin{array}{ll}
  0, &z \leqslant  a\tau\\
  \frac{z-a\tau}{1+\tau} , \hspace{20mm}&z\in( a\tau, 1-\kappa_1)\\ 
   \{ \frac{z-a\tau}{1+\tau} ,  1 \}, &z= 1-\kappa_1\\  
    1,&z> 1-\kappa_1
 \end{array}\right.$ \\ \hline

  $g(\cdot,1/2,a,a,2,2)$   &
 $\left\{\begin{array}{ll}
  0, &z<{1/2}\\
  \{0, 1\}, &z= {1/2}\\ 
   1,&z> {1/2}
 \end{array}\right.$ &&
 $\left\{\begin{array}{ll}
  0, &z \leqslant a\tau\\
  \frac{z-a\tau}{1+\tau} , &z\in( a\tau, {1/2})\\ 
   \{ \frac{z-a\tau}{1+\tau} ,  \frac{z+(1+a)\tau}{1+\tau}  \}, &z= {1/2}\\ 
     \frac{z+(1+a)\tau}{1+\tau}, &z= ({1/2}, 1-a\tau)\\ 
    1,&z \geqslant 1-a\tau
 \end{array}\right.$ \\ \hline
 

  $h(\cdot,0,a,a,2,2)$    &
 $\left\{\begin{array}{ll}
  0, &z<{1/2}\\
  \{0, 1\}, &z= {1/2}\\ 
   1,&z> {1/2}
 \end{array}\right.$ &&
 $\left\{\begin{array}{ll}
  0, &z<  \kappa_2\\
   \{0 ,  \frac{z+(a-1)\tau}{1-\tau}  \}, \hspace{6mm} &z= \kappa_2\\
     \frac{z+(a-1)\tau}{1-\tau}  , &z\in( \kappa_2, 1-a\tau)\\  
    1,&z \geqslant 1-a\tau
 \end{array}\right.$ \\ \hline
 
  $h(\cdot,1,a,a,2,2)$    &
 $\left\{\begin{array}{ll}
  0, &z<{1/2}\\
  \{0, 1\}, &z= {1/2}\\ 
   1,&z> {1/2}
 \end{array}\right.$ &&
 $\left\{\begin{array}{ll}
  0, &z \leqslant  a\tau\\
  \frac{z-a\tau}{1-\tau} ,\hspace{20mm} &z\in( a\tau, 1-\kappa_2)\\ 
   \{ \frac{z-a\tau}{1-\tau} ,  1 \}, &z= 1-\kappa_2\\  
    1,&z> 1-\kappa_2
 \end{array}\right.$ \\ \hline

  $h(\cdot,1/2,a,a,2,2)$    &
 $\left\{\begin{array}{ll}
  0, &z<{1/2}\\
  \{0, 1\}, &z= {1/2}\\ 
   1,&z> {1/2}
 \end{array}\right.$ &&
 $\left\{\begin{array}{ll}
  0, &z \leqslant  a\tau\\
  \frac{z-a\tau}{1-\tau} , &z\in( a\tau, {1/2})\\ 
   \{ \frac{z-a\tau}{1-\tau} ,  \frac{z+(a-1)\tau}{1-\tau}  \}, &z= {1/2}\\ 
     \frac{z+(a-1)\tau}{1-\tau}, &z= ({1/2}, 1-a\tau)\\ 
    1,&z \geqslant 1-a\tau
 \end{array}\right.$ \\\hline
 
\end{tabular}
\end{table}

\section{Exact Penalty Theory} \label{Section-model}
Let $g$ be a SPF defined by Definition \ref{def-SPF}, where for any $0 \leqslant x_i \leqslant 1$,  $g(x_i) =0$ if and only if $x_i\in\{0,1\}$. Then one can observe that
\begin{equation}\label{def-varphi}
\begin{aligned}
\x\in\{0,1\}^n \quad &\Longleftrightarrow \quad 0 \leqslant x_i \leqslant 1,~~ g(x_i) =0,~~i\in[n],\\
\quad &\Longleftrightarrow \quad \x\in\B, ~~\varphi(\x):=\sum_{i=1}^n g(x_i) =0.
\end{aligned}
\end{equation}
The first equivalence in \eqref{def-varphi} gives rise to the following equivalence,
$$\text{Problem}~ \eqref{UBIP}\quad \Longleftrightarrow \quad \text{Problem}~  \eqref{SCO}.$$
The second equivalence in \eqref{def-varphi} implies that the penalty model of   problem \eqref{SCO} takes the following form
\begin{equation}
\label{pi-penalty}
\min_{\x\in \B}~F(\x;{\mu}):=f(\x)+{\mu} {\varphi}(\x), \tag{EPM}
\end{equation} 
where $  {\mu}>0$ is chosen as  
\begin{equation}\label{lower-bd-pi}
 \mu >\overline{{\mu}} :=\max_{\x\in\B}\frac{\|\nabla f(\x)\|_\infty}{c}.
\end{equation}
Here, $c$ is relied on $g$ and defined in \eqref{lower-bd-subdiff}. Since $f$ is continuously differentiable, $\nabla f$ is continuous, which by the boundedness of $\B$ leads to the boundedness of $ \overline{{\mu}}$.   
 In the sequel, we show that problem \eqref{pi-penalty} is indeed an exact penalty model of problem \eqref{SCO}. 
\subsection{KKT points}
 \begin{definition} Point $\widetilde{\x}\in \B$ is called a KKT point of problem (\ref{pi-penalty}) if it satisfies
\begin{eqnarray}\label{KKT}
0 \in \nabla f(\widetilde{\x}) +{\mu} \partial \varphi(\widetilde{\x}) + N_{\B}(\widetilde{\x}).
\end{eqnarray}
\end{definition}
Recalling \eqref{normal-cone-B},  $\widetilde{\x}\in\B$ is a KKT point of \eqref{pi-penalty} if and only if there is $\boldsymbol{\nu}\in\partial \varphi(\widetilde{\x})$ satisfying
\begin{eqnarray}\label{KKT-equivalent-fomula}
 \nabla_i f(\widetilde{\x})+{\mu}\nu_i\begin{cases}
 \geqslant 0,& {\rm if}~\widetilde{x}_i=0,\\
= 0 ,& {\rm if}~\widetilde{x}_i\in(0,1),\\
\leqslant 0,& {\rm if}~\widetilde{x}_i=1.
\end{cases}\end{eqnarray}
Our first result establishes the first-order optimality condition for problem \eqref{pi-penalty}.
\begin{theorem}\label{firstKKT} The following statements hold for problem (\ref{pi-penalty}).  
\begin{itemize}[leftmargin=16pt]
\item[1)] Any KKT point is binary. 
\item[2)] A KKT point is a local minimizer if $f$ and $\varphi$  are locally convex around this point.
\item[3)] A local minimizer is also a KKT point. 
\end{itemize}
\end{theorem}
\begin{proof} 1)  Suppose a KKT point $\widetilde{\x}\notin \{0,1\}^n$, namely, there is an $\widetilde{x}_i \in (0,1)$. Then from \eqref{KKT-equivalent-fomula}, $-\nabla_i f (\widetilde{\x})\in \mu \partial g(\tilde{x}_i) $,  which by \eqref{lower-bd-subdiff} leads to  $|\nabla_i f (\widetilde{\x})|\geqslant c\mu$ and further yields a contradiction,
$$c \overline{{\mu}}  =   \max_{\x\in\B}\|\nabla f(\x)\|_\infty  \geqslant   |\nabla_i f(\widetilde{\x})|   \geqslant  c{\mu} >  c\overline{{\mu}}.$$
 
2) Let  $\widetilde{\x}$ be a KKT point. Then $\widetilde{\x}\in \{0,1\}^n$ from 1).  For any $\x\in \B$ around $\widetilde{\x}$,  the local convexity of $f$ and $\varphi$  enables that
\begin{equation*}
\begin{aligned}
 F(\widetilde{\x};{\mu}) -  F(\x;{\mu}) 
 \geqslant    \left\langle \nabla  f(\widetilde{\x})+{\mu} \bze , \widetilde{\x}-\x \right\rangle \geqslant   0,
\end{aligned}
\end{equation*}
where $\bze\in \varphi(\widetilde{\x})$ and the last inequality is from \eqref{KKT}.

3) By Definition \ref{def-SPF} c3), for any $x\in\{0,1\}$, we have $${\partial (g(x)+\delta_{B}(x))=\partial g(x)+N_{B}(x)}.$$ 
For any $x\in(0,1)$, there is $\delta_{B}(x)=0$ and $N_{B}(x)=\{0\}$. This means the above condition holds for any $x\in B$, which by letting $G(\widetilde{\x}):={\mu} \varphi(\widetilde{\x}) + \delta_\B(\widetilde{\x})$ yields
\begin{equation}
	\label{subdiff-rule}
	\partial G(\widetilde{\x}) = \partial \Big({\mu} \varphi(\widetilde{\x}) + \delta_\B(\widetilde{\x})\Big)
	= {\mu} \partial  \varphi(\widetilde{\x}) + N_\B(\widetilde{\x}),
\end{equation}
By Fermat's rule \cite[Theorem 10.1]{rock98}, if $\widetilde{\x}$ is a local minimizer of \eqref{pi-penalty}, then
\begin{equation}\label{opt-local-0}
	0 \in \partial \Big(f(\widetilde{\x})+{\mu} \varphi(\widetilde{\x}) + \delta_\B(\widetilde{\x})\Big)= \partial \Big(f(\widetilde{\x})+ G(\widetilde{\x})\Big).
\end{equation}
As $f$ is continuously differentiable,  $\widehat{\partial }f(\widetilde{\x})=\partial f(\widetilde{\x})=\{\nabla f(\widetilde{\x})\}$, where $\widehat{\partial } f$ stands for the regular subdifferential \cite[Definition 8.3]{rock98} of $f$. According to the  definition of $\widehat{\partial } f$, it is easy to verify that
\begin{align*}
&\widehat{\partial } \Big(f(\widetilde{\x})+G(\widetilde{\x})  \Big)
	 \supseteq  \nabla f(\widetilde{\x})+ \widehat{\partial }  G(\widetilde{\x}),\\
&\widehat{\partial } G(\widetilde{\x})  = \widehat{\partial } \Big(f(\widetilde{\x})+ G(\widetilde{\x}) - f(\widetilde{\x}) \Big)  \supseteq  \widehat{\partial }  \Big(f(\widetilde{\x})+G(\widetilde{\x})\Big) - \nabla f(\widetilde{\x}).
\end{align*}
Therefore,
\begin{align*}
\widehat{\partial } \Big(f(\widetilde{\x})+G(\widetilde{\x})  \Big)
	=  \nabla f(\widetilde{\x})+ \widehat{\partial }  G(\widetilde{\x}).
\end{align*}
Using the above fact and \cite[Definition 8.3]{rock98} of limiting subdifferential $\partial f$, one can verify that
\begin{align*}
\partial   \Big(f(\widetilde{\x})+G(\widetilde{\x})  \Big)
	 \supseteq  \nabla f(\widetilde{\x})+  {\partial }  G(\widetilde{\x}).
\end{align*}
Similar reasoning to show the above condition can derive
\begin{align*}
{\partial } G(\widetilde{\x})  =  {\partial } \Big(f(\widetilde{\x})+G(\widetilde{\x}) - f(\widetilde{\x}) \Big)  \supseteq   {\partial }  \Big(f(\widetilde{\x})+G(\widetilde{\x})\Big) - \nabla f(\widetilde{\x}).
\end{align*}
Consequently,
\begin{align*}
 {\partial } (f(\widetilde{\x})+G(\widetilde{\x})  )
	=  \nabla f(\widetilde{\x})+  {\partial } G(\widetilde{\x}),
\end{align*}
which together with \eqref{subdiff-rule}  and \eqref{opt-local-0} suffices to
\begin{equation}\label{opt-local}
	0 \in \nabla f(\widetilde{\x})+\mu\,\partial\varphi(\widetilde{\x})+N_{\mathbb B}(\widetilde{\x}),
\end{equation}
showing the desired result.
\end{proof} 
In the above theorem, the assumption of the local convexity of $\varphi$ around a KKT point is mild. In fact, many SPFs $g(z)$ are locally convex at $z=0$ and $z=1$. Examples include $g(x) = 1 - |2x - 1|$, as well as  function $g(x; \omega, a, b, p, q)$ with any $\omega \in (0,1)$ and $p = 1$ or $2$ defined in \eqref{def-g-h}.  This means $\varphi(\x)=\sum_{i=1}^n g(x_i) $ is locally convex at any $\x\in\B$. Together with the binary nature of KKT points, this local convexity condition suffices to meet the assumption in the theorem.  Now, we provide an example to illustrate Theorem \ref{firstKKT}, as well as the advantage of using SPFs over some existing functions. In particular, we compare SPFs with the function $g(x)=x(1-x)$, which is widely used in the literature for characterizing binary variables and reformulating binary optimization problems; see, for example, \cite{Zhang07,lucidi2010exact,murray2010algorithm}.
\begin{example}
Given $s\in(1,2)$, consider the following two problems,
\begin{eqnarray}
\label{ex-pi-penalty-1}&&\min_{x\in B}~ F(x;{\mu})=({1}/{s})\left|x-  {1}/{2}\right|^s+{\mu} x(1-x),\\
\label{ex-pi-penalty-2}&&\min_{x\in B}~ F(x;{\mu})=({1}/{s})\left|x-  {1}/{2}\right|^s+{\mu} g(x;1/2,1,1,2,2).
\end{eqnarray} 
\begin{itemize}[leftmargin=15pt]
\item Problem (\ref{ex-pi-penalty-1}) always has a local minimizer  $1/2$ for any given ${\mu}>0$.  In fact, $\partial F({1}/{2};{\mu})=\{0\}$ for any given ${\mu}$. Moreover, for any sufficiently small $|\varepsilon|$, we have
\begin{eqnarray*}
\begin{aligned} 
F\left(1/2+\varepsilon;{\mu}\right)-F\left(1/2;{\mu}\right)=  \left|\varepsilon\right|^s/s+{\mu} \left(1/4+\varepsilon^2\right) -  {{\mu}}/{4} =  \left( |\varepsilon|^{p-1}/s -{\mu}\right) \varepsilon^2 \geqslant 0 . 
\end{aligned}\end{eqnarray*}
Therefore, $x=1/2$ is a local minimizer of problem  (\ref{ex-pi-penalty-1}).
\item Problem (\ref{ex-pi-penalty-2}) always has binary local minimizers when ${\mu}>\overline{\mu}=2^{1-s}$. In fact, 
\begin{equation*}
0\notin  \partial F(x;{\mu}) =  \left|x- 1/2\right|^{s-1} +{\mu}\begin{cases}
\{x+1\},& x \in(0,1/2),\\
\{-3/2,3/2\},& x =1/2,\\
\{x-2\},& x \in(1/2,1).
\end{cases}
\end{equation*} 
Therefore, the local minimizer of problem  (\ref{ex-pi-penalty-2}) can not lie within $(0,1)$. Two local (also global) minimizers are $x=0$ and $x=1$.
\end{itemize}
\end{example}
 This example illustrates that problem \eqref{pi-penalty} with a penalty like $g(x)=x(1-x)$ may have non-binary local minimizers. In contrast, all local minimizers of the problem with SPFs are binary when ${\mu}$ is over a certain threshold. 
Based on Theorem \ref{firstKKT}, we establish the exact penalty theorem for problems \eqref{SCO} and \eqref{pi-penalty}. 
\begin{theorem}[Exact Penalty Theorem]\label{theorem-ept} A point is a global minimizer of problem (\ref{SCO})  if and only if it is a global minimizer of problem (\ref{pi-penalty}).
\end{theorem}
\begin{proof} Let $\x^*$ and $\widetilde{\x}$ be global minimizers of problems \eqref{SCO} and \eqref{pi-penalty}, respectively. Then, $\widetilde{\x}$ is a KKT point of problem \eqref{pi-penalty} and thus $\widetilde{\x}\in \mathcal{F}:=\{0,1\}^n$ by Theorem \ref{firstKKT}.   Since the feasible region of \eqref{SCO} is $\mathcal{F}$, it follows $f(\x^*)\leqslant f(\widetilde{\x})$, leading to   
\begin{equation*}
 F(\widetilde{\x};{\mu})  \leqslant  F({\x}^*;{\mu})  
=
f(\x^*)\leqslant f(\widetilde{\x}) 
 =   F(\widetilde{\x};{\mu}) . 
\end{equation*} 
where the first inequality holds because of the  optimality of $\widetilde{\x}$ to problem \eqref{pi-penalty} and ${\x}^*\in\mathcal{F} \subseteq\B$, the two equations hold due to $\varphi({\x}^*)=\varphi(\widetilde{\x})=0$. The above condition implies the four values are the same, concluding the conclusion. \end{proof}  
\begin{remark}
Theorem \ref{theorem-ept} means that the global minimizers to problem  (\ref{SCO}) (equivalent to problem (\ref{UBIP})) and its penalty model (\ref{pi-penalty}) coincide when $\mu>\overline{\mu}$, while threshold $\overline{\mu}$ defined in (\ref{lower-bd-pi}) is  independent of the solution set of (\ref{SCO}). This contrasts with traditional exact penalty theory (e.g., \cite[Theorem 17.3]{Nocedal99}), where the penalty parameter relies on the solution to the original problem. 
It is worth mentioning that a similar exact penalty theorem was established in \cite{Yuan17}, but under the assumption that $f$ is Lipschitz continuous and convex on $\B$, and the penalty parameter is relied on the Lipschitz constant. In contrast, out Theorem \ref{theorem-ept} is derived without imposing any additional conditions on $f$.  Moreover, although function $g(x)=1-|2x-1|^p$ with ${p\in(0,1]}$, which turns out to be a SPF, was proposed in \cite{Wu19}, its properties have not been thoroughly explored, and the corresponding exact penalty theory has not yet been established.
 \end{remark}
 \subsection{P-stationary point}
In this subsection,  to facilitate the algorithm development for solving problem \eqref{pi-penalty}, we define a P-stationary point associated with the proximal operator of $\varphi$.
\begin{definition}
Point $\overline{\x}$ is called a P-stationary point of (\ref{pi-penalty}) if there is a $\tau>0$ such that
\begin{equation}\label{Pstationary}
\begin{aligned}
\overline{\x} &\in {\rm Prox}_{\tau{\mu} \varphi}^\B\Big(\overline{\x} - \tau \nabla f(\overline{\x})\Big)\\
& =  {\rm argmin}_{\z\in\B}~\frac{1}{2}\| \z -(\overline{\x}- \tau \nabla f(\overline{\x}))\|^2+\tau{\mu} \varphi(\z).
\end{aligned}\end{equation} 
\end{definition}
It is worth noting that many SPFs $g$ enable a closed-form of ${\rm Prox}_{\tau{\mu} \varphi}^\B$, see Lemma \ref{lemma-prox-gh}, thereby enabling the design of an efficient numerical algorithm presented in the next section. By applying similar reasoning to derive \eqref{opt-local}, we obtain the necessary optimality condition of \eqref{Pstationary} as follows, 
\begin{equation} \label{opt-P-sta}
\begin{aligned}
0&\in \partial\left(\frac{1}{2}\|\overline{\x} -(\overline{\x}- \tau \nabla f(\overline{\x}))\|^2 + \tau{\mu}  \varphi(\overline{\x}) + \delta_\B(\overline{\x})\right)\\
&=    \tau \nabla f(\overline{\x})  +  \tau{\mu}  \partial  \varphi(\overline{\x}) + N_\B(\overline{\x}),
\end{aligned}\end{equation} 
which implies \eqref{KKT}. So  a P-stationary point is a KKT point. The following result establishes the relationships among P-stationary points, KKT points, and local minimizers of  (\ref{pi-penalty}).

 \begin{theorem}\label{firstPST} The following relationships hold for problem (\ref{pi-penalty}).  
\begin{itemize}[leftmargin=16pt]
\item[1)] A  P-stationary point is a KKT point. Conversely, a KKT point is a P-stationary point if $\varphi$ is locally convex around this point. 
\item[2)] A local minimizer is a P-stationary point if either $f$ and  $\varphi$ are locally convex around this point or $f$ is $L$-strongly smooth on $\B$. 
\item[3)] A  P-stationary point with $\tau \geqslant 1/\ell$ is a global minimizer if $f$ is $\ell$-strongly convex on $\B$. 
\end{itemize} 
\end{theorem}
\begin{proof} 1) It follows from  \eqref{opt-P-sta} that a  P-stationary point is a KKT point.   Conversely, let $\widetilde{\x}$ be a KKT point of \eqref{pi-penalty}. Then $ {\widetilde{\x}}\in\{0,1\}^n$ by Theorem \ref{firstKKT}. Let $\z$ satisfy
\begin{equation} \label{z-prox-x}
\begin{aligned}
 \z \in {\rm Prox}_{\tau{\mu} \varphi}^\B (\widetilde{\x} - \tau \nabla f(\widetilde{\x}) ).
\end{aligned}\end{equation} 
By the definition of ${\rm Prox}_{\tau{\mu} \varphi}^\B$, we obtain
\begin{equation*} 
\begin{aligned}
 \| \z -(\widetilde{\x}- \tau \nabla f(\widetilde{\x}))\|^2+2\tau{\mu} \varphi(\z) \geqslant    \| \widetilde{\x} -(\widetilde{\x}- \tau \nabla f(\widetilde{\x}))\|^2+2\tau{\mu} \varphi(\widetilde{\x}), 
\end{aligned}\end{equation*} 
which by the local convexity of $\varphi$, \eqref{KKT}, and $\z\in\B$ results in 
\begin{equation} \label{cond-P-sta}
\begin{aligned}
- \| \z - \widetilde{\x} \|^2&\geqslant   2\tau \langle \z - \widetilde{\x}, \nabla f({\x}) \rangle  + 2\tau {\mu} (\varphi(\z)-\varphi(\widetilde{\x}))\\
&\geqslant   2\tau \langle \z - \widetilde{\x}, \nabla f(\widetilde{\x}) +{\mu}\bze\rangle \geqslant   0,
\end{aligned}\end{equation} 
where $\bze\in\partial\varphi(\widetilde{\x})$. The above condition delivers  $\z = \widetilde{\x} $. Therefore, $\widetilde{\x}$ is a P-stationary point.

2) If $f$ and  $\varphi$ are locally convex around this point, then from Theorem \ref{firstKKT}, a local minimizer is a KKT, which by 1) indicates that it is a P-stationary point. Now we consider the case of strong smoothness of $f$ on $\B$. Let $\x$ be a local minimizer of \eqref{pi-penalty} and $\z$ satisfy $ \z \in {\rm Prox}_{\tau{\mu} \varphi}^\B ( {\x} - \tau \nabla f({\x}) ).$ Then it is a KKT point and $ {\x}\in\{0,1\}^n$ from Theorem \ref{firstKKT}, thereby $\varphi(\x)=0$.  Similar reasoning allows us to derive  the first inequality in \eqref{cond-P-sta}, namely,
\begin{equation*} 
\begin{aligned}
 \| \z - {\x} \|^2 &\leqslant   2\tau\langle {\x}-\z , \nabla f({\x}) \rangle  +  2\tau{\mu} (\varphi(\x)-\varphi(\z))\\
 &\leqslant   2\tau\|\z- {\x}\|\| \nabla f({\x})\|,
\end{aligned}\end{equation*} 
where the second inequality is due to $\varphi(\x)=0$ and $\varphi(\z)\geqslant  0$, which implies $\| \z - {\x} \|\leqslant   2\tau \| \nabla f({\x})\|$ and thus $\z$ is around ${\x}$ when $\tau$ is sufficiently small. 
By the $L$-strong smoothness of $f$ on $\B$ and the first inequality in \eqref{cond-P-sta},
 \begin{equation*} 
\begin{aligned}
 2F(\z;{\mu})- 2F(\x;{\mu}) 
\leqslant   &~2\langle \z - {\x}, \nabla f({\x}) \rangle  + 2{\mu} (\varphi(\z)-\varphi(\x))+L\| \z - {\x} \|^2\\
\leqslant  &~\left({L}-{1}/{\tau }\right)\| \z - {\x} \|^2 \leqslant   0,
\end{aligned}\end{equation*}
where the last inequality is because $\tau$ is sufficiently small.
The above condition leads to  $\z = {\x} $ due to the local optimality of $\x$, namely, $\x$ is a P-stationary point. 

3) Let $\overline{\x}$ be a P-stationary point.  Then it satisfies
\begin{equation*} 
\begin{aligned}
 \overline{\x} \in {\rm Prox}_{\tau{\mu} \varphi}^\B (\overline{\x} - \tau \nabla f(\overline{\x}) ).
\end{aligned}\end{equation*} 
By the definition of ${\rm Prox}_{\tau{\mu} \varphi}^\B$, we obtain
\begin{equation*} 
\begin{aligned}
 \| \overline{\x} -(\overline{\x}- \tau \nabla f(\overline{\x}))\|^2+2\tau{\mu} \varphi(\overline{\x}) \leqslant  \| \w -(\overline{\x}- \tau \nabla f(\overline{\x}))\|^2+ 2\tau{\mu} \varphi(\w), 
\end{aligned}\end{equation*} 
for any $\w\in\B$,  which results in 
\begin{equation}  
\begin{aligned}
2\langle \overline{\x} - {\w}, \nabla f(\overline{\x}) \rangle  + 2{\mu} (\varphi(\overline{\x})-\varphi(\w))  &\leqslant  ( {1}/{\tau })\| \w - \overline{\x} \|^2. 
\end{aligned}\end{equation} 
Using the above condition and the $\ell$-strong convexity of $f$, we have
 \begin{equation*} 
\begin{aligned}
 2F(\overline{\x};{\mu}) - 2F(\w;{\mu}) 
\leqslant  &~2\langle \overline{\x} - {\w}, \nabla f(\overline{\x}) \rangle  + 2{\mu} (\varphi(\x)-\varphi(\w))- \ell\| \w - \overline{\x} \|^2\\
\leqslant &~\left({1}/{\tau }-{\ell}\right)\| \w - \overline{\x} \|^2 \leqslant 0,
\end{aligned}\end{equation*}
due to $\tau \geqslant 1/\ell$. Therefore, $\overline{\x}$ is global minimizer of \eqref{pi-penalty}. 
\end{proof} 
Based on Theorems \ref{firstKKT}$-$\ref{firstPST}, we can establish the relationships among different solutions to problems \eqref{SCO} and \eqref{pi-penalty}. These relationships suggest that pursuing a P-stationary point for problem \eqref{pi-penalty} is a promising way to find a solution to \eqref{SCO} or \eqref{UBIP}.
\begin{corollary}\label{cor-rel} The relationships among solutions to problems (\ref{UBIP}), (\ref{SCO}), and (\ref{pi-penalty}) are given in Figure \ref{fig:relation}. For example, global minimizers of three problems coincide;  a global minimizer of (\ref{pi-penalty})  is a P-stationary point if either $f$ and $\varphi$ are locally convex or $f$ is strongly smooth on $\B$; a P-stationary point of problem (\ref{pi-penalty}) is also a global minimizer if $f$ is strongly convex on $\B$.
\end{corollary}

 \begin{figure}[H]
 \centering
 \includegraphics[scale=.65]{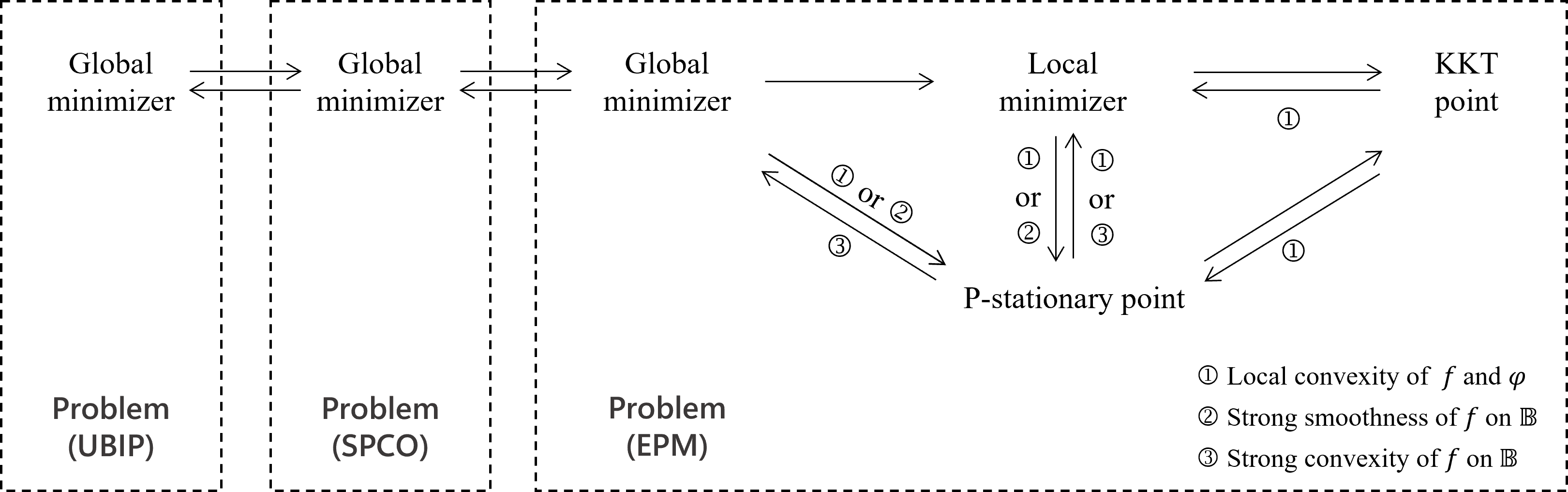}
 \caption{Relationships among different points for thee problems.\label{fig:relation}}
 \end{figure}

\begin{example}
Consider the following quadratic binary problem,
\begin{eqnarray}\label{ex-qbip}
\min_{\x\in\{0,1\}^2}~f(\x):=\frac{1}{2}\left\langle \x, {\bf Q} \x \right\rangle+\left \langle {\bf q}, \x \right\rangle,\qquad
\text{with}~~{\bf Q} = 
\left[\begin{array}{rr}
1 &-1\\-2& 0
\end{array}
\right],~~{\bf q} = 
\left[\begin{array}{r}
-2\\ 1/2
\end{array}
\right].
\end{eqnarray} 
We aim to solve the following penalty model,
\begin{eqnarray}\label{ex-qbip-epm}
\min_{\x\in\B }~ F(\x;{\mu}):=\frac{1}{2} \left\langle \x, {\bf Q} \x \right\rangle+\left \langle {\bf q}, \x \right\rangle+{\mu} \varphi(\x), 
\end{eqnarray} 
where $\varphi(\x)= g(x_1;1/2,1,1,1,1)+g(x_2;1/2,1,1,1,1)$ and
\begin{eqnarray*}
 g(x;1/2,1,1,1,1)= 
\begin{cases}
x,&  {\rm if}~x \leqslant 1/2,\\
1-x,&  {\rm if}~x \geqslant 1/2. 
\end{cases} 
\end{eqnarray*} \end{example}
The objective is strongly smooth but is nonconvex since  ${\bf Q}$ has two eigenvalues $-1$ and $2$. One can verify that ${c=1}$ and ${\overline{\mu}=3}$. Therefore, according to Theorem \ref{theorem-ept}, when ${\mu>3}$, global minimizers to problems (\ref{ex-qbip})  and (\ref{ex-qbip-epm}) coincide. The global minimizer is ${\x^*=(1,1)^\top}$. Now, we verify the relationships between the global minimizer and a P-stationary point.  By fixing ${\mu=4}$ and ${\tau=1/2}$,   
one can check that
\begin{equation} \label{z-x-half}
{\rm Prox}_{2\varphi}^\B\left(\x - \nabla f(\x)/2\right) =  
\begin{cases}
\{(1,0)^\top\},&  {\rm if}~\x=(0,0)^\top,\\
\{(1,1)^\top\},&  {\rm if}~\x=(1,0)^\top, (0,1)^\top,~{\rm and}~ (1,1)^\top.
\end{cases}
\end{equation} 
Therefore, $\x^*$ is the only point  satisfying
\begin{equation*} 
\x^*\in{\rm Prox}_{2\varphi}^\B\left(\x^* - \nabla f(\x^*)/2\right),  
\end{equation*} 
namely, the only  P-stationary point with $\tau=1/2$ is $\x^*$. This justifies the relationship in Corollary \ref{cor-rel}  that a global minimizer of problem (\ref{pi-penalty})  is a P-stationary point if $f$ is strongly smooth on $\B$.

 \section{Algorithm and convergence} \label{Section-algorithm}
Rewrite problem \eqref{pi-penalty} as
\begin{equation}
\label{pi-penalty-w}
\min_{\x,\w}~f(\x)+{\mu} {\varphi}(\w), ~~~ \text{s.t.}~ \x-\w={\bf 0},~\w\in \B. 
\end{equation} 
The corresponding augmented Lagrange function is
\begin{equation}
\label{pi-penalty-y}
L(\w,\x,\y;{\mu})=f(\x)+{\mu} {\varphi}(\w) +\langle \y, \x-\w\rangle+\frac{\sigma}{2}\|\x-\w\|^2, 
\end{equation} 
where $\y\in\R^n$ is the Lagrange multiplier and ${\sigma>0}$. We aim to solve 
problem \eqref{pi-penalty-w} using a variant of the inexact alternating direction method of multipliers (iADMM). Specifically, given current point $(\w^k,\x^k,\y^k;{\mu}_k)$, the next point is updated by
\begin{align}
\label{pi-penalty-update-w}
\w^{k+1}&\in{\rm argmin}_{\w\in \B} ~ {\mu}_k {\varphi}(\w) +\langle \y^k, \x^k-\w\rangle+\frac{\sigma}{2}\|\x^k-\w\|^2, \\
\label{pi-penalty-update-x}
\x^{k+1}&={\rm argmin}_{\x}   ~\langle \nabla f(\w^{k+1})+\y^k, \x -\w^{k+1}\rangle + \frac{1}{2}\| \x -\w^{k+1} \|^2_{\Q^{k+1}}  +\frac{\sigma}{2}\|\x -\w^{k+1}\|^2, \\
\label{pi-penalty-update-y}
\y^{k+1} &= \y^{k} + \sigma (\x^{k+1}-\w^{k+1}),
\end{align}
where $\Q^{k+1}\succeq{\bf 0}$ is known as a pre-conditioning matrix updated adaptively during iterations or predefined in advance.  Subproblem \eqref{pi-penalty-update-w} is equivalent to
\begin{align}
\label{pi-penalty-update-w-solution}
\w^{k+1}
&\in{\rm Prox}_{({\mu}_k/{\sigma})\varphi}^\B(\x^k+\y^k/\sigma).
\end{align}
Given particular $\varphi(\x)=\sum_{i=1}^n g(x_i)$, such as $g$  defined by \eqref{def-g-h} and \eqref{def-h-h}, the above problem admits a closed-form solution due to Lemma \ref{lemma-prox-gh}. While for subproblem \eqref{pi-penalty-update-x}, we approximate  $f(\x)$ at $\w^{k+1}$ instead of $\x^k$ using a Taylor-like expansion. This design encourages $\x^{k+1}$ to  approach $\w^{k+1}$ quickly, potentially leading to faster convergence.   For subproblem \eqref{pi-penalty-update-x}, its objective function is strongly convex since $\Q^{k+1}\succeq 0$ and $\sigma>0$. Hence, it  has a unique minimizer, which satisfies the following  first-order optimality condition, 
$$
	\nabla f(\w^{k+1})+\y^k+\Q^{k+1}(\x-\w^{k+1})+\sigma(\x-\w^{k+1})=0.
$$
The above condition immediately gives rise to the unique solution to \eqref{pi-penalty-update-x} as
\begin{align}
	\label{pi-penalty-update-x-solution}
	\x^{k+1} =  \w^{k+1} -  (\sigma\I+\Q^{k+1}) ^{-1}  ( \nabla f(\w^{k+1})+\y^k ), 
\end{align}
where $\I$ is an identity matrix. To ensure binary  solutions eventually, we update $\mu_k$ as follows, 
\begin{equation}\label{update-mu}
\mu_{k}=\mu_{k-1}+\begin{cases}
\min\left\{(\eta-1)\mu_{k-1},~\dfrac{\rho\sigma\|\x^{k}-\w^{k}\|^2}{\varphi(\w^{k})+\epsilon}\right\},& \text{if~~mod}(k,k_0)=0,~\varphi(\w^{k})\neq 0,\\
0, &\text{otherwise},
\end{cases}
\end{equation}
where ${\eta>1}, {\rho\in(0,1/6]}$, ${\epsilon>0}$, ${k_0>0}$ is an integer, and mod$(a,b)$ returns  the remainder after division of $a$ by $b$.  We note that $\w^{k+1}$ is binary if ${\varphi(\w^{k+1})=0}$. In such a scenario, it is unnecessary to increase $\mu_k$ further for better solutions. The algorithm stops when $(\w^k, \x^k, \y^k)$ satisfies 
\begin{eqnarray}\label{stop}\w^{k}\in\{0,1\}^n,~~ {\rm tol}_k:=\max\left\{\dfrac{\|\x^k-\w^k\|}{1+\|\w^k\|},~ \dfrac{\|\y^k + \nabla f(\w^k)\|}{1+\|\w^k\|}\right\}  <\varepsilon,\end{eqnarray} 
where ${\varepsilon\in(0,1)}$ is a given tolerance. From \eqref{pi-penalty-update-w-solution}, if $(\w^k, \x^k, \y^k)$ satisfies the above conditions, then $\w^k$ is approximately a binary P-stationary point. The proposed algorithm is presented in Algorithm \ref{BN-algo}. We term the algorithm, ShaPeak, which is derived from the sharp-peak function-based ADMM.  

\begin{algorithm}[!t]
    \SetAlgoLined

 \textbf{Input} $\w^0=\x^0\in\{0,1\}^n, \y^0=-\nabla f(\w^0)$, and  $({\mu}_0, \sigma, \varepsilon)>0$.
 
	\For{$k=0,1,2,\ldots$}{
	
  Update $(\w^{k+1},\x^{k+1},\y^{k+1})$ by \eqref{pi-penalty-update-w-solution}, \eqref{pi-penalty-update-x-solution}, and \eqref{pi-penalty-update-y}.
	 
	 If \eqref{stop} holds, then stop.
	   
	 Update $\mu_{k+1}$ by \eqref{update-mu}.
}
 
	\caption{(ShaPeak) Sharp-peak function-based iADMM.}\label{BN-algo}
\end{algorithm}
 
\begin{remark}
Update rule (\ref{update-mu}) is motivated by both practical and theoretical considerations. 
		On the one hand, the exact penalty theory requires the penalty parameter to be sufficiently large. 
		Although this threshold does not depend on the optimal solution set, it is still generally difficult 
		to estimate tightly in practice. To overcome this difficulty, a natural strategy is to increase $\mu_k$ adaptively, leading to the proposed update rule, (\ref{update-mu}), where $(\eta-1)\mu_{k-1}$ controls the growth of $\mu_k$, whereas  
		$\rho\sigma\|\x^{k}-\w^{k}\|^2/(\varphi(\w^{k})+\epsilon)$ is designed based on quantities arising in the convergence analysis. In this way, the rule keeps 
		$\{\mu_k\}$ nondecreasing while avoiding overly aggressive updates.
\end{remark}

\subsection{Convergence analysis}
Hereafter, let $\{(\w^{k},\x^{k},\y^{k};{\mu}_{k})\}$ be the sequence generated by Algorithm \ref{BN-algo} with $\Q^k$ chosen as
\begin{align}\label{choice-Q}
\Q^k\succeq{\bf 0},~~\|\Q^k\|\leqslant \lambda ,
\end{align} 
where $\lambda >0$ is a predefined constant. For example, we can choose $\Q^k\equiv \lambda \I$ for problem \eqref{ex-qbip}  and $\lambda $ can be set as $\lambda =\|\Q\|$. Moreover, we define a useful region by
\begin{align}\label{def-omega}
\Omega:=\left\{(\w,\x)\in\B\times\R^n: f(\w)+\frac{\sigma}{2}\|\w-\x\|^2 \leqslant f(\w^0)\right\}.
\end{align}
Since $f$ is continuous, it follows $f_B:=\min_{\w\in\B}f(\w)>-\infty$, which indicates that $\Omega$ is a bounded set, thereby resulting in \begin{align}\label{def-delta}
 \alpha := \frac{2f(\w^0)-2f_B}{\sigma} \geqslant  \|\w-\x\|^2,~~\forall~(\w,\x)\in\Omega.
\end{align}
To establish the convergence, we need the following assumption.
\begin{assumption}\label{ass-Lip} The gradient of $f$ is locally Lipschitz continuous on $\N$, where
\begin{align}\label{def-N}
\N:=\left\{\x\in\R^n:  \|\x\|_{\infty}\leqslant 1+\frac{\sqrt{n}+ \sqrt{\alpha}}{8}\right\}.
\end{align}  
\end{assumption}
One can observe that $\N$ is a larger box than $\B$, namely $\N\supset\B$. According to Algorithm \ref{BN-algo}, the update of $\x^k$ does not involve any box constraint, making it difficult to guarantee  $\x^k \in \B$. However, Lemma \ref{Lemma-decrease-L} 1) shows that $\x^k$ always lies within this larger region $\N$. 
As we mentioned in Section \ref{Section-pre}, the local Lipschitz continuity on a compact set implies the Lipschitz continuity on this compact set. Therefore, Assumption \ref{ass-Lip} indicates that the gradient of $f$ is Lipschitz continuous on $\N$ with a constant $\beta\in(0,\infty)$, which further implies
\begin{align}\label{Lip-Strong-smooth}
f(\w)\leqslant  f(\x) + \langle \nabla f(\z), \w-\x \rangle + \frac{\beta}{2}\|\w-\x\|^2,~~\forall~\w,\x\in \N,~\forall~\z\in\{\w,\x\}. 
\end{align} 
Based on two constants $\beta$ and $\lambda $, we choose $\sigma$ to satisfy
\begin{align}\label{choice-sigma}
 \sigma \geqslant 8\max\{\lambda, \beta\}.
\end{align} 
The first result shows that, under Assumption \ref{ass-Lip} and appropriate parameter choices, the objective function associated with $L(\w^{k},\x^{k},\y^{k};{\mu}_{k})$ is bounded below.
\begin{lemma}\label{Lemma-lower-bd-L}
 Suppose Assumption \ref{ass-Lip} holds. Let $\{(\w^{k},\x^{k},\y^{k};{\mu}_{k})\}$ be the sequence generated by Algorithm \ref{BN-algo} with $\Q^k$ and $\sigma$ chosen as (\ref{choice-Q}) and (\ref{choice-sigma}). If $\x^{k}\in \N$, then  
 \begin{equation}
 \label{lower-bd-L}
\begin{aligned}
\widetilde{L}^k:=   &~L(\w^{k},\x^{k},\y^{k};{\mu}_{k}) + \frac{3\sigma}{8}\|  \x^{k} -\w^{k} \|^2 \\
\geqslant  & ~f(\w^{k}) +  \frac{\sigma}{2} \|\x^{k}-\w^{k}\|^2 >-\infty.
\end{aligned}
 \end{equation}
\end{lemma}
\begin{proof} For any $\w, \x\in\R^n$ and any $t>0$,
\begin{equation}
\label{tri-ineq}
\begin{aligned}
\langle\w, \x\rangle &\leqslant ({1}/({2t}))\|\w\|^2+ ({t}/{2})\|\x\|^2,\\[1.0ex]
\|\w+ \x\|^2 &\leqslant  (1+1/t)\|\w\|^2+ (1+t)\|\x\|^2.
\end{aligned}\end{equation}
Moreover, for any $k \geqslant 0$ and $\w\in\R^n$,
\begin{align}\label{upbd-Q}
\| \Q^{k}\w\|\leqslant\lambda  \| \w\|\leqslant \frac{\sigma}{8}\| \w\|.  
\end{align}
The optimality conditions of \eqref{pi-penalty-update-w} and \eqref{pi-penalty-update-x} are 
\begin{align}
{\bf 0}&\in {\mu}_k \v^{k+1} - \y^k - \sigma(\x^k-\w^{k+1}) + N_{\B}(\w^{k+1}) \nonumber\\
\label{opt-w-2}
&\in {\mu}_k \v^{k+1} - \y^{k+1} + \sigma(\x^{k+1} -\x^k)+ N_{\B}(\w^{k+1}) ,\\
\label{opt-x-1}
{\bf 0}&=   \nabla f(\w^{k+1})+\y^k + (\sigma\I+\Q^{k+1})  ( \x^{k+1} -\w^{k+1})\\
\label{opt-x-2}
&= \nabla f(\w^{k+1})+\y^{k+1} +  \Q^{k+1} ( \x^{k+1} -\w^{k+1}),
\end{align}
where $\v^{k+1}\in\varphi(\w^{k+1})$ and  \eqref{opt-w-2} and \eqref{opt-x-2} result from \eqref{pi-penalty-update-y}. By Algorithm \ref{BN-algo}, $\w^k\in\B$ and thus $\varphi(\w^k) \geqslant 0$, thereby deriving that
\begin{align*}
\widetilde{L}^k & =  f(\x^{k}) + \mu_k\varphi(\w^k)+\langle\y^{k},\x^{k}-\w^{k}\rangle+ \frac{7\sigma}{8} \|\x^{k}-\w^{k}\|^2\\
 &\geqslant f(\w^{k})  +\langle \nabla f(\w^{k}) + \y^{k},\x^{k}-\w^{k}\rangle+\frac{7\sigma-4\beta}{8} \|\x^{k}-\w^{k}\|^2\\
  &\geqslant f(\w^{k}) - \langle \Q^k(\x^{k}-\w^{k}),\x^{k}-\w^{k}\rangle+\frac{13\sigma}{16} \|\x^{k}-\w^{k}\|^2\\
  &\geqslant f(\w^{k}) - \frac{1}{\sigma}\| \Q^k(\x^{k}-\w^{k})\|^2 +\frac{9\sigma}{16} \|\x^{k}-\w^{k}\|^2\\
        &\geqslant f(\w^{k})  +\frac{\sigma}{2} \|\x^{k}-\w^{k}\|^2\\
        &\geqslant f_B >- \infty,
\end{align*}
where the first inequality is from $\x^{k}\in \N$ and \eqref{Lip-Strong-smooth} by Assumption \ref{ass-Lip}, the second inequality is due to \eqref{opt-x-2} and \eqref{choice-sigma}, the third and fourth  inequalities hold by \eqref{tri-ineq} and \eqref{upbd-Q}, and the fifth inequality is because of $\w^k\in\B$ and $f_B=\min_{\w\in\B}f(\w)$.
\end{proof}
Based on the above lemma, we next prove that sequence $\{(\w^{k},\x^{k})\}$ is bounded in $\Omega \cap(\B\times \N)$ and sequence $\{\widetilde{L}^k\}$ is strictly increasing under the same conditions.
\begin{lemma}\label{Lemma-decrease-L}
 Suppose Assumption \ref{ass-Lip} holds. Let $\{(\w^{k},\x^{k},\y^{k};{\mu}_{k})\}$ be the sequence generated by Algorithm \ref{BN-algo} with $\Q^k$ and $\sigma$ chosen as (\ref{choice-Q}) and (\ref{choice-sigma}). The following results hold.
 \begin{itemize}[leftmargin=15pt]
 \item[1)] For any $k \geqslant 0$, $(\w^k, \x^{k})\in \Omega \cap(\B\times \N)$; 
 \item[2)] For any $k \geqslant 0$,
 \begin{align}\label{decrease-L}
\widetilde{L}^{k+1} -   \widetilde{L}^k
\leqslant   -  \frac{\sigma}{8} \|\x^{k+1}-\x^{k}\|^2. 
\end{align}
 \end{itemize}
\end{lemma}
\begin{proof} We prove 1) and 2) together. By \eqref{update-mu},  if mod$(k+1,k_0)=0$ and $\varphi(\w^{k+1})\neq 0$, then
 \begin{align}\label{update-mu-fact}
(\mu_{k+1}-\mu_k) \varphi(\w^{k+1}) \leqslant  {\rho\sigma} \|\x^{k+1}-\w^{k+1}\|^2  \leqslant    \frac{1}{6\sigma} \|\y^{k+1}-\y^{k}\|^2.
\end{align}
where the last inequality is due to $\rho\in(0,1/6]$ and \eqref{pi-penalty-update-y}. 
Otherwise $\mu_{k+1}=\mu_k$. Therefore, both cases yield the above fact. 
Moreover, as $\w^k\in\B\subset\N$ for any $k \geqslant 0$, Assumption \ref{ass-Lip} indicates  
\begin{align}
\label{Lip-wk1-wk}
\|   \nabla f(\w^{k+1})-\nabla f(\w^{k}) \| \leqslant \beta\|\w^{k+1} - \w^{k}\|\leqslant \frac{\sigma}{8}\|\w^{k+1} - \w^{k}\|,
\end{align}
where the last inequality is by \eqref{choice-sigma}.
By denoting
\begin{align}
\phi^{k+1}:=  \|  \x^{k} -\w^{k} \|^2 -   \| \x^{k+1} -\w^{k+1} \|^2,  
\end{align}
it follows from \eqref{opt-x-2} that 
\begin{align} 
\| \y^{k+1} - \y^{k} \|^2 &= \| \nabla f(\w^{k+1})-\nabla f(\w^{k}) +  \Q^{k+1} ( \x^{k+1} -\w^{k+1}) - \Q^{k} ( \x^{k} -\w^{k})\|^2\nonumber\\[1ex]
  &\leqslant \frac{4}{3} \| \nabla f(\w^{k+1})-\nabla f(\w^{k}) \|^2 +  8 \|\Q^{k+1} ( \x^{k+1} -\w^{k+1})\|^2 +8\| \Q^{k} ( \x^{k} -\w^{k})\|^2\nonumber\\
&\leqslant  \frac{ \sigma^2}{48} \|  \w^{k+1} -\w^{k} \|^2 +    \frac{    \sigma^2}{8}  \| \x^{k+1} -\w^{k+1} \|^2 +    \frac{    \sigma^2}{8}\|  \x^{k} -\w^{k} \|^2\nonumber\\
\label{gap-y1-y}
&\leqslant  \frac{ \sigma^2}{16} \|  \x^{k+1} -\x^{k} \|^2 +  \frac{3 \sigma^2}{16}    \| \x^{k+1} -\w^{k+1} \|^2 +  \frac{3 \sigma^2}{16}   \|  \x^{k} -\w^{k} \|^2 \\
&=  \frac{ \sigma^2}{16} \|  \x^{k+1} -\x^{k} \|^2 +  \frac{3 \sigma^2}{16} \phi^{k+1} +   \frac{3 \sigma^2}{8}   \| \x^{k+1} -\w^{k+1} \|^2  \nonumber\\
&=  \frac{ \sigma^2}{16} \|  \x^{k+1} -\x^{k} \|^2 +  \frac{3 \sigma^2}{16} \phi^{k+1} +   \frac{3}{8}   \| \y^{k+1} -\y^{k} \|^2,\nonumber 
\end{align}
where the first inequality is from \eqref{tri-ineq}, the second  inequality is from \eqref{Lip-wk1-wk} and \eqref{upbd-Q}, the third  inequality is from \eqref{choice-sigma},  and the last inequality is from $\|\w^{k+1} -\w^{k}\|^2 \leqslant 3\|\w^{k+1} -\x^{k+1}\|^2 + 3\|\x^{k+1} - \x^{k}\|^2 + 3\|\x^{k} -\w^{k}\|^2$. The above condition results in
\begin{align}\label{gap-yk1-yk}
\| \y^{k+1} - \y^{k} \|^2 
\leqslant \frac{\sigma^2}{10} \|  \x^{k+1} -\x^{k} \|^2 + \frac{3\sigma^2}{10}\phi^{k+1}.
\end{align}
We now decompose a gap $G^{k+1}$ by
\begin{align*}
G^{k+1}&:=L(\w^{k+1},\x^{k+1},\y^{k+1};{\mu}_{k+1})-L(\w^{k},\x^{k},\y^{k};{\mu}_{k})=G_{\mu}^{k+1} +G_{\y}^{k+1}+ G_{\x}^{k+1}+ G_{\w}^{k+1}, 
\end{align*}
where $G_{\mu}^{k+1}, G_{\y}^{k+1}, G_{\x}^{k+1}$, and $G_{\w}^{k+1}$ are defined by
\begin{align*}
G_{\mu}^{k+1}&:=L(\w^{k+1},\x^{k+1},\y^{k+1};{\mu}_{k+1})-L(\w^{k+1},\x^{k+1},\y^{k+1};{\mu}_{k}),\\
G_{\y}^{k+1}&:=L(\w^{k+1},\x^{k+1},\y^{k+1};{\mu}_{k})-L(\w^{k+1},\x^{k+1},\y^{k};{\mu}_{k}),\\
G_{\x}^{k+1}&:=L(\w^{k+1},\x^{k+1},\y^{k};{\mu}_{k})-L(\w^{k+1},\x^{k},\y^{k};{\mu}_{k}),\\
G_{\w}^{k+1}&:=L(\w^{k+1},\x^{k},\y^{k};{\mu}_{k})-L(\w^{k},\x^{k},\y^{k};{\mu}_{k}).
\end{align*}
Using \eqref{update-mu-fact} and \eqref{pi-penalty-update-y}, one can check that
\begin{align} \label{Gmu}
G_{\mu}^{k+1}&=(\mu_{k+1}-\mu_k) \varphi(\w^{k+1})\leqslant    \frac{ 1}{6\sigma} \|\y^{k+1}-\y^{k}\|^2\\
\label{Gy}
G_{\y}^{k+1}&=\langle\y^{k+1}-\y^{k},  \x^{k+1}-\w^{k+1}\rangle= \frac{1}{\sigma} \|\y^{k+1}-\y^{k}\|^2. 
\end{align}
In the sequel, we aim to estimate $G_{\x}^{k+1}$ and $G_{\w}^{k+1}$ so as to estimate gap $G^{k+1}$ by induction.

\underline{For $k=0$,}  $(\w^0, \x^0)\in \Omega$ because $\w^0=\x^0\in\{0,1\}^n\subseteq\B$ in Algorithm \ref{BN-algo}.
Note that $\y^0=-\nabla f(\w^0)$ and  $\w^0=\x^0$, which means that \eqref{opt-x-2} holds for $k=0$. This together with \eqref{pi-penalty-update-x-solution} leads to
\begin{align}
\label{bd-x1}
\|\x^{1} \|_{\infty}& = \| \w^{1} -  (\sigma\I+\Q^{1}) ^{-1}  ( \nabla f(\w^{1})+\y^0 )\|_{\infty} \nonumber\\
&= \| \w^{1} -  (\sigma\I+\Q^{1}) ^{-1}  ( \nabla f(\w^{1})-\nabla f(\w^{0}) + \Q^{0}(\x^0-\w^0) )\|_{\infty} \nonumber\\
&\leqslant \| \w^{1}\|_{\infty} +\|  (\sigma\I+\Q^{1}) ^{-1}  ( \nabla f(\w^{1})-\nabla f(\w^{0}) + \Q^{0}(\x^0-\w^0) )\|  \nonumber\\
&\leqslant 1+\|  (\sigma\I+\Q^{1}) ^{-1} \|(\|   \nabla f(\w^{1})-\nabla f(\w^{0}) \|+  \|\Q^{0}(\x^0-\w^0)\|)\nonumber\\
&\leqslant  1+ (\beta/\sigma)\|   \w^{1} - \w^{0} \|+ (\lambda /\sigma)\|   \x^{0} - \w^{0} \|\nonumber\\
&\leqslant 1+ (1/8)\|   \w^{1} - \w^{0} \|+ (1/8)\|   \x^{0} - \w^{0} \|\nonumber\\
&\leqslant  1 + (\sqrt{n}+\sqrt{\alpha})/8,
\end{align}
where  the second equality is from \eqref{opt-x-2} with $k=0$, the first inequality is because of $\|\w\|_\infty \leqslant \|\w\|$, the third inequality is from \eqref{Lip-wk1-wk} and \eqref{upbd-Q}, the fourth inequality is from \eqref{choice-sigma}, and the last inequality is due to $\w^0, \w^1\in\B$ and \eqref{def-delta} (because of $(\w^0, \x^0)\in\Omega$). Therefore, $(\w^1, \x^1)\in\B\times\N$, which by $\x^{0}\in\N$, Assumption \ref{ass-Lip}, and \eqref{Lip-Strong-smooth} derives that
\begin{align*}
 f(\x^{1}) \leqslant f(\x^{0}) +  \langle \nabla f(\x^{1}) ,\x^{1}-\x^{0}\rangle + \frac{ \beta}{2}\|\x^{1}-\x^{0}\|^2.  
\end{align*}
This allows us to obtain the following chain of inequalities,
\begin{align}
G_{\x}^{1}&=f(\x^{1}) - f(\x^{0}) + \langle\y^{0},\x^{1}-\x^{0}\rangle+\frac{\sigma}{2}\left(\|\x^{1}-\w^{1}\|^2-\|\x^{0}-\w^{1}\|^2\right)\nonumber\\
&\leqslant - \frac{\sigma-\beta}{2}\|\x^{1}-\x^{0}\|^2  + \langle \nabla f(\x^{1})+ \y^{0} + \sigma (\x^{1}-\w^{1}),\x^{1}-\x^{0}\rangle \nonumber\\
&\leqslant - \frac{7\sigma}{16}\|\x^{1}-\x^{0}\|^2+ \langle \nabla f(\x^{1}) - \nabla f(\w^{1}) -\Q^{1} (\x^{1}-\w^{1}),\x^{1}-\x^{0}\rangle  \nonumber\\
&\leqslant - \frac{ \sigma}{4}\|\x^{1}-\x^{0}\|^2+ \frac{4}{3\sigma}\| \nabla f(\x^{1}) - \nabla f(\w^{1}) -\Q^{1} (\x^{1}-\w^{1})\|^2  \nonumber\\
&\leqslant - \frac{\sigma}{4}\|\x^{1}-\x^{0}\|^2 + \frac{ 8\beta^2 }{ 3\sigma} \| \x^{1} - \w^{1} \|^2 +\frac{ 8\lambda ^2}{ 3\sigma} \| \x^{1} - \w^{1} \|^2  \nonumber\\ 
&\leqslant  - \frac{ \sigma}{4}\|\x^{1}-\x^{0}\|^2 + \frac{ \sigma}{12} \| \x^{1} - \w^{1} \|^2  \nonumber\\
 \label{Gx}
 &\leqslant  - \frac{ \sigma}{4}\|\x^{1}-\x^{0}\|^2 + \frac{1}{12\sigma} \| \y^{1} - \y^{0} \|^2 ,
\end{align}
where the second inequality is from \eqref{choice-sigma} and  \eqref{opt-x-1}, the third inequality is from \eqref{tri-ineq}, the fourth inequality is due to $\w^1, \x^{1}\in\N$ and Assumption \ref{ass-Lip} and \eqref{upbd-Q}, the fifth inequality is by \eqref{choice-sigma}, and the last inequality is from \eqref{pi-penalty-update-y}. Finally, the optimality of $\w^1$ to problem \eqref{pi-penalty-update-w-solution} suffices to
\begin{align}\label{Gw}
G_{\w}^{1}=\mu_0 \varphi(\w^{1}) - \langle\y^{0},\w^{1}\rangle+\frac{\sigma}{2}\|\x^{0}-\w^{1}\|^2 - \left(\mu_0 \varphi(\w^{0})  - \langle\y^{0}, \w^{0}\rangle+ \frac{\sigma}{2} \|\x^{0}-\w^{0}\|^2\right)
\leqslant 0.
\end{align}
Combining (\ref{Gmu}), (\ref{Gy}), (\ref{Gx}), and (\ref{Gw}) derives that
\begin{align*}
G^{1}&\leqslant   - \frac{\sigma}{4}\|\x^{1}-\x^{0}\|^2  + \frac{5}{4\sigma} \| \y^{1} - \y^{0} \|^2   \\
&\leqslant  -  \frac{\sigma}{8} \|\x^{1}-\x^{0}\|^2  +\frac{3\sigma}{8} \phi^{1}  ,
\end{align*}
where the last inequality is from \eqref{gap-yk1-yk}. The above condition can show \eqref{decrease-L} with $k=0$, namely, 
\begin{align}\label{decrease-L-0} 
\widetilde{L}^1 - \widetilde{L}^0   \leqslant    -  \frac{\sigma}{8} \|\x^{1}-\x^{0}\|^2.  
\end{align}
\underline{For $k=1$,} we first prove $(\w^1,\x^1)\in\Omega$. Based on the initialization of Algorithm \ref{BN-algo}, there is
$$\widetilde{L}^0 = f(\w^0).$$
In addition, we have shown $\x^1\in\N$, which recalls \eqref{lower-bd-L} and \eqref{decrease-L-0} derives
\begin{align*} 
f(\w^{1}) +  \frac{\sigma}{2} \|\x^{1}-\w^{1}\|^2 \leqslant \widetilde{L}^1 \leqslant \widetilde{L}^0=f(\w^0). 
\end{align*}
As a result of the above condition and \eqref{def-omega}, we obtain $(\w^1,\x^1)\in\Omega$, thereby resulting in $$(\w^1, \x^1)\in \Omega \cap(\B\times\N).$$ Then the similar reasoning to prove $\x^1\in\N$ (see \eqref{bd-x1}) and  \eqref{decrease-L-0} enables showing $\x^2\in\N$  and  
\begin{align*}\label{decrease-L-1} 
\widetilde{L}^2 - \widetilde{L}^1   \leqslant      -  \frac{\sigma}{8} \|\x^{1}-\x^{2}\|^2. 
\end{align*}
\underline{For $k \geqslant  2$,} repeating the above proof enables the establishment of the results.
\end{proof}
Using the above lemma, the following theorem establishes the global convergence of Algorithm \ref{BN-algo}, showing that any accumulation point of $\{ \x^{k}\}$ is a P-stationary point of a new penalty model $\min_{\x\in \B}F(\x;{\mu_\infty})$, where $\mu_\infty$ is the limit of $\{\mu_k\}$. As stated in Theorem \ref{firstPST}, P-stationary points are closely related to local/global minimizers. To obtain a stronger convergence property (i.e., the whole sequence converges to a P-stationary point), an additional condition on final penalty parameter $\mu_\infty$ is imposed, that is, ${\mu_\infty>\overline{\mu}}$. This condition can be ensured if initializing $\mu_0$ as ${\mu_0>\overline{\mu}}$ due to ${\mu_\infty \geqslant \mu_0}$ from \eqref{update-mu}.
\begin{theorem}\label{theorem-global-con}
 Suppose Assumption \ref{ass-Lip} holds. Let $\{(\w^{k},\x^{k},\y^{k};{\mu}_{k})\}$ be the sequence generated by Algorithm \ref{BN-algo} with $\Q^k$ and $\sigma$ chosen as (\ref{choice-Q}) and (\ref{choice-sigma}). The following results hold.
 \begin{itemize}[leftmargin=15pt]
 \item[1)] $\{ L(\w^{k},\x^{k},\y^{k};{\mu}_{k})\}$ and $\{\widetilde{L}^{k}\}$ converge to the same value, $\{\mu_{k}\}$ converges (to $\mu_\infty$),  and
$$\lim_{k\to\infty}\|\w^{k+1}-\w^{k}\|=\lim_{k\to\infty}\|\x^{k+1}-\x^{k}\|=\lim_{k\to\infty}\|\y^{k+1}-\y^{k}\|=\lim_{k\to\infty}\|\w^{k}-\x^{k}\|=0.$$
 \item[2)] Let $\x^\infty$ be any accumulation point  of sequence $\{ \x^{k}\}$. Then there is $\tau=1/\sigma$ such that $\x^\infty$ is a P-stationary point of problem $ \min_{\x\in \B} F(\x;\mu_\infty)$.
 \item[3)]  Sequence $\{ \x^{k}\}$ converges to a binary P-stationary point $\x^{\infty}$  if $\mu_\infty>\overline{\mu}$.
 \end{itemize}
\end{theorem}
\begin{proof} 1) From \eqref{decrease-L}, sequence $\{\widetilde{L}^{k}\}$ is non-increasing.  Moreover, from \eqref{lower-bd-L}, for any $k \geqslant 0$,
$$\widetilde{L}^{k} \geqslant f(\w^{k}) + \frac{\sigma}{2} \|\x^{k}-\w^{k}\|^2 \geqslant  f_B. $$
Therefore, sequence $\{\widetilde{L}^{k}\}$ converges. Again by \eqref{decrease-L} and $\widetilde{L}^0 = f(\w^0)$, for any $\ell \geqslant k \geqslant 0$, 
\begin{equation}
\label{sum-x1-x} 
\begin{aligned} 
\sum_{k=0}^{\ell-1}\frac{\sigma}{8} \|\x^{k+1}-\x^{k}\|^2   \leqslant \sum_{k=0}^{\ell-1} (\widetilde{L}^{k} -   \widetilde{L}^{k+1}) = \widetilde{L}^{0} -   \widetilde{L}^{\ell}  \leqslant f(\w^0)-f_B.
\end{aligned}\end{equation}
 The above condition implies $\lim_{k\to\infty}\|\x^{k+1}-\x^{k}\|=0$. 
 It follows from \eqref{gap-y1-y} that
\begin{align*}
\| \y^{k+1} - \y^{k} \|^2 & \leqslant  \frac{ \sigma^2}{16} \|  \x^{k+1} -\x^{k} \|^2 +  \frac{3 \sigma^2}{16}  \| \x^{k+1} -\w^{k+1} \|^2 + \frac{3 \sigma^2}{16}  \|  \x^{k} -\w^{k} \|^2\\ 
& =  \frac{ \sigma^2}{16} \|  \x^{k+1} -\x^{k} \|^2 +  \frac{3   }{16} \| \y^{k+1} - \y^{k} \|^2 +  \frac{3   }{16}   \| \y^{k} - \y^{k-1} \|^2,
\end{align*}
where the equality is from \eqref{pi-penalty-update-y}, which results in
\begin{align*}
\sum_{k=0}^{\ell-1} 13 \| \y^{k+1} - \y^{k} \|^2   \leqslant  \sum_{k=0}^{\ell-1} \sigma^2  \|  \x^{k+1} -\x^{k} \|^2 +   \sum_{k=0}^{\ell-1} 3  \| \y^{k} - \y^{k-1} \|^2.
\end{align*}
This further gives rise to
\begin{align}\label{sum-y1-y}
\sum_{k=0}^{\ell-1} 10 \| \y^{k+1} - \y^{k} \|^2 &\leqslant 
13 \| \y^{\ell} - \y^{\ell-1} \|^2  + \sum_{k=0}^{\ell-1} 10 \| \y^{k} - \y^{k-1} \|^2\nonumber\\
&=\sum_{k=0}^{\ell-1} 13 \| \y^{k+1} - \y^{k} \|^2 -\sum_{k=0}^{\ell-1} 3 \| \y^{k} - \y^{k-1} \|^2\nonumber\\
&   \leqslant  \sum_{k=0}^{\ell-1} \sigma^2  \|  \x^{k+1} -\x^{k} \|^2\nonumber\\
& \leqslant  8\sigma(f(\w^0)-f_B),  
\end{align}
where the last inequality is due to \eqref{sum-x1-x}. In the above inequalities, we let  $\y^{-1}=\y^{0}$. The above condition implies $\|\y^{k+1}-\y^{k}\|\to 0$ as $k\to\infty$, which by \eqref{pi-penalty-update-y} derives $ \|\x^{k+1}-\w^{k+1}\|\to 0$. This and \eqref{lower-bd-L} yield $ \widetilde{L}^{k}- L(\w^{k},\x^{k},\y^{k};{\mu}_{k}) \to 0$. Finally,    $$ \|\w^{k+1}-\w^{k}\|=\|\w^{k+1}-\x^{k+1}+\x^{k+1}-\x^{k}+\x^{k}-\w^{k}\|\to 0.$$
It follows from \eqref{update-mu} that if mod$(k+1,k_0)=0$ and $\varphi(\w^{k+1})\neq 0$, then
 \begin{align*} 
 \mu_{k+1}-\mu_k  \leqslant  \frac{\rho\sigma}{\epsilon}\|\x^{k+1}-\w^{k+1}\|^2 \leqslant  \frac{1}{6\epsilon\sigma}\|\y^{k+1}-\y^{k}\|^2.
\end{align*}
Otherwise $\mu_{k+1}=\mu_k$. Therefore, both cases result in the above condition.  Then
 \begin{align*} 
\mu_{\ell}-\mu_0 = \sum_{k=0}^{\ell-1} (\mu_{k+1}-\mu_k)  \leqslant   \sum_{k=0}^{\ell-1} \frac{1}{6\epsilon\sigma}\|\y^{k+1}-\y^{k}\|^2 
\leqslant \frac{2(f(\w^0)-f_B)}{15\epsilon},
\end{align*}
where the last inequality is from \eqref{sum-y1-y}. The above condition means that sequence $\{\mu_k \}$ is bounded from above, which by \eqref{update-mu} that it is non-decreasing shows its convergence.

2) Let $(\w^\infty,\x^\infty,\y^\infty;{\mu}_\infty)$ be an accumulation point of sequence $\{(\w^{k},\x^{k},\y^{k};{\mu}_{k})\}$, namely, there is a subset $\mathcal{K}\subseteq\{1,2,3,\ldots\}$ such that $\lim_{k\in\mathcal{K},k\to\infty}(\w^{k},\x^{k},\y^{k};{\mu}_{k})=(\w^\infty,\x^\infty,\y^\infty;{\mu}_\infty)$. Clearly, $\w^\infty=\x^\infty$ due to claim 1).  Then taking the limit of the right-hand side of \eqref{opt-x-2} along with $k\in\mathcal{K}$ results in $\y^\infty=-\nabla f(\w^\infty) = -\nabla f(\x^\infty).$ It is observed from \eqref{pi-penalty-update-w-solution}  that
\begin{align*}
\w^{k+1}
 \in{\rm Prox}_{({\mu}_k/{\sigma})\varphi}^\B(\x^k+\y^k/\sigma) = {\rm Prox}_{({\mu}_k/{\sigma})(\varphi+\delta_\B)}(\x^k+\y^k/\sigma).
\end{align*}
Using \cite[Theorem 1.25]{rock98}, taking the limit of both sides of the above inclusion along with $k\in\mathcal{K}$ yields the following inclusion,
\begin{align*} 
\x^\infty=\w^\infty&\in{\rm Prox}_{({\mu}_\infty/{\sigma})(\varphi+\delta_\B)}\left(\x^\infty+\y^\infty/\sigma\right) \\
&= {\rm Prox}_{({\mu}_\infty/{\sigma})(\varphi+\delta_\B)}\left(\x^\infty-\nabla f(\x^\infty)/\sigma\right)\\
&=   {\rm Prox}_{({\mu}_\infty/{\sigma})\varphi}^\B\left(\x^\infty-\nabla f(\x^\infty)/\sigma\right).
\end{align*}
Therefore, $\x^\infty$ is a P-stationary point with $\tau=1/\sigma$ of problem $\min_{\x\in \B}F(\x;{\mu_\infty})$. 

3) Based on Theorems \ref{firstPST} and \ref{firstKKT}, a P-stationary point is a KKT point that is binary. Hence, $\x^\infty\in\{0,1\}^n$. This means that $\x^\infty$ is isolated, which together with $\lim_{k\to\infty}\|\x^{k+1}-\x^k \|=0$ and \cite[Lemma 4.10]{more83} delivers the  whole sequence convergence. 
\end{proof}
Finally, under an additional condition on the initialization of $\mu_0$, we establish the local convergence rate. As outlined below, sequence ${\x^k}$ converges to its binary limit point $\x^{\infty}$, which is also a P-stationary point of problem $\min_{\x\in \B}F(\x;{\mu_\infty})$, at a linear rate, and sequence $\{\w^k\}$ eventually coincides with this limit, namely, $\w^{k}\equiv\x^\infty$.
\begin{theorem}\label{finite-termination}
Suppose Assumption \ref{ass-Lip} holds. Let $\{(\w^{k},\x^{k},\y^{k};{\mu}_{k})\}$ be the sequence generated by Algorithm \ref{BN-algo} with $\Q^k$ and $\sigma$ chosen as (\ref{choice-Q}) and (\ref{choice-sigma}).  Suppose that
\begin{equation}\label{choice-mu}
\mu_0>\bar{\mu}.
\end{equation}
Then there is an integer $\kappa >0$ such that for any $k \geqslant \kappa $,
\begin{equation}\label{convergence-rate}\begin{aligned}
\w^{k+1}\equiv\w^{k}\equiv\x^\infty&\in\{0,1\}^n,~~ \\
\|\x^k-\x^\infty\| &\leqslant \sqrt{\alpha}\gamma^{k-\kappa },~~\\
\|\y^k + \nabla f(\x^\infty)\| &\leqslant \sqrt{\alpha}\lambda \gamma^{k-\kappa },
\end{aligned}\end{equation}
where $\gamma:=\lambda /(\sigma-\lambda )\in(0,1/7]$. Consequently, after the following number of iterations, \begin{equation}\label{finite-K}
\kappa + \left\lceil \log_{\gamma} \left(\frac{\varepsilon}{\max\{1,\lambda \}\sqrt{\alpha}}\right)\right\rceil,
\end{equation}
stopping conditions in (\ref{stop}) can be satisfied, where $\lceil a\rceil$ is the ceiling of $a$.
\end{theorem}
\begin{proof}
Since $\{\mu_k\}$ is nondecreasing, condition
\eqref{choice-mu} implies
\[
\mu_\infty\geq\mu_0>\bar{\mu}.
\]
It then follows from Theorem~\ref{theorem-global-con} that $\x^k\to\x^\infty$ and $\|\x^k-\w^k\|\to0$, where $\x^\infty\in\{0,1\}^n$ is a P-stationary point due to $\mu_\infty>\overline{\mu}$. Therefore,
\begin{equation*}
\w^k\to\x^\infty,\qquad \|\Q^k(\x^k-\w^k)\|
\leq
\lambda\|\x^k-\w^k\|
\to0.
\end{equation*}
Moreover, from \eqref{opt-x-2}, we have
\[
\y^k
=
-\nabla f(\w^k)-\Q^k(\x^k-\w^k).
\]
The above two conditions and the continuity of
$\nabla f$ yield 
\begin{equation}
\y^k\to-\nabla f(\x^\infty).
\label{yk-limit}
\end{equation}
Furthermore,  by $\x^k-\w^{k+1}=\x^k-\w^{k}+\w^k-\w^{k+1}\to0$, it follows
\begin{align*}
\lim_{k\to\infty}
\left\|
\y^k+\sigma(\x^k-\w^{k+1})
\right\|_\infty=
\|\nabla f(\x^\infty)\|_\infty  \leq
\max_{\x\in\mathbb{B}}
\|\nabla f(\x)\|_\infty =
c\bar{\mu}
<
c\mu_\infty. 
\end{align*}
Based on this condition and $\mu_k\to\mu_\infty$ in Theorem \ref{theorem-global-con} 1), we have
\[
c\mu_k-
\left\|
\y^k+\sigma(\x^k-\w^{k+1})
\right\|_\infty
\to
c\mu_\infty-\|\nabla f(\x^\infty)\|_\infty
\geq
c(\mu_\infty-\bar{\mu})>0.
\]
Therefore, there exists an integer $\kappa_0>0$ such that, for any
$k\geq\kappa_0$,
\begin{equation}
\left\|
\y^k+\sigma(\x^k-\w^{k+1})
\right\|_\infty
<
c\mu_k.
\label{strict-bound}
\end{equation}
We next show that $\w^{k+1}$ is binary for all sufficiently large
$k$. Suppose, to the contrary, that for some $k\geq\kappa_0$ there
exists an index $i$ such that $w_i^{k+1}\in(0,1).$
Then condition \eqref{opt-w-2} indicates that
\[
\mu_k v_i^{k+1}
-y_i^k
-\sigma(x_i^k-w_i^{k+1})
=0,
\]
where $v_i^{k+1}\in\partial g(w_i^{k+1}),$ which by condition c2) in Definition~\ref{def-SPF} leads to
\begin{align*}
c\mu_k
\leq
\mu_k|v_i^{k+1}|
=
\left|
y_i^k+\sigma(x_i^k-w_i^{k+1})
\right|
\leq
\left\|
\y^k+\sigma(\x^k-\w^{k+1})
\right\|_\infty
<
c\mu_k,
\end{align*}
which is a contradiction. Thus,
\begin{equation*}
\w^{k+1}\in\{0,1\}^n,
\qquad
k\geq\kappa_0. 
\end{equation*}
Since $\w^k\to\x^\infty$ and both $\w^k$ and $\x^\infty$ are
binary for all sufficiently large $k$, there exists an integer
$\kappa\geq\kappa_0+1$ such that
\begin{equation}\label{wk1=wk}
\w^{k+1}\equiv\w^k\equiv\x^\infty,
\qquad
k\geq\kappa. 
\end{equation}
Using this condition and \eqref{opt-x-2}, for any $k\geq\kappa$, we obtain
\begin{align*}
\sigma\|\x^{k+1}-\w^{k+1}\|
=
\|\y^{k+1}-\y^k\|
&=
\left\|
\Q^{k+1}(\x^{k+1}-\w^{k+1})
-
\Q^k(\x^k-\w^k)
\right\|\\
&\leq
\lambda\|\x^{k+1}-\w^{k+1}\|
+
\lambda\|\x^k-\w^k\|.
\end{align*}
which immediately results in
\begin{align*}
\|\x^{k+1} -\w^{k+1} \|  \leqslant   \gamma \|   \x^{k} -\w^{k}\|,
\end{align*}
where $\gamma=\lambda /(\sigma-\lambda )\in(0,1/7]$ due to \eqref{choice-sigma}. 
This further yields that
\begin{align*}
 \|   \x^{k} -\w^{k}\|\leqslant \gamma \|   \x^{k-1} -\w^{k-1}\| \leqslant \cdots \leqslant \gamma^{k-\kappa } \|   \x^{\kappa } -\w^{\kappa }\|\leqslant \sqrt{\alpha}\gamma^{k-\kappa } ,
\end{align*}
where the last inequality is from $(\w^k, \x^{k})\in \Omega$ for any $k \geqslant 0$ and \eqref{def-delta}. The above condition and  \eqref{opt-x-2} allow us to derive that
\begin{align*}
\| \nabla f(\w^{k})+\y^{k}\|\leqslant \|  \Q^{k} ( \x^{k} -\w^{k})\| \leqslant \sqrt{\alpha}\lambda \gamma^{k-\kappa }.
\end{align*}
Using \eqref{wk1=wk} and replacing $\w^{k}$ by $\x^\infty$ in the above two conditions show  \eqref{convergence-rate}. Finally, based on \eqref{convergence-rate}, it is easy to verify that conditions in (\ref{stop}) can be satisfied after the $t$th iteration, where $t$ is the number defined in \eqref{finite-K}.
\end{proof}
\begin{remark}\label{remark-sufficient-cond}
Theorem \ref{finite-termination} establishes a sufficient condition for achieving the linear convergence rate, but it is not necessary. In particular, the algorithm may still converge even when the initial parameter $\mu_0$ does not satisfy the required inequality. In practice, the quantity $\bar{\mu}$ is problem-dependent and typically unavailable in exact form; only a loose upper bound can be obtained, which offers limited guidance for choosing $\mu_0$. 

More importantly, our algorithm enforces a monotonic increase of $\mu_k$ along the iterations. Therefore, even if $\mu_0$ fails to meet the sufficient condition, there exists a certain iteration index $k$ such that $\mu_k$ exceeds the theoretical threshold, namely, $\mu_k>\overline{\mu}$. From that point onward, $\mu_k$ can be viewed as a new initialization satisfying the condition, ensuring the linear convergence rate thereafter. This observation justifies the use of heuristic choices for $\mu_0$ in our experiments: while they provide a reasonable starting point, the monotonic update mechanism guarantees eventual fulfillment of the sufficient condition.

Finally, we would like to emphasize that the threshold $\overline{\mu}$ is generally conservative. 
	Indeed, the proof of Theorem~\ref{firstKKT} 1) only requires the gradient bound to be considered over a certain smaller subset (say $\mathbb{U}\subseteq\B$) consisting of the KKT points, rather than over the entire box $\B$. Hence, a potentially tighter threshold should be
	$$
	\mu_*:= \sup_{\x\in\mathbb{U}}\frac{\|\nabla f(\x)\|_\infty}{c}
	\leq
	 \max_{\x\in\B}\frac{\|\nabla f(\x)\|_\infty}{c}
	=\overline{\mu}.
	$$
	However, since $\mathbb{U}$ is generally difficult to characterize explicitly, we use 
	the simple box $\B$ to obtain a readily computable sufficient condition, which 
	may make $\overline{\mu}$ substantially larger than the penalty level $\mu$ actually 
	needed for practical implementation. Therefore, if $\mathbb{U}$ is attainable (i.e., $\mu_*$ is computable), then $\mu_0>\bar{\mu}$ in Theorem \ref{finite-termination} can be relaxed to
$\mu_0>\mu_*$, which provides a more informative guideline for choosing $\mu_0$ in practice.

\end{remark}
A direct result of Theorem \ref{finite-termination} is the following termination within finitely many steps.
\begin{corollary} 
Suppose Assumption \ref{ass-Lip} holds. Let $\{(\w^{k},\x^{k},\y^{k};{\mu}_{k})\}$ be the sequence generated by Algorithm \ref{BN-algo} with $\Q^k\equiv\mathbf{0}$, $\sigma$ and $\mu_0$ chosen as (\ref{choice-sigma}) and (\ref{choice-mu}).
Then there is an integer $\kappa >0$ such that for any $k \geqslant \kappa $,
\begin{equation}\label{convergence-rate-finite}\begin{aligned}
\w^{k} \equiv \x^k \equiv\x^\infty&\in\{0,1\}^n,~~ 
\y^k \equiv - \nabla f(\x^\infty).
\end{aligned}\end{equation}
Consequently, Algorithm \ref{BN-algo} terminates within finitely many iterations.
\end{corollary}
\begin{proof} When $\Q^k\equiv\mathbf{0}$, for any $k \geqslant \kappa $,
$$\y^{k}=-\nabla f(\w^{k})=-\nabla f(\w^\infty),$$
where the two equalities are from \eqref{opt-x-2} and \eqref{wk1=wk}. This condition and \eqref{pi-penalty-update-y} enable  $\x^{k}=\w^{k}=\w^\infty$, where the second equality is due to \eqref{wk1=wk}. These prove conditions in \eqref{convergence-rate-finite}, which guarantee \eqref{stop} within finitely many iterations.
\end{proof}

\section{Numerical experiment}
In this section, we test ShaPeak (available at \url{https://github.com/ShuaiLi2025/ShaPeak}) on solving four problems: the recovery problems, the Multiple-Input-Multiple-Output (MIMO) detection, the Max-Cut problems,  and the quadratic unconstrained binary optimization (QUBO). All MATLAB codes are implemented in MATLAB (R2021a) on a desktop computer equipped with an Intel(R) Xeon W-3465X CPU (2.50 GHz) and 128 GB of RAM. The Python implementation is carried out using JAX 0.6.1 on the same machine with an NVIDIA RTX A5000 GPU. Unless otherwise specified, we use MATLAB’s default random number generator \texttt{rng(`default')}. 

\subsection{Hyperparameter settings} \label{tab:Pars}
Hyperparameters in Algorithm \ref{BN-algo} are set as: $\w^0=\x^0$, $\rho=1$, $\epsilon=10^{-10}$, $\varepsilon=\sqrt{n}/10^5$ for all testing problems. It is noted that to guarantee the convergence in Theorem \ref{theorem-global-con}, $\sigma$ should be chosen to satisfy   ${\sigma > 8\max\{\lambda,\beta\}}$ in (\ref{choice-sigma}). However, this condition is sufficient but unnecessary. Moreover, since $\beta$ is relevant to the Lipschitz constant of $\nabla f$ and is problem-dependent, its estimation is uneasy. Therefore, enforcing condition  (\ref{choice-sigma}) provides little guidance for numerical experiments.

\begin{table}[!th]
	\renewcommand{\arraystretch}{1.0}\addtolength{\tabcolsep}{1.0pt}
	\centering
	\caption{Parameter settings of ShaPeak for different problems.}
	\label{tab:shapeak-parameters}
		\begin{tabular}{lcccccc}
			\toprule
			
			& Recovery
			& Classical 
			& One-bit
			& Max-
			& QUBO
			& QUBO\\		 
			Para.& problem
			& MIMO
			& MIMO
			& Cut
			& (Simulated)
			& (BiqBin) \\
			\midrule
			$\x^0$
			& $\mathbf{0}$
			& $\mathbf{0}$
			& $\mathbf{0}$
			& $\widetilde{\x}^0$ 
			& $\mathbf{0.5}$
			& $\widetilde{\x}^0$  \\
			
			
			$\y^0$
			& $\mathbf{0}$
			& $\mathbf{0}$
			& $\mathbf{0}$
			& $-\nabla f(\x^0)$
			& $\mathbf{0}$
			& $-\nabla f(\x^0)$ \\
			
			$\Q^k$
			& $\mathbf{A}^{\top}\mathbf{A}$
			& $\mathbf{A}^{\top}\mathbf{A}$
			& $ \widetilde{\mathbf{A}}^{\top}\widetilde{\mathbf{A}} $
			& $\widetilde{\mathbf{Q}}^k$
			& $\mathbf{0}$
			& $\widetilde{\mathbf{Q}}^k$ \\\addlinespace[2pt]
			
			$\mu_0$
			& $\dfrac{5\|\mathbf{b}^{\top}\mathbf{A}\|}{\sqrt{n}10^t}$
			& $ \dfrac{\sqrt{n} \|\y^{\top}\H\|_{\infty}}{10^4}$
			& $\dfrac{\|\y^{\top}\H\|_{\infty}\log n^r}{10^3}$
			& $\dfrac{1}{10^{6}}$
			& $\dfrac{\|\mathbf{Q}\|_F}{2\times10^5}$
			& $\dfrac{1}{10^{5}}$ \\\addlinespace[4pt]
			
			$\sigma_0$
			& $\widetilde{\sigma}$
			& $\dfrac{32}{\log n}$
			& $\dfrac{n}{10^3}$
			& $\widetilde{\sigma}$
			& $0.01$
			& $12$ \\
			
			$k_0$
			& $\widetilde{k}_0$
			& $10$
			& $10$
			& $100$
			& $10$
			& $100$ \\
			
			$\eta$
			& $2.5$
			& $3$
			& $\dfrac{r+35}{16}$
			& $2.25$
			& $2.1$
			& $2.25$ \\
%
%
%
			\bottomrule
	\end{tabular}
\end{table}

 Based on empirical experience,  $\sigma$ are usually updated during iterations for augmented Lagrangian methods. Therefore, in our numerical experiments, we update it as follows:
\begin{equation}\label{update-sigma}
\sigma_{k+1}=
\begin{cases}
1.2\,\sigma_k, & 
\text{if } \text{~~mod}(k,10)=0,~{\rm tol}_k  
>  \epsilon,\\ 
\sigma_k / 1.1, & 
\text{if } \text{~~mod}(k,10)=0,~{\rm tol}_k  <  \epsilon,~ \varphi(\w^{k})\neq 0,\\ 
\sigma_k, & \text{otherwise}.
\end{cases}
\end{equation}
where ${\rm tol}_k $ is defined by \eqref{stop}. 
Other hyperparameters $(\x^0,\y^0,\Q^k,\mu_0,\sigma_0,k_0,\eta)$ for different testing problems are set as these in Table \ref{tab:shapeak-parameters}.  
For recovery problems, $t=2q-4+10s/n$, 
		$\widetilde{\sigma}=\min(0.5,(0.6-s/n)10^{p-3})$, and 
		$\widetilde{k}_0=\max\{10,\,2\lceil 100s/(n(p-1))\rceil\}$.
		For 1-bit MIMO, $r=1$ if $N<5000$, and $r=5$ otherwise. Here, $\widetilde{\mathbf{A}}:=\mathrm{Diag}(\y)\mathbf{H}/{\varrho}$ and ${\rm Diag}(\y)$ denotes the diagonal matrix formed by $\y$.
		For Max-Cut, $\widetilde{\sigma}=1$ if the number of nodes is less than $7000$, and $\widetilde{\sigma}=2$ otherwise.		
	Following the approach in \cite{chen2023montecarlopolicygradient,liu2025smoothing}, where multiple runs under random initializations are used to obtain better solutions, we apply the same strategy to Max-Cut and QUBO (BiqBin). Specifically, we randomly sample 100 initial points $\widetilde{\x}^0$ from the uniform distribution on $[0,1]$, run ShaPeak in parallel, and retain the best solution across all runs.
	Moreover, for these two testing problems, matrix $\widetilde{\mathbf{Q}}^k$ is updated using the Adam  second-moment  \cite{kinga2015method}. To be more specific, by letting $\boldsymbol{g}^k:=(\nabla f(\w^k)+\y^k)/\sigma$  and $\odot$ denote the Hadamard product, we update 
\begin{equation*}\label{choice-v-shapeak}
	\m^k =\beta_1 \m^{k-1}+(1-\beta_1)\boldsymbol{g}^k, \qquad  \v^k =\beta_2 \v^{k-1}+(1-\beta_2)(\boldsymbol{g}^k\odot \boldsymbol{g}^k), 
\end{equation*}
where  $\beta_1\in(0,1)$ and $\beta_2\in(0,1)$. Then choose preconditioning matrix $\Q^k$ to satisfy 
\begin{equation*}\label{choice-Q-adam}
	\sigma\mathbf{I} + \Q^k
	=  \mathrm{Diag}\left(\sigma\boldsymbol{g}^k\odot\left(\frac{\v^k+10^{-8}}{1-\beta_2^k}\right)^{1/2}\odot\left( \frac{\alpha{\m}^k}{1-\beta_1^k} \right)^{-1}\right),
\end{equation*} 
where $\alpha>0$, $\v^{1/2}:=( \sqrt{v_1},\ldots,\sqrt{v_n})^\top$, $\m^{-1}:=( 1/m_1,\ldots, 1/m_n)^\top$, and $\mathrm{Diag}(\v)$ denotes the diagonal matrix formed by $\v$. For Max-Cut and QUBO (BiqBin), $(\beta_1,\beta_2,\alpha)=(0.9,0.999,3.5)$. 

%
 
\subsection{Recovery problem}\label{sec:rec}
The problem is to recover a signal $\x^*$ from a linear model
${{\textbf b}={\textbf A}\x^*+  {\boldsymbol \varepsilon}}$, where ${{\textbf A}\in\R^{m\times n}}$,  ${\textbf b}\in \R^m$, and ${\boldsymbol \varepsilon}\in \R^m$ is the noise. By letting ${\|\x \|_q^{q}=\sum_{i=1}^n|x_i|^q}$ and ${q>1}$, this problem can be addressed by solving the following optimization,
 \begin{eqnarray*}
\min\limits_{\x\in{\R^{n}}}~~ \frac{1}{2}\|\textbf{A}\x-\textbf{b}\|_q^{q},~~~
{\rm s.t.}~~   \x\in \{0,1\}^{n}.
\end{eqnarray*} 
\subsubsection{Data generation}
 We generate $({\textbf A}, {\textbf b}, \x^*)$ as follows. Entries of ${\textbf A}$ are  independently and identically distributed (IID) from ${\cal N}(0,1)$, a normal distribution with mean 0 and variance $1$, and then let ${{\textbf A}={\textbf A}/c}$, where ${c=\sqrt{m}}$ if ${n \leqslant 10^4}$ and ${c=1}$ otherwise. Ground truth signal $\x^*$ is taken from $\{0,1\}^n$ with randomly picked $s$ indices on which entries are $1$. Finally, let ${{\textbf b}={\textbf A}\x^*+ {\rm nf}\cdot {\boldsymbol \varepsilon}}$, where ${\rm nf}$ is the noise ratio and the $i$th entry of the noise is ${\varepsilon_i\sim\mathcal{N}(0,1)}$. 

 \begin{table}[!t]
	\renewcommand{\arraystretch}{1.05}\addtolength{\tabcolsep}{-3pt}
	\centering
\caption{Comparison of different SPFs for recovery problems. Values are reported as mean (std).}\label{tab:SPF-comp} 
\begin{tabular}{lccccccccc}
\toprule
 & \multicolumn{3}{c}{$q=1.5$} & \multicolumn{3}{c}{$q=2$} & \multicolumn{3}{c}{$q=2.5$} \\
\cmidrule(lr){2-4} \cmidrule(lr){5-7} \cmidrule(lr){8-10}
 & $s=100$ & $s=300$ & $s=500$ & $s=100$ & $s=300$ & $s=500$ & $s=100$ & $s=300$ & $s=500$ \\
\midrule
$\mathrm{SPF}_1$ & 0(0) & 0(0) & 0.12(0.15) & 0(0) & 0.04(0.17) & 0.19(0.37) & 0(0) & 0.24(0.38) & 0.36(0.44) \\
$\mathrm{SPF}_2$ & 0(0) & 0(0) & 0.17(0.15) & 0(0) & 0.00(0.00) & 0.07(0.24) & 0(0) & 0.00(0.00) & 0.34(0.40) \\
$\mathrm{SPF}_3$ & 0(0) & 0(0) & 0.16(0.17) & 0(0) & 0.00(0.00) & 0.03(0.17) & 0(0) & 0.00(0.00) & 0.31(0.44) \\
$\mathrm{SPF}_4$ & 0(0) & 0(0) & 0.28(0.43) & 0(0) & 0.04(0.17) & 0.14(0.33) & 0(0) & 0.16(0.35) & 0.45(0.46) \\
$\mathrm{SPF}_5$ & 0(0) & 0(0) & 0.12(0.18) & 0(0) & 0.00(0.00) & 0.04(0.12) & 0(0) & 0.00(0.00) & 0.52(0.45) \\
$\mathrm{SPF}_6$ & 0(0) & 0(0) & 0.18(0.37) & 0(0) & 0.00(0.00) & 0.06(0.12) & 0(0) & 0.00(0.00) & 0.37(0.44) \\
$\mathrm{SPF}_7$ & 0(0) & 0(0) & 0.21(0.40) & 0(0) & 0.07(0.24) & 0.14(0.33) & 0(0) & 0.09(0.28) & 0.45(0.45) \\
$\mathrm{SPF}_8$ & 0(0) & 0(0) & 0.10(0.17) & 0(0) & 0.00(0.00) & 0.05(0.15) & 0(0) & 0.00(0.00) & 0.26(0.40) \\
$\mathrm{SPF}_9$ & 0(0) & 0(0) & 0.11(0.13) & 0(0) & 0.00(0.00) & 0.03(0.17) & 0(0) & 0.00(0.00) & 0.31(0.40) \\
\bottomrule
\end{tabular}
\end{table}
\subsubsection{Benchmark methods}
The algorithms selected for the comparison include NPGM,  MEPM \cite{Yuan17}, L2ADMM \cite{Wu19}, EMADM \cite{liu23}, and GUROBI. Specifically, NPGM denotes Nesterov’s proximal gradient method applied to the LP relaxation of \eqref{UBIP}, while EMADM is employed to solve its semidefinite relaxation. It should be emphasized that both GUROBI and EMADM are only capable of handling binary quadratic programming instances, i.e., the recovery problem with ${q=2}$. Parameters for all algorithms are configured as follows. 
For MEPM, the parameters are set to ${(\rho,\sigma,T) = (0.01, \sqrt{10}, 10)}$, and for L2ADMM, ${(\alpha,\sigma,T) = (0.01, \sqrt{10}, 10)}$. Both methods are terminated after at most 200 iterations or when the tolerance reaches $0.01$. In addition, each requires a step-size parameter $L$, corresponding to the Lipschitz constant of the gradient, which is set to ${L = \|\mathbf{A}^{\top}\mathbf{A}\|}$.
For EMADM, the parameters are configured as ${(\beta, \lambda_0, \delta) = (0.1, 0.001, 10^{-4})}$, with the maximum number of iterations set to $10^4$ and tolerance to $10^{-4}$.
All algorithms are initialized at ${\x_0 = \mathbf{0}}$ and are restricted to a maximum runtime of 600 seconds. We assess performance by the normalized objective gap, defined by 
\begin{eqnarray*}
	\text{Gap} = \frac{ \text{obj}-\text{lowest} }{1+\text{obj}+\text{lowest}},
\end{eqnarray*} 
where `obj' denotes the objective value obtained by a specific algorithm, and `lowest' refers to the lowest objective value achieved among all compared methods.
Clearly, a smaller gap indicates better recovery performance. We point out that for the recovery problems in the noiseless setting (i.e., ${\rm nf}=0$), optimal solution $ \x^*$ is provided. Therefore, \text{lowest} corresponds to the optimal objective value, and $\text{Gap}=0$ indicates that a globally optimal solution is achieved. In contrast, for the Max-Cut and QUBO problems, a reported $\text{Gap}=0$ only indicates that the method attains the best objective value among the compared methods under the stated reference; it does not imply that the global optimum is known.

\subsubsection{Numerical comparison}
We first demonstrate the effect of different SPFs for ShaPeak by selecting three types of SPF functions: $g(x;\omega,1,1,1,1)$, $g(x;\omega,a,a,2,2)$, and $h(x;\omega,a,a,2,2)$ defined in \eqref{def-g-h} and \eqref{def-h-h}. Setting ${a=2.5}$ and ${\omega \in \{0, 0.5, 1\}}$ yields nine distinct functions. Let $\text{SPF}_1-\text{SPF}_9$ stand for them as $g(x$;0,1,1,1,1$)$, $g(x$;0,2.5,2.5,2,2$)$, $h(x$;0,2.5,2.5,2,2$)$, $g(x$;1,1,1,1,1$)$, $g(x$;1,2.5,2.5,2,2$)$, $h(x$;1,2.5,2.5,2,2$)$, $g(x$;0.5,1,1,1,1$)$, $g(x$;0.5,2.5,2.5,2,2$)$, and $h(x$;0.5,2.5,2.5,2,2 $)$, respectively.  The mean and standard deviation (std) over 50 independent trials are reported in Table \ref{tab:SPF-comp},  where  $(m,n,\mathrm{nf}) = (500,1000,0)$. As shown in the table, in most cases, the last two functions $\text{SPF}_8$ and $\text{SPF}_9$ defined by
\begin{eqnarray}\label{particular-gh}
&&g(x;0.5,2.5,2.5,2,2)=
\begin{cases} { (2x+5)^2}/{8}- {25}/{8},&x \leqslant  {1}/{2},\\ 
 { (2x-7)^2}/{8}- {25}/{8},&x>  {1}/{2},\\
\end{cases}  \\ 
&& h(x;0.5,2.5,2.5,2,2)=
\begin{cases} {25}/{8}- {(2x-5)^2}/{8},&x \leqslant  {1}/{2},\\ 
 {25}/{8}- {(2x+3)^2}/{8},&x>  {1}/{2},
\end{cases}
\end{eqnarray}  
consistently achieve the highest accuracy.  Therefore, we select them in the subsequent numerical experiment, and denote the resulting algorithms as ShaPeakg and ShaPeakh.
 \begin{table}[!t]
\renewcommand{\arraystretch}{1.05}\addtolength{\tabcolsep}{-1pt}
\centering
\caption{Effect of $m$, $s$, nf, and $q$ for recovery problems. Values are presented as mean (std) [best].}\label{tab:recovery-msnfq}
\begin{tabular}{lccccc} \toprule
Algs. & $m=300$ & $m=350$ & $m=400$ & $m=450$ & $m=500$ \\
\midrule
ShaPeakg & 0.00(0.00)[0.00] & 0.00(0.00)[0.00] & 0.00(0.00)[0.00] & 0.00(0.00)[0.00] & 0.00(0.00)[0.00] \\
ShaPeakh & 0.00(0.00)[0.00] & 0.00(0.00)[0.00] & 0.00(0.00)[0.00] & 0.00(0.00)[0.00] & 0.00(0.00)[0.00] \\
MEPM & 0.00(0.00)[0.00] & 0.00(0.00)[0.00] & 0.00(0.00)[0.00] & 0.00(0.00)[0.00] & 0.00(0.00)[0.00] \\
L2ADMM & 0.00(0.00)[0.00] & 0.00(0.00)[0.00] & 0.00(0.00)[0.00] & 0.00(0.00)[0.00] & 0.00(0.00)[0.00] \\
NPGM & 0.89(0.03)[0.84] & 0.82(0.06)[0.73] & 0.86(0.13)[0.65] & 0.83(0.10)[0.70] & 0.28(0.32)[0.00] \\
GUROBI & 0.41(0.43)[0.00] & 0.27(0.35)[0.00] & 0.00(0.00)[0.00] & 0.00(0.00)[0.00] & 0.00(0.00)[0.00] \\
EMADM & 0.84(0.01)[0.83] & 0.42(0.30)[0.16] & 0.00(0.00)[0.00] & 0.00(0.00)[0.00] & 0.00(0.00)[0.00] \\
\midrule
 & $s=100$ & $s=200$ & $s=300$ & $s=400$ & $s=500$ \\
\midrule
ShaPeakg & 0.00(0.00)[0.00] & 0.00(0.00)[0.00] & 0.00(0.00)[0.00] & 0.00(0.00)[0.00] & 0.09(0.28)[0.00] \\
ShaPeakh & 0.00(0.00)[0.00] & 0.00(0.00)[0.00] & 0.00(0.00)[0.00] & 0.00(0.00)[0.00] & 0.09(0.28)[0.00] \\
MEPM & 0.00(0.00)[0.00] & 0.00(0.00)[0.00] & 0.00(0.00)[0.00] & 0.00(0.00)[0.00] & 0.80(0.28)[0.00] \\
L2ADMM & 0.00(0.00)[0.00] & 0.00(0.00)[0.00] & 0.00(0.00)[0.00] & 0.00(0.00)[0.00] & 0.26(0.42)[0.00] \\
NPGM & 0.31(0.24)[0.00] & 0.46(0.49)[0.00] & 0.57(0.49)[0.00] & 0.57(0.49)[0.00] & 0.62(0.43)[0.00] \\
GUROBI & 0.00(0.00)[0.00] & 0.00(0.00)[0.00] & 0.00(0.00)[0.00] & 0.35(0.45)[0.00] & 0.53(0.41)[0.00] \\
EMADM & 0.00(0.00)[0.00] & 0.00(0.00)[0.00] & 0.09(0.29)[0.00] & 0.36(0.47)[0.00] & 0.46(0.48)[0.00] \\
\midrule
  & $\mathrm{nf}=0.00$ & $\mathrm{nf}=0.02$ & $\mathrm{nf}=0.04$ & $\mathrm{nf}=0.06$ & $\mathrm{nf}=0.08$ \\
\midrule
ShaPeakg & 0.00(0.00)[0.00] & 0.00(0.00)[0.00] & 0.00(0.00)[0.00] & 0.00(0.00)[0.00] & 0.00(0.00)[0.00] \\
ShaPeakh & 0.00(0.00)[0.00] & 0.00(0.00)[0.00] & 0.00(0.00)[0.00] & 0.00(0.00)[0.00] & 0.00(0.00)[0.00] \\
MEPM & 0.00(0.00)[0.00] & 0.00(0.00)[0.00] & 0.00(0.00)[0.00] & 0.00(0.00)[0.00] & 0.00(0.00)[0.00] \\
L2ADMM & 0.00(0.00)[0.00] & 0.00(0.00)[0.00] & 0.00(0.00)[0.00] & 0.00(0.00)[0.00] & 0.01(0.00)[0.01] \\
NPGM & 0.57(0.49)[0.00] & 0.71(0.40)[0.00] & 0.64(0.36)[0.01] & 0.86(0.06)[0.69] & 0.81(0.03)[0.75] \\
GUROBI & 0.00(0.00)[0.00] & 0.00(0.00)[0.00] & 0.13(0.26)[0.00] & 0.22(0.34)[0.00] & 0.37(0.20)[0.00] \\
EMADM & 0.00(0.00)[0.00] & 0.00(0.00)[0.00] & 0.02(0.02)[0.00] & 0.24(0.35)[0.00] & 0.44(0.26)[0.00] \\
\midrule
& $q=1.50$ & $q=1.75$ & $q=2.00$ & $q=2.25$ & $q=2.50$ \\
\midrule
ShaPeakg & 0.00(0.00)[0.00] & 0.00(0.01)[0.00] & 0.09(0.27)[0.00] & 0.08(0.26)[0.00] & 0.16(0.21)[0.00] \\
ShaPeakh & 0.03(0.08)[0.00] & 0.00(0.00)[0.00] & 0.00(0.00)[0.00] & 0.00(0.00)[0.00] & 0.10(0.16)[0.00] \\
MEPM & 0.23(0.40)[0.00] & 0.19(0.39)[0.00] & 0.80(0.28)[0.00] & 0.84(0.01)[0.84] & 0.52(0.36)[0.03] \\
L2ADMM & 0.13(0.31)[0.00] & 0.01(0.02)[0.00] & 0.45(0.47)[0.00] & 0.34(0.43)[0.00] & 0.50(0.37)[0.00] \\
NPGM & 0.83(0.31)[0.00] & 0.14(0.33)[0.00] & 0.67(0.46)[0.00] & 0.37(0.48)[0.00] & 0.54(0.40)[0.00] \\
\bottomrule
\end{tabular}
\end{table}

\begin{table}[!th]
\renewcommand{\arraystretch}{1.1}\addtolength{\tabcolsep}{4pt}
\centering
\caption{CPU time  for recovery problems with higher $n$. Values are presented as mean (std) [best].}\label{tab:highdim-nf}
\begin{tabular}{llrrrr}
\toprule
&Algs. & $n=10^4$ & $n=10^5$ & $n=10^6$ \\
\midrule
\multicolumn{5}{c}{{\rm nf} = 0.0}\\
\midrule
& ShaPeakg & 0.25(0.03)[0.22] & 15.86(3.25)[12.30] & 61.53(6.26)[51.64] \\
& ShaPeakh & 0.25(0.03)[0.22] & 16.76(4.68)[11.90] & 64.53(8.01)[49.79] \\
$q=1.5$ & MEPM & 26.80(0.09)[26.71] & 930.0(28.75)[894.7] & 2212.9(49.93)[2144.7] \\
& L2ADMM & 51.14(0.15)[50.92] & 1764.8(68.58)[1659.0] & 3612.1(16.26)[3601.4] \\
& NPGM & 14.91(0.23)[14.51] & 588.8(14.06)[573.1] & 1326.0(50.80)[1222.7] \\
\midrule
& ShaPeakg & 0.26(0.06)[0.21] & 49.76(3.38)[45.81] & 74.44(39.93)[32.94] \\
& ShaPeakh & 0.25(0.05)[0.21] & 53.10(4.72)[47.10] & 76.85(39.48)[31.70] \\
$q=2$ & MEPM & 20.64(0.28)[20.28] & 672.6(32.75)[625.0] & 698.0(37.44)[651.0] \\
& L2ADMM & 30.22(0.16)[29.79] & 1057.2(48.14)[997.7] & 1721.3(58.94)[1634.9] \\
& NPGM & 13.14(0.24)[12.74] & 587.3(19.06)[564.2] & 1187.5(55.92)[1085.8] \\
\midrule
& ShaPeakg & 0.52(0.05)[0.44] & 49.91(4.43)[45.49] & 110.1(40.44)[54.58] \\
& ShaPeakh & 0.52(0.05)[0.42] & 48.26(3.27)[43.68] & 110.2(39.33)[60.47] \\
$q=2.5$ & MEPM & 15.19(0.18)[14.83] & 733.6(36.04)[679.4] & 925.4(28.17)[860.0] \\
& L2ADMM & 19.43(0.26)[19.11] & 1122.3(42.67)[1053.6] & 2051.1(53.34)[1974.5] \\
& NPGM & 15.30(0.08)[15.20] & 592.7(18.60)[566.1] & 1408.3(55.67)[1301.3] \\
\midrule
\multicolumn{5}{c}{{\rm nf} = 0.1}\\
\midrule
& ShaPeakg & 6.37(7.93)[0.22] & 23.33(4.28)[12.57] & 74.44(3.01)[70.16] \\
& ShaPeakh & 6.35(7.93)[0.23] & 23.16(4.97)[12.27] & 73.70(5.65)[64.50] \\
$q=1.5$ & MEPM & 31.65(0.32)[31.37] & 1083.2(18.27)[1055.6] & 2541.7(68.08)[2458.9] \\
& L2ADMM & 54.26(0.31)[53.72] & 1971.2(41.36)[1878.5] & 3627.0(16.45)[3607.5] \\
& NPGM & 15.06(0.17)[14.76] & 615.5(25.52)[560.7] & 1401.0(58.43)[1312.7] \\
\midrule
& ShaPeakg & 0.26(0.05)[0.22] & 12.22(1.66)[9.55] & 66.33(3.35)[61.55] \\
& ShaPeakh & 0.26(0.04)[0.21] & 12.03(2.08)[9.18] & 65.15(2.23)[62.67] \\
$q=2$ & MEPM & 22.35(0.15)[22.16] & 820.2(20.16)[780.8] & 1361.4(38.73)[1303.2] \\
& L2ADMM & 34.80(0.27)[34.44] & 1523.7(25.26)[1467.4] & 3389.9(37.79)[3348.7] \\
& NPGM & 13.43(0.17)[13.28] & 613.3(18.43)[587.3] & 1223.2(23.37)[1194.6] \\
\midrule
& ShaPeakg & 0.55(0.08)[0.45] & 23.09(1.90)[20.09] & 53.14(3.54)[47.44] \\
& ShaPeakh & 0.54(0.07)[0.44] & 23.15(2.40)[20.08] & 52.84(2.14)[50.69] \\
$q=2.5$ & MEPM & 19.77(0.16)[19.54] & 777.1(51.53)[693.6] & 1337.1(51.97)[1265.9] \\
& L2ADMM & 28.51(0.17)[28.31] & 1373.1(86.06)[1231.8] & 3590.2(28.67)[3539.1] \\
& NPGM & 15.52(0.09)[15.34] & 599.9(39.06)[549.1] & 1373.7(41.64)[1332.7] \\
\bottomrule
\end{tabular}
\end{table}

{\bf a) Effect of $m$.} We fix ${(n,s,q,{\rm nf})=(1000,100,2,0)}$ and vary ${m\in\{300,350,\ldots,500\}}$. For each case of $(m,n,s,q,{\rm nf})$, we run 20 independent trials and report the best value, as well as the mean and standard deviation of the results. As shown in Table \ref{tab:recovery-msnfq}, the objective gap generally decreases as $m$ increases, indicating improved recovery performance. Once $m\geqslant 400$, all algorithms except NPGM achieve zero gap.

{\bf b) Effect of $s$.} 
We fix ${(m,n,q,\mathrm{nf}) = (500,1000,2,0)}$ and vary ${s \in \{100,200,\ldots,500\}}$. The best values, together with the mean results and the corresponding standard deviation over 20 independent trials, are reported in Table \ref{tab:recovery-msnfq}. As $s$ increases, recovering the ground-truth signal becomes progressively more challenging. It can be seen that ShaPeakg and ShaPeakh consistently attain the lowest gap across all settings. Moreover, they achieve exact recovery, indicated by zero gap, in all cases except those with $s=500$.

{\bf c) Effect of $\text{\rm nf}$.} We fix ${(m,n,s,q)=(500,1000,300,2)}$ but increase ${\rm nf}$ over range $\{0,0.02,\ldots,$ $0.08\}$.  The best values, together with the mean results and the corresponding standard deviation over 20 independent trials, are reported in Table \ref{tab:recovery-msnfq}. ShaPeakg, ShaPeakh, and MEPM consistently attain the lowest gap across all settings, indicating stronger robustness to noise.
 
{\bf d) Effect of $q$.} We fix ${(m,n,s,\rm{nf})=(500,1000,500,0)}$ but increase $q$ over range $\{1.5,1.75,\ldots,2.5\}$. The best values, together with the mean results and standard deviation over 20 trials, are reported in Table \ref{tab:recovery-msnfq}. As mentioned earlier, GUROBI and EMADM cannot handle non-quadratic cases (namely, $q \ne 2$) and are therefore excluded from this experiment. Again, ShaPeakg and Shapeakh maintain the lowest gap for all cases.

{\bf e) Effect of higher dimensions.} In this experiment, we exclude GUROBI and EMADM for the comparison due to their dramatically increased computational time when $n \geqslant 10^4$. By fixing $(m, s) = (0.5n, 0.01n)$, we choose ${n\in\{10^4, 10^5, 10^6\}}$ and $\rm{nf}\in\{0,0.1\}$. Moreover,  sparse matrices $\textbf{A}$ with $10^8$ nonzero entries are used when ${n \geqslant 10^5}$, resulting in density $10^9/(5n^2)$. 
Since all mean, standard deviation, and best values of the gap obtained by ShaPeak, MEPM, L2ADMM, and NPGM over 10 independent trials are identical (to zero), these results have been omitted. For running time reported in Table~\ref{tab:highdim-nf}, ShaPeak consistently outperforms the other three algorithms.

\subsection{Classical MIMO detection}
The task of the classical MIMO detection is to recover a signal $\w^*$ from a linear model, $\y=\H\w^*+\boldsymbol{\varepsilon}$, where $\H\in\C^{m\times n/2}$ is the channel matrix,  $\y\in \C^m$ is the received signal, and $\boldsymbol{\varepsilon}\in \C^m $ is the Gaussian noise whose entries are IID from ${\cal N}(0,\varrho^2)$. Let ground truth signal $\w^*\in \Omega:=\{\w\in\C^{n/2}: {\rm Re}(\w)\in\{0,1\}^{n/2},{\rm Im}(\w)\in\{0,1\}^{n/2}\}$, where ${\rm Re}(\w)$ and ${\rm Im}(\w)$ represent the real and imagery part of $\w$. The classical MIMO detection can be addressed by maximizing the likelihood as follows,
 \begin{eqnarray*}
 \min\limits_{\x\in{\R^{n}}}~~ \frac{1}{2}\|\textbf{A}\x-\textbf{b}\|^2,~~~
{\rm s.t.}~~   \x\in \{0,1\}^{n},
\end{eqnarray*}
where
\begin{eqnarray}\label{def-A-b-x}
\textbf{A}:= \left[
 \begin{array}{lr}
 {\rm Re}(\H)& -{\rm Im}(\H)\\
 {\rm Im}(\H)& {\rm Re}(\H)\\
 \end{array}
 \right],\qquad \textbf{b}:= \left[
 \begin{array}{ll}
 {\rm Re}(\y)\\
 {\rm Im}(\y)\\
 \end{array}
 \right],\qquad \x:= \left[
 \begin{array}{ll}
 {\rm Re}(\w)\\
 {\rm Im}(\w)\\
 \end{array}
 \right].
\end{eqnarray}  
   \begin{figure}[!t]
 	\centering
 	
 	\begin{subfigure}[b]{0.49\textwidth}
 		\centering
 		\includegraphics[width=0.99\textwidth]{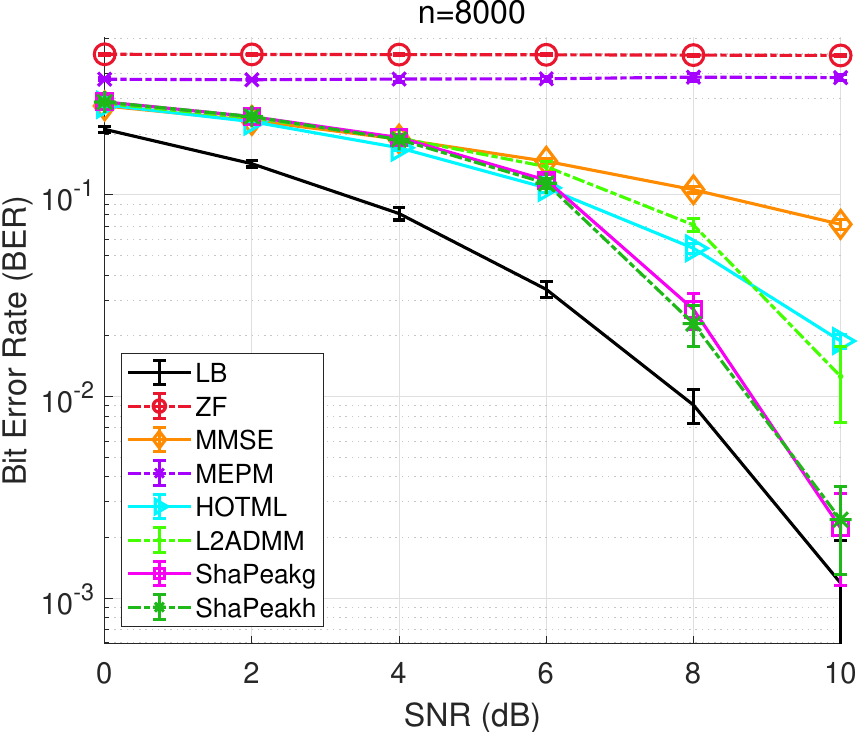}
 	\end{subfigure}~ 
 	\begin{subfigure}[b]{0.49\textwidth}
 		\centering
 		\includegraphics[width=0.99\textwidth]{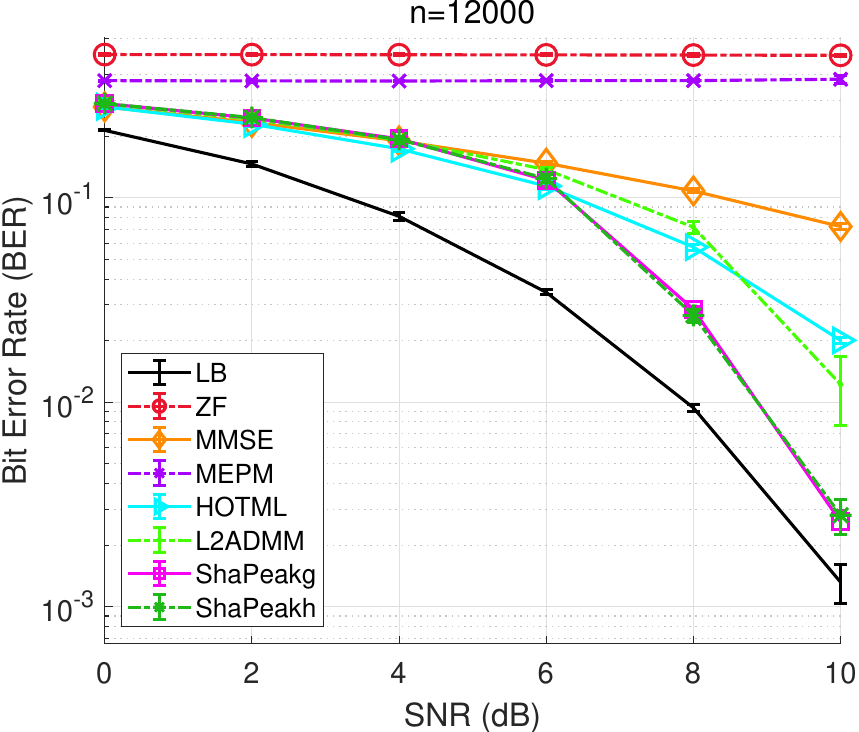}
 	\end{subfigure}
 	
 	\vspace{0.4cm}
 	
 	\begin{subfigure}[b]{0.49\textwidth}
 		\centering
 		\includegraphics[width=0.99\textwidth]{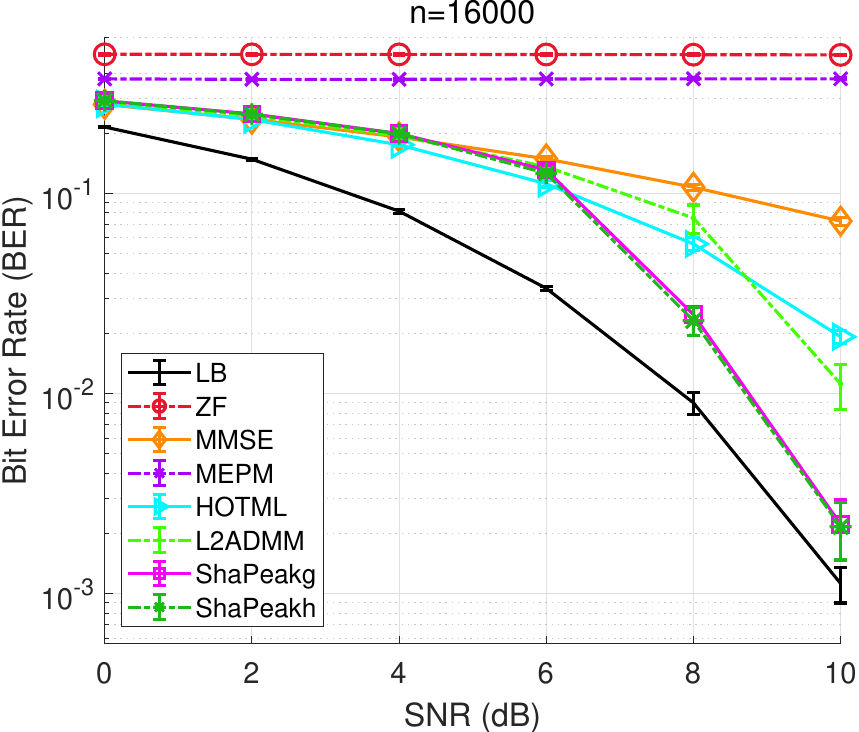} 
 	\end{subfigure}~ 
 	\begin{subfigure}[b]{0.49\textwidth}
 		\centering
 		\includegraphics[width=0.99\textwidth]{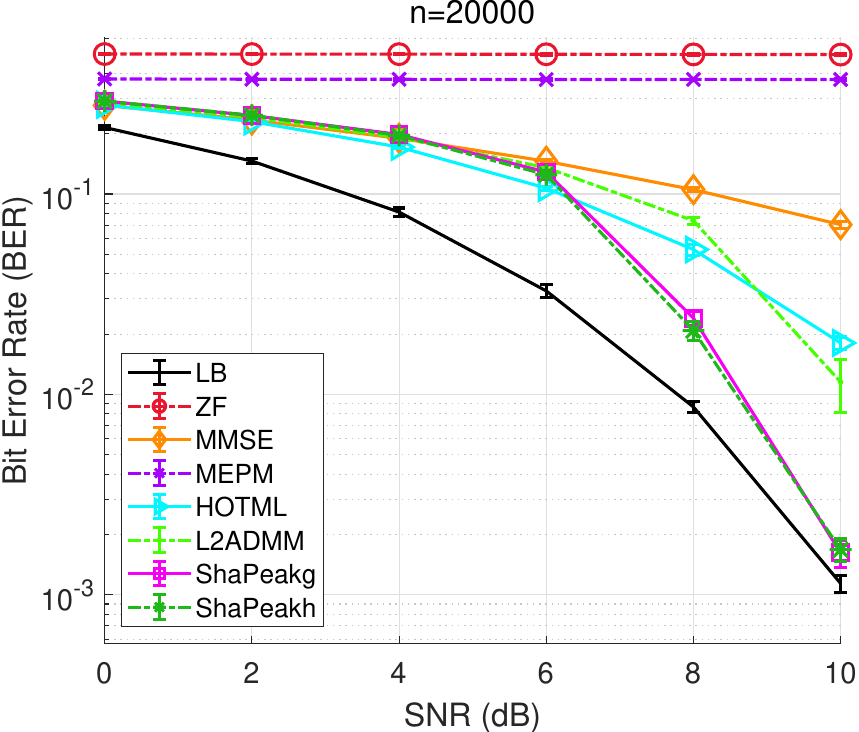}
 	\end{subfigure}
 	
 	\caption{Results for classical MIMO detection problems using IID channels.}
 	\label{fig:BER-iid}
 \end{figure}
 
  \begin{figure}[!t]
 	\centering
 	
 	\begin{subfigure}[b]{0.49\textwidth}
 		\centering
 		\includegraphics[width=0.99\textwidth]{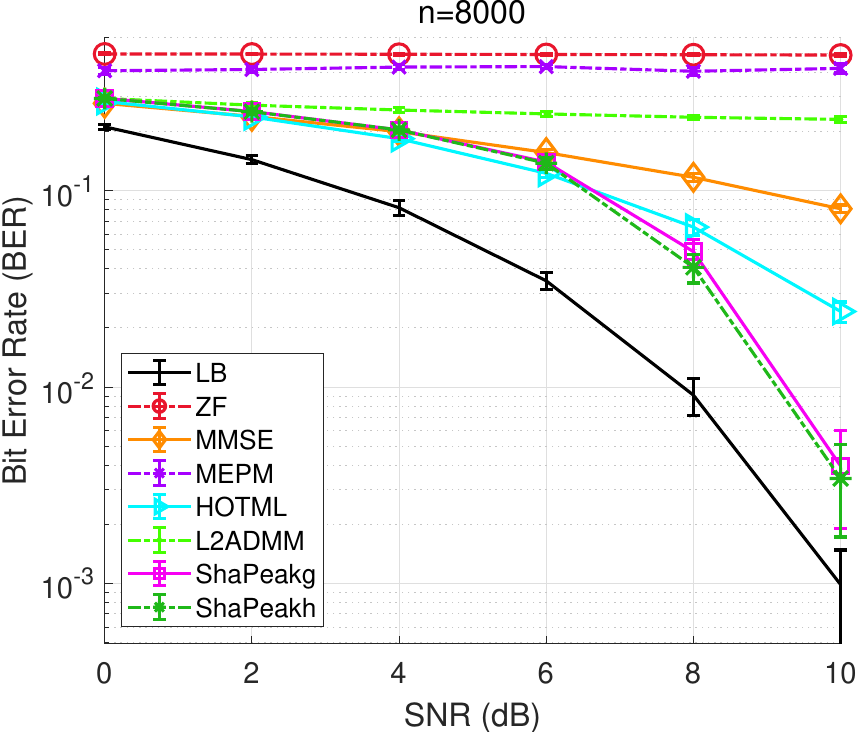}
 	\end{subfigure}~
 	\begin{subfigure}[b]{0.49\textwidth}
 		\centering
 		\includegraphics[width=0.99\textwidth]{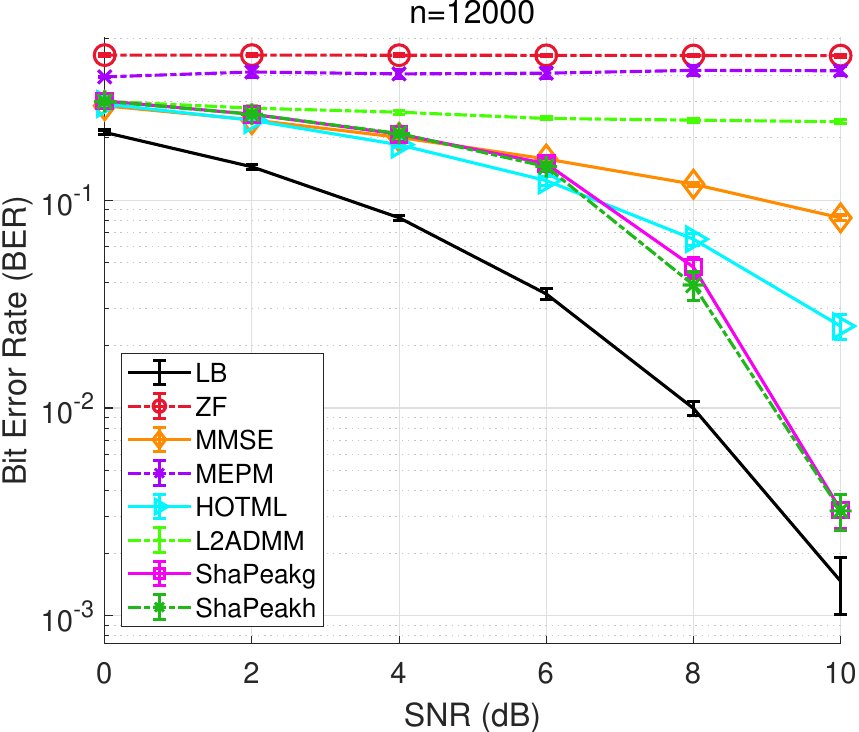}
 	\end{subfigure}
 	
 	\vspace{0.4cm}
 	
 	\begin{subfigure}[b]{0.49\textwidth}
 		\centering
 		\includegraphics[width=0.99\textwidth]{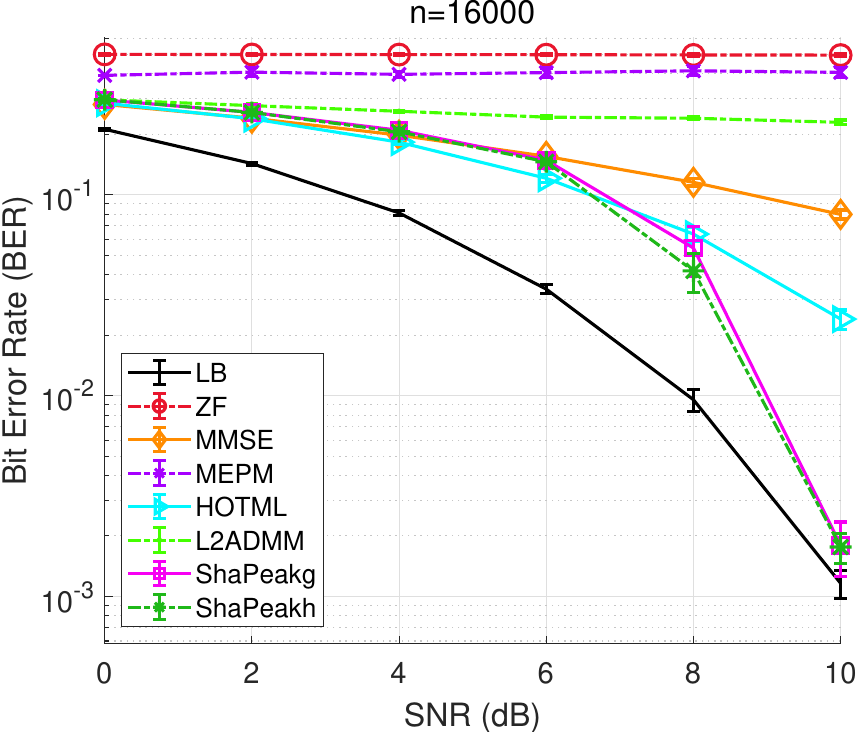} 
 	\end{subfigure}~
 	\begin{subfigure}[b]{0.49\textwidth}
 		\centering
 		\includegraphics[width=0.99\textwidth]{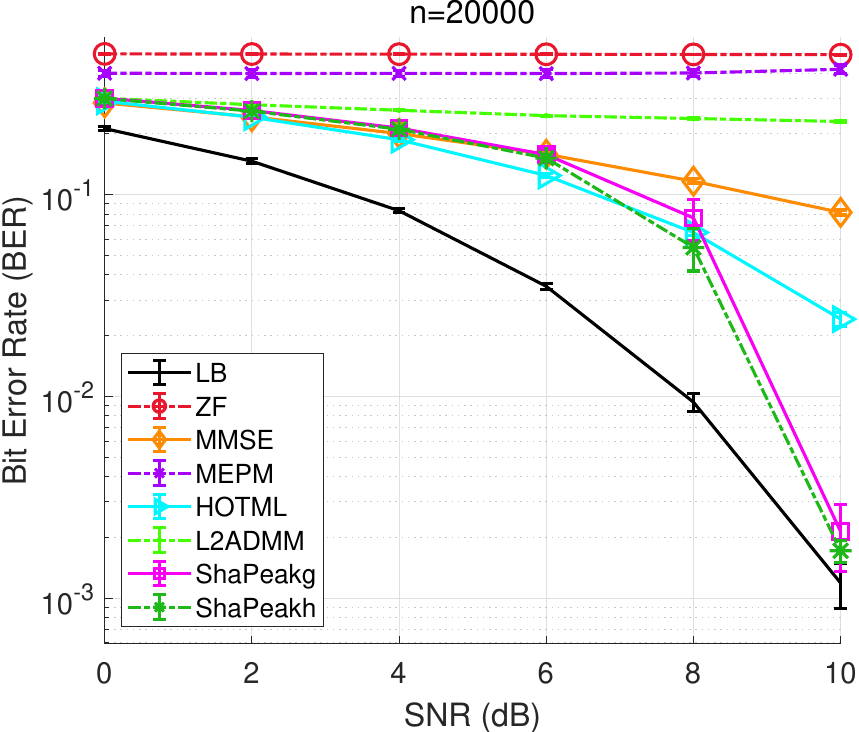}
 	\end{subfigure}
 	
 	\caption{Results for classical MIMO detection problems using correlated channels.} 
 	\label{fig:BER-cor}
 \end{figure}
 
\subsubsection{Data generation} \label{sec:cmimo-data-gen}
We consider two types of channel matrices $\textbf{H}$. The first corresponds to the IID Gaussian channel case, where the entries of ${\rm Re}(\textbf{H})$ and ${\rm Im}(\textbf{H})$ are generated in the same manner as the entries of $\textbf{A}$ in Section~\ref{sec:rec}. The second corresponds to the correlated MIMO channel case, following a similar approach to \cite{loyka2001channel,shao2020binary}. Specifically, let $\widetilde{\textbf{H}}$ be an element-wise IID circularly symmetric complex Gaussian matrix with zero mean and unit variance. Let $\textbf{R}$ and $\textbf{T}$ denote the spatial correlation matrices at the receiver and transmitter, respectively, with each entry $R_{ij}$ of $\textbf{R}$ defined as
$$
R_{ij}= \begin{cases}r^{i-j}, & \text { if } i \leqslant j, \\ {R_{ij}^*,} & \text { if } i > j,\end{cases}
$$
where ${|r| \leqslant 1}$.  The matrix $\textbf{T}$ is generated in the same way. By performing Cholesky-like decompositions ${\textbf{R} = \textbf{P}\textbf{P}^*}$ and ${\textbf{T} = \textbf{Q}\textbf{Q}^*}$, the correlated channel matrix is constructed as $
\textbf{H}=\textbf{P} \widetilde{\textbf{H}} \textbf{Q}
$. In the sequel, we set $r = 0.2$ and $\textbf{R} = \textbf{T}$. Transmit signal $\textbf{w}^*$  is generated element-wise from a uniform distribution over the quaternary phase-shift keying constellation; see \cite{shao2020binary} for more details. The signal-to-noise ratio (SNR), which is used to determine the noise variance $\varrho^2$, is defined as
$$
\mathrm{SNR}=\mathbb{E}\left\|\textbf{H} \textbf{w}^*\right\|^2 / \mathbb{E}\left\|\boldsymbol{\varepsilon}\right\|^2.
$$

\subsubsection{Benchmark methods} We select MEPM, L2ADMM, and HOTML \cite{shao2020binary} for the comparison and also report the results directly calculated by two detectors: zero-forcing detector (ZF) and minimum-mean-square-error (MMSE) decision-feedback detector. Moreover, we assume no inter-signal interference and establish a lower bound (LB) for the problem, as described in \cite{shao2020binary}. Parameters for all algorithms are given as follows. 
For MEPM and L2ADMM,  parameters are kept the same as those in Section \ref{sec:rec} but set $L=\|\mathbf{A}\|_F^2/n$.
 For HOTML, set $\sigma_0=0.5$ and $\lambda_0=0.001$. It terminates when either the number of iterations $k>200$ or $\|\lambda_k-\lambda_{k-1}\|\le 10^{-4}$ is satisfied. Once again, for fairness, all algorithms are initialized with $\x^0={\bf 0}$. 
 We evaluate performance using the bit error rate (BER)  defined by
 $$
 \text{BER} = \frac{|\{i\in[n]: x_i \neq x^*_i\}|}{n}, 
 $$
where $|\Omega|$ is the cardinality of set $\Omega$. Note that a lower BER indicates more accurate recovery.

\subsubsection{Numerical comparison}  We consider the case of critically determined channels (i.e., ${m = n/2}$) and choose $n\in\{8000$, $12000$, $16000$, $20000\}$. Figures~\ref{fig:BER-iid} and~\ref{fig:BER-cor} present the numerical results for the IID~and correlated channels, respectively. All reported results are averaged over 20 independent trials, and the error bars represent the corresponding standard deviation. For the IID~channel case, ShaPeak achieves the best BER performance, followed by L2ADMM and HOTML, while ZF performs the worst. In the correlated channel setting, ShaPeak continues to attain the lowest BER across most SNR levels, followed by HOTML, whereas L2ADMM fails under this configuration.  The corresponding mean and the standard deviation of the computational time are reported in Table~\ref{tab:MIMO-time} and Table \ref{tab:MIMO-time-std} (in Appendix A), respectively. ShaPeak demonstrates a significant advantage in all scenarios. For example, in the IID~channel with SNR = 10 and $n=20000$, ShaPeakg and ShaPeakh complete in 6.045 and 5.934 seconds, respectively, while the other three methods require over 200 seconds.

 \begin{table}[!t]
	\renewcommand{\arraystretch}{1.00}\addtolength{\tabcolsep}{2.25pt}
	\centering
	\caption{Average CPU time (in seconds) for classical MIMO detection, where A1-A5 stand for  ShaPeakg,  ShaPeakh,   MEPM,   L2ADMM, and HOTML, respectively.}
 \begin{tabular}{c ccccc c ccccc}
 			\hline
 			    &  \multicolumn{5}{c}{IID channels} && \multicolumn{5}{c}{Correlated channels} \\
 			\cline{2-6} \cline{8-12}
 			SNR & A1 & A2 & A3 & A4 & A5 && A1 & A2 & A3 & A4 & A5  \\
 			\hline
 			&\multicolumn{11}{c}{$n=8000$}\\\hline
 			 0 & 1.345 & 1.227 & 28.06 & 37.60 & 28.25 && 1.128 & 1.122 & 28.42 & 36.60 & 28.12 \\
			 2 & 1.441 & 1.393 & 28.09 & 38.18 & 26.08 && 1.185 & 1.104 & 27.70 & 36.72 & 25.85 \\
			 4 & 1.438 & 1.463 & 28.17 & 38.22 & 26.12 && 1.264 & 1.039 & 28.42 & 34.68 & 25.86 \\
			 6 & 1.536 & 1.266 & 28.04 & 38.38 & 26.18 && 1.268 & 1.174 & 27.29 & 34.46 & 25.87 \\
			 8 & 1.406 & 1.355 & 27.94 & 38.47 & 26.40 && 1.245 & 1.219 & 27.99 & 34.39 & 25.85 \\
			 10 & 1.129 & 1.123 & 27.98 & 38.38 & 26.12 && 0.930 & 0.916 & 27.93 & 34.44 & 25.84 \\
			 
\hline
	&\multicolumn{11}{c}{$n=12000$}\\\hline
			 0 & 2.324 & 2.051 & 67.15 & 88.25 & 103.8 && 2.901 & 2.844 & 68.21 & 93.14 & 104.8 \\
			 2 & 2.669 & 2.190 & 68.48 & 92.39 & 103.4 && 2.903 & 2.844 & 70.58 & 91.65 & 101.5 \\
			 4 & 2.695 & 2.373 & 68.74 & 92.60 & 103.3 && 2.910 & 2.854 & 70.24 & 86.82 & 101.5 \\
			 6 & 2.298 & 2.069 & 68.24 & 92.21 & 103.4 && 2.867 & 2.864 & 69.64 & 84.01 & 101.5 \\
			 8 & 2.169 & 2.014 & 68.21 & 92.66 & 103.8 && 2.891 & 2.864 & 70.35 & 84.27 & 101.5 \\
			 10 & 1.767 & 1.728 & 68.33 & 92.45 & 103.7 && 2.059 & 2.026 & 68.53 & 85.37 & 101.6 \\
			\hline
				&\multicolumn{11}{c}{$n=16000$}\\\hline
			 0 & 5.668 & 5.625 & 133.9 & 181.5 & 254.5 && 5.652 & 5.614 & 128.4 & 177.5 & 256.9 \\
			 2 & 5.661 & 5.619 & 134.1 & 180.7 & 251.6 && 5.641 & 5.612 & 131.9 & 177.5 & 261.9 \\
			 4 & 5.643 & 5.554 & 134.2 & 180.9 & 251.3 && 5.585 & 5.558 & 132.4 & 157.5 & 261.9 \\
			 6 & 5.675 & 5.609 & 134.1 & 181.8 & 251.3 && 5.655 & 5.574 & 129.9 & 157.3 & 261.8 \\
			 8 & 5.668 & 5.609 & 134.5 & 182.5 & 251.7 && 5.754 & 5.551 & 130.7 & 156.9 & 262.0 \\
			 10 & 4.047 & 3.988 & 133.0 & 179.9 & 253.1 && 3.992 & 3.960 & 127.6 & 159.7 & 261.8 \\
			 
			\hline
				&\multicolumn{11}{c}{$n=20000$}\\\hline
			 0 & 8.474 & 8.421 & 204.8 & 254.3 & 493.9 && 8.492 & 9.915 & 199.7 & 268.5 & 495.6 \\
			 2 & 8.706 & 8.570 & 204.4 & 274.7 & 496.4 && 8.349 & 9.326 & 205.4 & 290.1 & 509.4 \\
			 4 & 8.574 & 8.476 & 204.2 & 275.3 & 496.7 && 8.502 & 8.946 & 204.6 & 263.0 & 509.6 \\
			 6 & 8.610 & 8.435 & 205.0 & 275.3 & 496.2 && 8.431 & 8.203 & 202.5 & 246.4 & 498.9 \\
			 8 & 8.507 & 8.435 & 204.4 & 277.8 & 496.5 && 9.344 & 8.581 & 206.2 & 244.0 & 509.1 \\
			 10 & 6.045 & 5.934 & 204.8 & 276.6 & 497.0 && 5.883 & 5.789 & 206.2 & 258.9 & 509.1 \\
			\hline
		\end{tabular} 
	\label{tab:MIMO-time}
\end{table}

\subsection{One-bit MIMO detection}
The task of one-bit MIMO detection \cite{shao2020binary} is to detect $\textbf{z}\in\{-1,1\}^n$ from the following model,
\begin{equation}
	\label{eq:onebit}
	\textbf{b}=\operatorname{sgn}(\textbf{A} \textbf{z}+\textbf{v}),
\end{equation}
where  sgn$(\cdot)$ denotes the element-wise sign function, $\textbf{A} \in \mathbb{R}^{m \times n}$ is a system matrix, and $\textbf{v} \in \mathbb{R}^m$ is the noise whose entries are IID~variable from ${\cal N}(0,\varrho^2)$. Let $\Phi(t):=\int_{-\infty}^t \frac{1}{\sqrt{2 \pi}} e^{-\tau^2 / 2} d \tau$ and $\textbf{a}_i$ denote the $i$ th row of $\textbf{A}$. Under model \eqref{eq:onebit}, the maximum-likelihood detector takes the form
	\begin{equation}
		\label{eq:onebitprob}
		\min\limits_{\z\in{\R^{n}}}~~ -\sum_{i=1}^m \log \Phi\left( {(b_i /{\varrho})\langle\textbf{a}_i, \z\rangle}\right),~~~
		{\rm s.t.}~~   \z\in \{-1,1\}^{n},
	\end{equation}
By substituting $\x=(\mathbf{z}+1)/2$, one can easily rewrite \eqref{eq:onebitprob} in the form of \eqref{UBIP}. 
\subsubsection{Data generation} 
Let $\H\in\mathbb{C}^{m\times n/2}$ and $\boldsymbol{\varepsilon}\in\mathbb{C}^{m}$, where the entries of $\rm{Re}(\H)$ and $\rm{Im}(\H)$ are IID from ${\cal N}(0,1)$, and  the entries of   $\rm{Re}(\boldsymbol{\varepsilon})$  and $\rm{Im}(\boldsymbol{\varepsilon})$ are IID from ${\cal N}(0,\varrho^2)$. Then, let $\w^*$ be generated the same as that of the classical MIMO detection, see Section \ref{sec:cmimo-data-gen}. Finally, by computing $\y=\operatorname{sgn}( \H\w^*+\boldsymbol{\varepsilon})$, data $\textbf{A}$ and $\textbf{b}$ are obtained according to the form in \eqref{def-A-b-x}.
\subsubsection{Benchmark methods} The methods selected for the comparison include ZF, MEPM, L2ADMM, and HOTML. Parameters are set as follows. 
 For MEPM and L2ADMM, we set $L=\|\mathbf{A}\|_F^2/n$. HOTML adopts the same parameter settings as in the classical MIMO case. Its maximum number of iterations is set to 300.

\subsubsection{Numerical comparison}  We first evaluate the performance of all algorithms under different values of $m/n$. Specifically, we fix ${\text{SNR}=15}$, set ${n \in \{500,1000\}}$, and vary ${m/n \in \{1, 1.5, \ldots, 3\}}$. Figure~\ref{fig:BER-1bit-effmn} reports the average results with error bars over 20 independent trials. As expected, a larger ratio results in lower BER, indicating easier recovery. It is evident that ShaPeakg and ShaPeakh consistently achieve the lowest BER across all cases.

	\begin{figure}[!t]
		\centering
		
		\begin{subfigure}[b]{0.49\textwidth}
			\centering
			\includegraphics[width=1\textwidth]{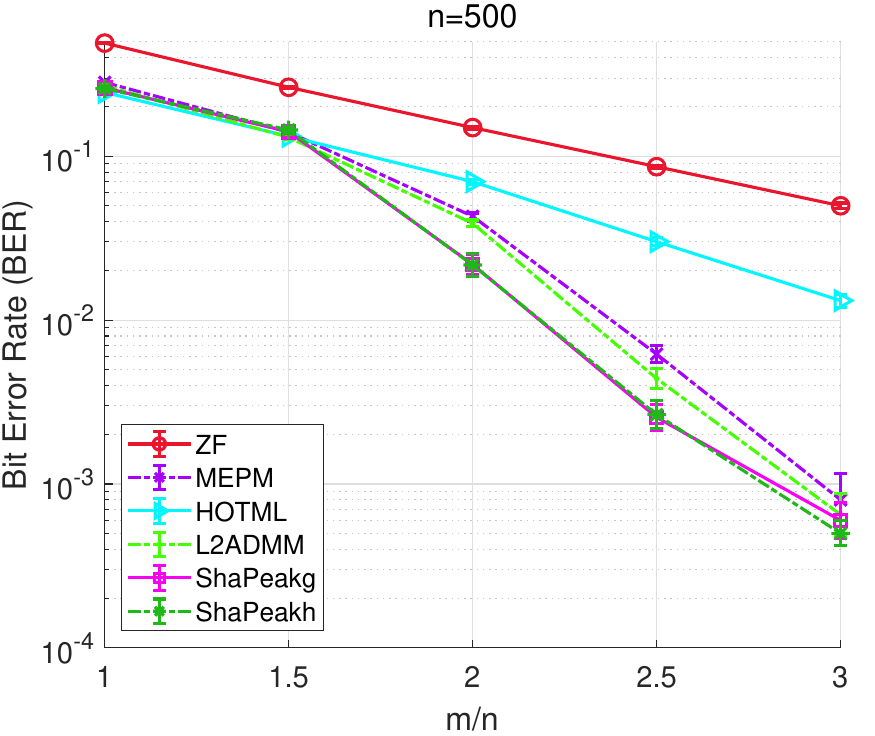}
		\end{subfigure}~~
		\begin{subfigure}[b]{0.49\textwidth}
			\centering
			\includegraphics[width=1\textwidth]{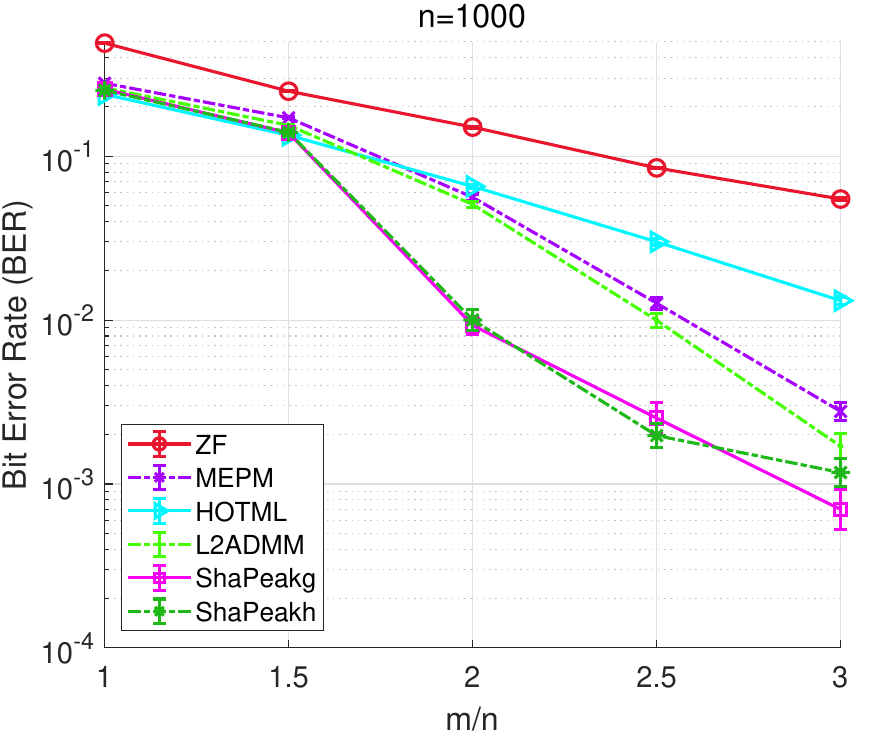}
		\end{subfigure}		
		\caption{Effect of $m/n$ for one-bit MIMO detection problems.}
		\label{fig:BER-1bit-effmn}
	\end{figure}

	\begin{figure}[!t]
		\centering
		
		\begin{subfigure}[b]{0.49\textwidth}
			\centering
			\includegraphics[width=1\textwidth]{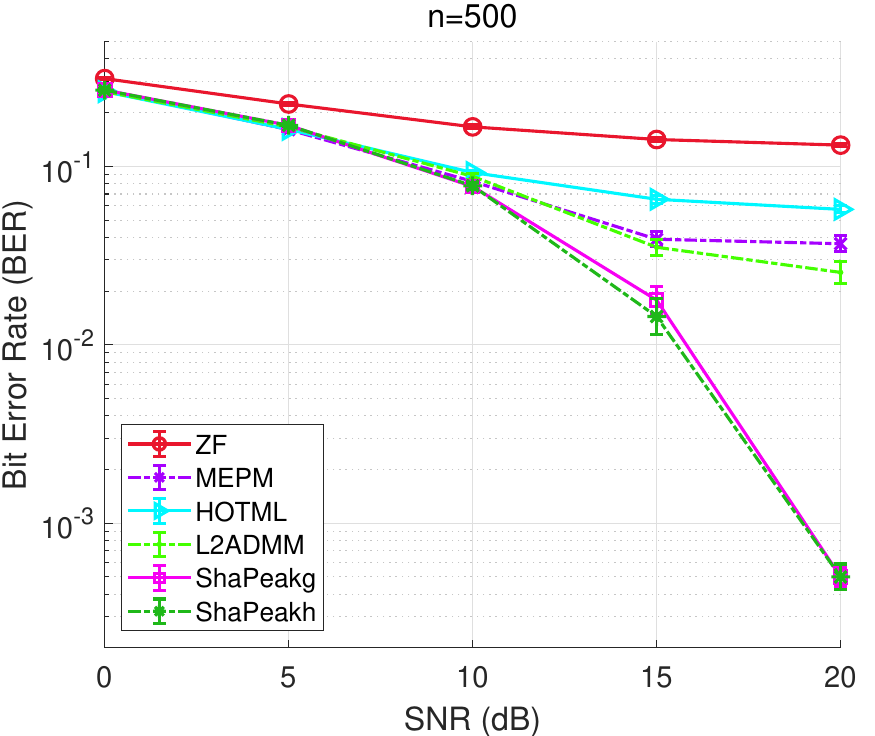}
		\end{subfigure}~~
				\begin{subfigure}[b]{0.49\textwidth}
			\centering
			\includegraphics[width=1\textwidth]{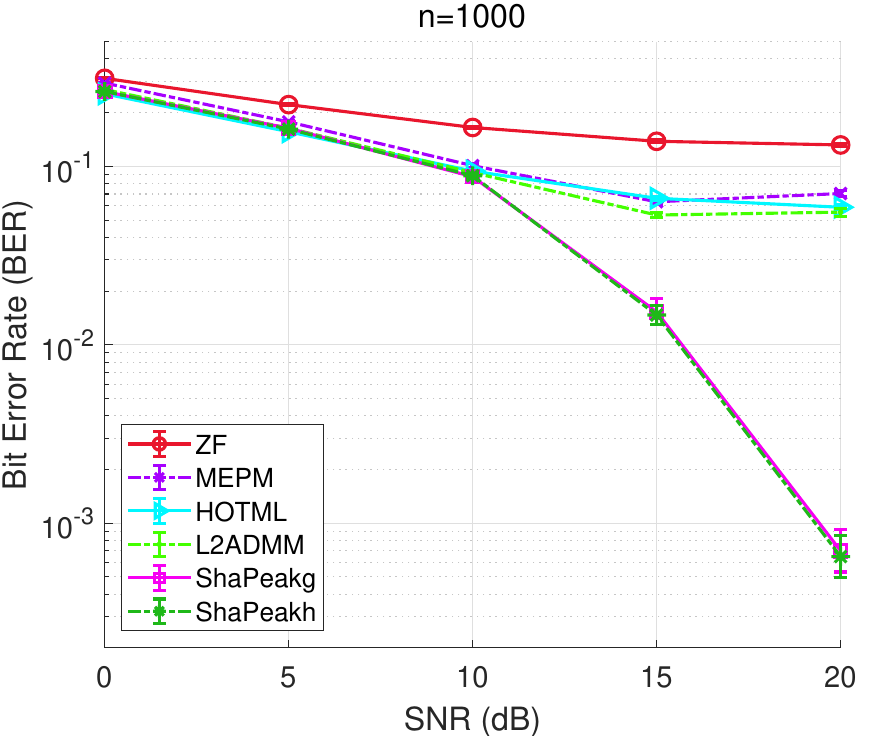}
		\end{subfigure}		
		\caption{Effect of SNR for one-bit MIMO detection problems.}
		\label{fig:BER-1bit}
	\end{figure}

Next, we fix ${m/n = 2}$, choose $n\in\{500,1000\}$, and vary ${\text{SNR} \in \{0, 5, \ldots, 20\}}$. The average results with error bars over 20 runs are presented in Figure~\ref{fig:BER-1bit}. When ${\text{SNR} \leqslant 10}$, none of the algorithms achieve satisfactory BER performance, indicating failure in signal recovery. As SNR increases, ShaPeak begins to outperform the other algorithms, with its advantage becoming more pronounced at higher SNR levels. Regarding computational efficiency, we consider a slightly high dimension ${n=5000}$. As reported in Table~\ref{tab:1bitMIMO-time}, ShaPeak demonstrates significantly lower runtime. For instance, in the case of  ${\text{SNR} = 20}$, ShaPeakg and ShaPeakh complete in 10.87 and 10.86 seconds, respectively, while L2ADMM requires nearly 200 seconds.

 \begin{table}[!t]
		\renewcommand{\arraystretch}{1.0}\addtolength{\tabcolsep}{0.8pt}
		\centering
		\caption{Results for one-bit MIMO detection with  higher dimensional instances, where A1-A6 stand for  ShaPeakg,  ShaPeakh,   MEPM,   L2ADMM, HOTML, and ZF, respectively.}
			\begin{tabular}{ccccccccccccc}
				\hline
			  	&\multicolumn{6}{c}{BER}&&\multicolumn{5}{c}{Time(s) }\\\cline{2-7}\cline{9-13}
			  	 SNR   & A1 & A2 & A3 & A4 & A5 & A6  & & A1 & A2 & A3 & A4 & A5  \\ \hline
			\multicolumn{13}{c}{Mean}\\ \hline
				  0 & 0.273 & 0.271 & 0.474 & 0.441 & 0.257 & 0.312 & & 2.044 & 1.949 & 142.9 & 188.4 & 9.978\\
				  5 & 0.171 & 0.173 & 0.340 & 0.289 & 0.157 & 0.224 & & 4.856 & 4.317 & 143.4 & 196.8 & 12.36    \\
				  10 & 0.092 & 0.088 & 0.219 & 0.177 & 0.095 & 0.167 & & 6.825 & 6.744 & 139.8 & 196.3 & 20.21 \\
				  15 & 0.015 & 0.016 & 0.164 & 0.136 & 0.068 & 0.144 & & 9.245 & 8.937 & 144.0 & 195.7 & 20.35   \\
				  20 & 0.001 & 0.001 & 0.154 & 0.138 & 0.064 & 0.134 & & 10.87 & 10.86 & 143.0 & 197.2 & 20.49  \\
				\hline
					\multicolumn{13}{c}{Std}\\ \hline
						0  & 0.005 & 0.005 & 0.005 & 0.006 & 0.006 & 0.008 && 0.502 & 0.460 & 1.340 & 8.144 & 0.067 \\
		5  & 0.007 & 0.008 & 0.004 & 0.010 & 0.007 & 0.006 && 0.825 & 0.741 & 1.607 & 0.827 & 0.127 \\
		10 & 0.006 & 0.005 & 0.006 & 0.006 & 0.005 & 0.005 && 0.830 & 0.847 & 2.025 & 0.377 & 0.156 \\
		15 & 0.004 & 0.008 & 0.007 & 0.012 & 0.003 & 0.007 && 0.473 & 0.550 & 1.750 & 0.922 & 0.046 \\
		20 & 0.001 & 0.001 & 0.007 & 0.007 & 0.003 & 0.005 && 0.906 & 0.783 & 1.648 & 0.626 & 0.029 \\ \hline
			\end{tabular}
		\label{tab:1bitMIMO-time}
	\end{table}

\subsection{Max-Cut}
The Max-Cut problem aims to partition the vertex set of an undirected weighted graph 
$G=(V,E)$ into two disjoint subsets such that the total weight of the edges crossing the two subsets is maximized. 
Let $x_i\in\{0,1\}$ indicate the side of vertex $i\in V$ and $W_{ij}\ge 0$ denote the weight of edge $(i,j)$. Then the Max-Cut problem can be formulated as
\begin{equation}\label{maxcut}
	\max_{\x\in\{0,1\}^n}\ \sum_{(i,j)\in E} W_{ij}(x_i+x_j-2x_ix_j).
\end{equation}
\subsubsection{Data generation} For this problem, we use the Gset benchmark \footnote{Available at \texttt{https://web.stanford.edu/\~{}yyye/yyye/Gset/}.}, which is a widely recognized collection of weighted graph instances for the Max-Cut problem. It contains 71 benchmark instances, including toroidal, planar, and random graphs, with the number of nodes ranging from 800 to 20,000. In our experiments, we mainly focus on the instances with at least 1,000 nodes, namely, G22–G81.

\subsubsection{Benchmark methods} We compare ShaPeak with four state-of-the-art methods: MCPG \cite{chen2023montecarlopolicygradient}, CRA \cite{ichikawa2024controlling}, and PI-GNN \cite{schuetz2022combinatorial}. MCPG is built upon a Monte Carlo policy gradient framework enhanced with a local search procedure. CRA is an annealing-based method that progressively controls the convexity of the objective function, with graph neural networks (GNNs) incorporated to improve performance. PI-GNN is an unsupervised GNN-based approach for graph partitioning problems. All algorithms are restricted to a maximum runtime of one hour. Since these four baseline methods are implemented in Python, we implement ShaPeak in Python as well using JAX 0.6.1 \cite{jax2018github} to ensure a fair and consistent comparison.

To assess performance, we directly use the reported best-known results from BLS \cite{benlic2013breakout}, a heuristic framework that combines local improvement with adaptive perturbation. 
Let $\mathrm{obj}$ denote the objective value obtained by a given algorithm, and let $\mathrm{obj}_{\mathrm{BLS}}$ denote the objective value reported by BLS. Since most of the best-known results on the Gset dataset are attributed to BLS, we evaluate performance using the relative error with respect to these reference values, which is defined by
\begin{eqnarray*}\label{def-gap}
	\text{Gap} = \frac{\text{obj}-\text{obj}_{\mathrm{BLS}}}{|\text{obj}_{\mathrm{BLS}}|} \times 100\%.
\end{eqnarray*}

\subsubsection{Numerical comparison} Our main results focus on the 12 largest graphs in Gset, namely, G60–G67, G70, G72, G77, and G81.  Each of these instances has more than 7,000 nodes and is highly challenging for existing algorithms. As shown in Table \ref{tab:mc_gset_large}, ShaPeakg achieves the lowest gap on 9 of the 12 graphs, indicating better objective values on these instances. In terms of computational time, ShaPeakg also shows a clear advantage. For example, on G81 with 20,000 nodes, ShaPeak requires only 54.87 seconds, whereas CRA, PI-GNN, and MCPG require 361.1 seconds, 187 seconds, and 1115.7 seconds, respectively. The results on smaller scale graphs in Gset can be found in Table \ref{tab:mc_gset_all} in Appendix B.

\begin{table}[!th]
	\renewcommand{\arraystretch}{1.0}\addtolength{\tabcolsep}{-0.5pt}
	\centering
	\caption{Results for Max-Cut with larger Gset instances.}
	\label{tab:mc_gset_large} 
	\begin{tabular}{lcccccccc}
		\toprule
		& \multicolumn{2}{c}{ShaPeakg} & \multicolumn{2}{c}{CRA} & \multicolumn{2}{c}{PIGNN} & \multicolumn{2}{c}{MCPG} \\
		\cmidrule(lr){2-3}\cmidrule(lr){4-5}\cmidrule(lr){6-7}\cmidrule(lr){8-9}
		Graphs & Gap (\%) & Time (s) & Gap (\%) & Time (s) & Gap (\%) & Time (s) & Gap (\%) & Time (s) \\
		\midrule
		G60 & 0.39 & 78.92 & 2.97 & 314.4 & 12.47 & 412.1 & 0.47 & 393.1 \\
		G61 & 1.12 & 90.20 & 14.49 & 359.6 & 21.07 & 908.5 & 1.28 & 391.9 \\
		G62 & 1.73 & 58.26 & 15.65 & 376.3 & 19.84 & 903.0 & 2.10 & 394.0 \\
		G63 & 3.47 & 52.22 & 3.00 & 288.2 & 11.18 & 568.0 & 0.53 & 399.2 \\
		G64 & 1.37 & 45.97 & 7.19 & 380.9 & 20.48 & 781.3 & 1.10 & 396.6 \\
		G65 & 1.69 & 39.74 & 13.31 & 359.0 & 20.15 & 845.4 & 2.16 & 446.9 \\
		G66 & 2.04 & 42.29 & 13.84 & 368.5 & 19.25 & 935.7 & 2.74 & 518.9 \\
		G67 & 1.73 & 40.29 & 14.47 & 343.1 & 19.94 & 991.0 & 2.48 & 570.3 \\
		G70 & 0.93 & 53.59 & 3.78 & 327.9 & 10.88 & 314.3 & 0.27 & 564.7 \\
		G72 & 1.91 & 60.26 & 13.52 & 393.1 & 19.69 & 868.5 & 2.54 & 576.7 \\
		G77 & 1.83 & 47.28 & 12.15 & 324.9 & 21.44 & 952.4 & 2.50 & 796.8 \\
		G81 & 2.10 & 54.87 & 11.38 & 361.1 & 20.96 & 187.0 & 2.82 & 1115.7 \\
		\bottomrule
	\end{tabular}
\end{table}

\subsection{QUBO}
Let ${\Q\in\R^{n\times n}}$ be a symmetric matrix.	The QUBO problem takes the following form, 
 \begin{eqnarray}\nonumber
\min\limits_{\x\in{\R^n}}~ \frac{1}{2}\langle \x, \Q\x\rangle,~~~
{\rm s.t.}~~ \x\in \{0,1\}^n.
\end{eqnarray}

\subsubsection{Data generation}   We consider two types of datasets to generate $\Q$. 
\begin{itemize}[leftmargin=13pt]
\item {\bf Simulated instances.} Following the approach in \cite{chen2023montecarlopolicygradient}, we generate the matrix $\textbf{Q}$ as follows. For $n<10^4$, $\textbf{Q}$ is a dense matrix with a density of $0.8$, and its nonzero entries are drawn independently from the uniform distribution over interval $[10, 100]$. For large-scale problems with $n \geqslant 10^4$, $\textbf{Q}$ is generated as a sparse matrix with a density of 0.005.

\item  {\bf BiqBin instances.} Instances from the BiqBin library\footnote{\texttt{http://www.biqbin.eu/Home/BenchmarkInstances}}  is a standard benchmark collection for binary quadratic optimization. In our experiments, we select 100 instances from this library, with variable dimensions ranging from 100 to 250 and densities ranging from 0.1 to 1, so as to provide a more comprehensive evaluation of the proposed method.
 \end{itemize}

\subsubsection{Benchmark methods}  For the simulated instances, the benchmark algorithms include MEPM, L2ADMM, GUROBI, and SDPNAL+ \cite{sun2020sdpnal+}. The hyperparameters are given as follows. For MEPM and L2ADMM, we set ${L = \|\textbf{Q}\|_F^2/n}$, while other parameters remain consistent with those used in previous experiments. Default settings are adopted for GUROBI and SDPNAL+. All algorithms are restricted to a maximum runtime of 600 seconds if $n\leqslant 10^4$ and one hour otherwise. 
 
For the BiqBin instances, we compare ShaPeak with the BiqBin solver. Again, we generate 100 initial points for $\x^0$ from the uniform distribution on $[0,1]$, run the algorithm in batch mode, and report the best result obtained. We use the default parameter settings of the BiqBin solver and set the maximum running time to 600 seconds as well.  

Let $\mathrm{obj}$ denote the objective value obtained by a given algorithm, and let $\mathrm{lowest}$ refers to the lowest objective value achieved among all compared methods. To evaluate the performance, we compute the relative optimality gap defined as follows:
 \begin{eqnarray*}\label{def-gap-obj}
 	\text{Gap} = \frac{\text{obj}-\text{lowest}}{|\text{lowest}|} \times 100\%.
 \end{eqnarray*} 

 \begin{table}[!t]
	\renewcommand{\arraystretch}{1.0}\addtolength{\tabcolsep}{-1.5pt}
	\centering
	\caption{Gap for QUBO with low dimensional simulated instances.}
	\begin{tabular}{ccc ccc ccc ccc ccc ccc}
		\hline
		 & \multicolumn{2}{c}{ShaPeakg} && \multicolumn{2}{c}{ShaPeakh} && \multicolumn{2}{c}{MEPM} && \multicolumn{2}{c}{L2ADMM}&& \multicolumn{2}{c}{GUROBI} && \multicolumn{2}{c}{SDPNAL+} \\ \cline{2-3}\cline{5-6}\cline{8-9}\cline{11-12}\cline{14-15}\cline{17-18}
		$n$& best & mean && best & mean && best & mean && best & mean && best & mean && best & mean \\ \hline
		1000 & 0.19 & 0.58 && 0.18 & 0.55 && 1.54 & 4.64 && 0.81 & 2.62 && 0.00 & 0.00 && 3.86 & 5.27 \\
		3000 & 0.46 & 0.59 && 0.46 & 0.59 && 3.64 & 4.08 && 3.27 & 3.64 && 0.00 & 0.00 && 5.01 & 6.02 \\
		5000 & 0.39 & 0.52 && 0.41 & 0.55 && 4.17 & 4.54 && 3.69 & 4.10 && 0.00 & 0.00 && 8.41 & 10.0 \\
		7000 & 0.28 & 0.41 && 0.24 & 0.39 && 4.85 & 5.11 && 4.35 & 4.62 && 0.00 & 0.00 && 42.5 & 47.5 \\
		10000 & 0.24 & 0.37 && 0.21 & 0.37 && 5.32 & 5.81 && 4.84 & 5.47 && 0.00 & 0.00 && 64.3 & 67.9 \\ \hline
	\end{tabular}
	\label{tab:qubo-gap-low}
\end{table} 

\begin{table}[!t]
	\renewcommand{\arraystretch}{1.0}\addtolength{\tabcolsep}{5.5pt}
	\centering
	\caption{Time (in seconds) for QUBO with  low dimensional simulated instances.}
	\begin{tabular}{ccccccc}
		\hline
		$n$ & ShaPeakg & ShaPeakh & MEPM & L2ADMM & GUROBI & SDPNAL+ \\ \hline
		1000 & 0.062 & 0.053 & 0.337 & 0.447 & 600.0 & 176.6 \\
		3000 & 0.299 & 0.271 & 2.956 & 3.605 & 600.0 & 600.0 \\
		5000 & 1.296 & 1.224 & 21.65 & 26.38 & 600.0 & 600.0 \\
		7000 & 2.748 & 2.696 & 47.50 & 58.01 & 600.0 & 600.0 \\
		10000 & 5.854 & 5.934 & 109.5 & 130.6 & 600.0 & 600.0 \\ \hline
	\end{tabular}
	\label{tab:qubo-time-low}
\end{table}

\begin{table}[!t]
	\renewcommand{\arraystretch}{1.05}\addtolength{\tabcolsep}{-3.5pt}
	\centering
	\caption{Results for QUBO with higher dimensional simulated instances.}
	\begin{tabular}{ccccccccccccccccc}
		\hline
		& &  \multicolumn{3}{c}{ShaPeakg} && \multicolumn{3}{c}{ShaPeakh} && \multicolumn{3}{c}{MEPM} && \multicolumn{3}{c}{L2ADMM} \\ \cline{3-5}\cline{7-9}\cline{11-13}\cline{15-17}
		$n$ & Case & Best & Mean & Time(s) && Best & Mean & Time(s) && Best & Mean & Time(s) && Best & Mean & Time(s) \\
		\hline
		\multirow{5}{*}{8e4} & 1 & 0.01 & 0.02 & 33.59 && 0.00 & 0.00 & 33.29 && 1.55 & 1.59 & 601.1 && 1.30 & 1.31 & 778.7 \\
		 & 2 & 0.00 & 0.01 & 32.72 && 0.00 & 0.01 & 33.48 && 1.59 & 1.62 & 615.8 && 1.32 & 1.34 & 783.2 \\
		 & 3 & 0.00 & 0.02 & 32.65 && 0.00 & 0.00 & 32.70 && 1.56 & 1.58 & 620.5 && 1.28 & 1.31 & 770.7 \\
		 & 4 & 0.00 & 0.01 & 33.07 && 0.00 & 0.00 & 33.64 && 1.59 & 1.60 & 615.6 && 1.30 & 1.32 & 786.3 \\
		 & 5 & 0.00 & 0.01 & 32.98 && 0.00 & 0.00 & 32.91 && 1.60 & 1.61 & 630.9 && 1.33 & 1.35 & 793.9 \\
		 \hline
		\multirow{5}{*}{1e5} & 1 & 0.00 & 0.01 & 51.73 && 0.00 & 0.00 & 50.84 && 1.69 & 1.72 & 993.6 && 1.41 & 1.43 & 1253.1 \\
		 & 2 & 0.01 & 0.02 & 52.19 && 0.00 & 0.00 & 51.94 && 1.72 & 1.76 & 1011.4 && 1.44 & 1.47 & 1266.1 \\
		 & 3 & 0.00 & 0.01 & 52.09 && 0.00 & 0.00 & 51.53 && 1.73 & 1.75 & 998.8 && 1.45 & 1.46 & 1254.0 \\
		 & 4 & 0.00 & 0.01 & 51.84 && 0.00 & 0.00 & 51.33 && 1.71 & 1.77 & 1000.2 && 1.41 & 1.47 & 1258.0 \\
		 & 5 & 0.00 & 0.01 & 52.31 && 0.00 & 0.00 & 50.98 && 1.64 & 1.70 & 1001.4 && 1.36 & 1.41 & 1250.9 \\
		\hline
	\end{tabular}
	\label{tab:qubo-high}
\end{table}

\begin{table}[!t]
\renewcommand{\arraystretch}{1.0}\addtolength{\tabcolsep}{4pt}
	\centering
	\caption{Results for QUBO with larger-scale BiqBin instances .}
	\label{tab:qubo_large_BB}
	\begin{tabular}{lccccccc}
		\toprule
		& & \multicolumn{2}{c}{ShaPeakg} & \multicolumn{2}{c}{ShaPeakh} & \multicolumn{2}{c}{BiqBin} \\
		\cmidrule(lr){3-4}\cmidrule(lr){5-6}\cmidrule(lr){7-8}
		Group & Case & Gap (\%) & Time (s) & Gap (\%) & Time (s) & Gap (\%) & Time (s) \\
		\midrule
		be250  & 1  & 0.00 & 11.84 & 0.00 & 11.26 & 0.00 & 78.59 \\
		be250  & 2  & 0.00 & 10.48 & 0.00 & 10.17 & 0.00 & 62.09 \\
		be250  & 3  & 0.00 & 10.53 & 0.00 & 10.30 & 0.00 & 32.71 \\
		be250  & 4  & 0.00 & 10.90 & 0.00 & 11.49 & 0.00 & 69.92 \\
		be250  & 5  & 0.00 & 11.91 & 0.00 & 10.32 & 0.00 & 82.59 \\
		be250  & 6  & 0.00 & 10.03 & 0.00 & 10.10 & 0.00 & 51.24 \\
		be250  & 7  & 0.00 & 10.80 & 0.00 & 10.84 & 0.00 & 22.34 \\
		be250  & 8  & 0.00 & 11.00 & 0.00 & 10.48 & 0.00 & 69.28 \\
		be250  & 9  & 0.00 & 10.81 & 0.00 & 10.33 & 0.00 & 104.4 \\
		be250  & 10 & 0.00 & 10.96 & 0.00 & 10.57 & 0.00 & 90.12 \\
		\midrule
		bqp250 & 1  & 0.00 & 12.48 & 0.00 & 11.84 & 0.00 & 69.51 \\
		bqp250 & 2  & 0.00 & 11.82 & 0.00 & 11.34 & 0.00 & 36.95 \\
		bqp250 & 3  & 0.00 & 12.78 & 0.00 & 13.51 & 0.00 & 31.91 \\
		bqp250 & 4  & 0.00 & 14.22 & 0.00 & 12.40 & 0.00 & 62.38 \\
		bqp250 & 5  & 0.00 & 13.34 & 0.00 & 12.79 & 0.00 & 55.65 \\
		bqp250 & 6  & 0.00 & 12.76 & 0.00 & 11.68 & 0.00 & 111.8 \\
		bqp250 & 7  & 0.00 & 12.77 & 0.00 & 12.13 & 0.00 & 71.73 \\
		bqp250 & 8  & 0.17 & 11.97 & 0.17 & 11.63 & 0.00 & 612.8 \\
		bqp250 & 9  & 0.00 & 12.65 & 0.00 & 13.69 & 0.00 & 72.98 \\
		bqp250 & 10 & 0.00 & 11.55 & 0.00 & 11.74 & 0.00 & 57.72 \\
		\bottomrule
	\end{tabular}
\end{table}

\subsubsection{Numerical comparison} For the simulated instances,  Tables~\ref{tab:qubo-gap-low} and \ref{tab:qubo-time-low} report results over 20 independent runs for small-scale problems, where `best' and `mean' denote the best and average Gap. In all scenarios, GUROBI achieves the smallest gap, while ShaPeakg and  ShaPeakh rank next and exhibit a significant speed advantage.  We then extend the comparison to higher dimensional settings, ${n=80000}$ and ${n=10^5}$, excluding GUROBI and SDPNAL+ due to excessive computational time.  As shown in Table~\ref{tab:qubo-high}, where case 1 - case 5 correspond to $5$ different proportions $\{0.4995, 0.4999, 0.49995, 0.49999, 0.5\}$ of negative elements in matrix $\Q$, ShaPeakg and  ShaPeakh achieve lower objective values with faster computational speed than MEPM and L2ADMM. For instance, in case 5 and  ${n = 10^5}$, ShaPeakg and ShaPeakh consume 52.31 and 50.98 seconds, compared with 1001.4 and 1250.9 seconds for MEPM and L2ADMM. 

For the BiqBin instances, our main results focus on the 20 larger instances, namely, be250.1–be250.10 and bqp250.1–bqp250.10, each of which has 250 nodes.  As shown in Table \ref{tab:qubo_large_BB}, ShaPeak attains the same objective value as BiqBin on 19 out of the 20 instances, with the only exception being bqp250.8. In terms of running time, ShaPeak also outperforms BiqBin. For example, on be250.9, both ShaPeakg and ShaPeakh require less than 20 seconds, whereas BiqBin needs 104.4 seconds. The complete results are reported in Table \ref{tab:biqbin_results} in Appendix C.
%
\section{Conclusion}

The proposed sharp-peak functions reformulate binary integer programming problems into constrained optimization problems, allowing continuous optimization techniques to address inherently discrete challenges. This class of functions not only leads to a new exact penalty theorem but also enhances the algorithm's performance theoretically and numerically. To demonstrate its effectiveness, we conducted numerical experiments on signal recovery, MIMO detection, Max-Cut, and QUBO problems. Nevertheless, the current experimental validation is still limited, and a broader empirical study on more diverse and larger-scale benchmarks would be valuable. Future work includes generalizing the framework to mixed-integer optimization with additional constraints, and developing distributed algorithms for large-scale problems.

\section*{Acknowledgments}
This work was funded by the  National Key R\&D Program of China (No. 2023YFA1011100 and No. 2024YFA1012901) and the Natural Science Foundation of China (No.12501455). The authors are grateful to the two anonymous referees for their constructive comments, which have helped to improve the quality of the paper.

\vskip 0.2in
\bibliographystyle{abbrv}

\bibliography{refs}

\newpage
\appendix

\section*{Appendix A: Variance for classical MIMO detection} \label{std-mimo}

{\renewcommand{\arraystretch}{1.1}\addtolength{\tabcolsep}{1.5pt}
\begin{longtable}{c ccccc c ccccc}	
	\caption{Standard deviation of CPU time (in seconds) for classical MIMO detection, where A1-A5 stand for ShaPeakg, ShaPeakh, MEPM, L2ADMM, and HOTML, respectively.}\label{tab:MIMO-time-std}
\\\midrule 
		&  \multicolumn{5}{c}{IID channels} && \multicolumn{5}{c}{Correlated channels} \\
		\cmidrule(lr){2-6} \cmidrule(lr){8-12}
		SNR & A1 & A2 & A3 & A4 & A5 && A1 & A2 & A3 & A4 & A5  \\

	\midrule
    \endfirsthead
	\toprule
	&  \multicolumn{5}{c}{IID channels} && \multicolumn{5}{c}{Correlated channels} \\	
	\cmidrule(lr){2-6} \cmidrule(lr){8-12}	
	SNR & A1 & A2 & A3 & A4 & A5 && A1 & A2 & A3 & A4 & A5  \\
	\endhead
	\midrule
	\multicolumn{12}{r}{Continued on next page} \\
	\midrule
	\endfoot
	\bottomrule
	\endlastfoot
	
		&\multicolumn{11}{c}{$n=8000$}\\\midrule
		0 & 0.171 & 0.136 & 0.135 & 0.139 & 0.555 && 0.087 & 0.127 & 0.789 & 0.186 & 0.607 \\
		2 & 0.189 & 0.197 & 0.176 & 0.125 & 0.424 && 0.151 & 0.132 & 0.822 & 2.870 & 0.445 \\
		4 & 0.206 & 0.192 & 0.220 & 0.166 & 0.541 && 0.185 & 0.264 & 0.707 & 2.002 & 0.655 \\
		6 & 0.180 & 0.169 & 0.200 & 0.528 & 0.822 && 0.141 & 0.160 & 0.689 & 0.217 & 0.423 \\
		8 & 0.151 & 0.141 & 0.598 & 0.175 & 0.445 && 0.118 & 0.118 & 0.935 & 0.055 & 0.579 \\
		10 & 0.185 & 0.182 & 0.665 & 0.191 & 0.415 && 0.025 & 0.112 & 0.795 & 0.045 & 0.681 \\\midrule
		&\multicolumn{11}{c}{$n=12000$}\\\midrule
		0 & 0.081 & 0.078 & 1.760 & 2.440 & 2.472 && 0.006 & 0.009 & 0.219 & 0.139 & 0.524 \\
		2 & 0.080 & 0.081 & 2.044 & 1.257 & 3.549 && 0.010 & 0.008 & 1.256 & 6.210 & 0.493 \\
		4 & 0.012 & 0.005 & 0.270 & 0.241 & 1.124 && 0.013 & 0.013 & 1.868 & 1.100 & 1.211 \\
		6 & 0.023 & 0.009 & 0.170 & 0.222 & 3.019 && 0.049 & 0.011 & 1.679 & 0.277 & 0.764 \\
		8 & 0.007 & 0.067 & 0.267 & 0.270 & 1.455 && 0.011 & 0.014 & 1.525 & 0.487 & 1.762 \\
		10 & 0.035 & 0.005 & 0.561 & 0.089 & 1.395 && 0.011 & 0.034 & 1.747 & 4.571 & 0.926 \\\midrule
		&\multicolumn{11}{c}{$n=16000$}\\\midrule
		0 & 0.205 & 0.213 & 4.912 & 7.376 & 14.902 && 0.097 & 0.114 & 3.367 & 4.514 & 4.706 \\
		2 & 0.211 & 0.145 & 3.314 & 4.894 & 7.433 && 0.104 & 0.131 & 5.471 & 14.100 & 4.107 \\
		4 & 0.124 & 0.129 & 3.786 & 4.961 & 9.391 && 0.119 & 0.120 & 3.070 & 4.330 & 2.379 \\
		6 & 0.127 & 0.134 & 4.159 & 5.493 & 6.441 && 0.039 & 0.064 & 3.710 & 10.807 & 4.133 \\
		8 & 0.135 & 0.130 & 3.980 & 5.363 & 6.521 && 0.081 & 0.031 & 3.914 & 0.994 & 3.726 \\
		10 & 0.093 & 0.150 & 3.896 & 6.736 & 6.104 && 0.021 & 0.063 & 3.691 & 4.157 & 1.758 \\\midrule
		&\multicolumn{11}{c}{$n=20000$}\\\midrule
		0 & 0.354 & 0.398 & 9.508 & 12.238 & 20.959 && 1.475 & 1.867 & 2.971 & 1.067 & 6.156 \\
		2 & 0.341 & 0.331 & 9.675 & 11.203 & 18.576 && 1.281 & 1.711 & 2.501 & 15.984 & 8.138 \\
		4 & 0.340 & 0.197 & 5.120 & 6.917 & 21.079 && 1.107 & 1.524 & 2.932 & 21.131 & 5.542 \\
		6 & 0.187 & 0.191 & 5.050 & 6.617 & 13.004 && 1.519 & 1.685 & 3.100 & 2.129 & 6.291 \\
		8 & 0.186 & 0.186 & 4.628 & 6.939 & 13.103 && 1.031 & 1.122 & 2.575 & 0.759 & 8.785 \\
		10 & 0.247 & 0.229 & 8.736 & 13.264 & 22.332 && 0.478 & 0.102 & 4.547 & 18.098 & 4.120 \\
\end{longtable}}

\newpage

%
%

\section*{Appendix B: Complete results for Max-Cut}\label{com-re-maxcut}

{\renewcommand{\arraystretch}{1.0}\addtolength{\tabcolsep}{-0.5pt}
\begin{longtable}{lcccccccc}
	\caption{Complete  results for Max-Cut with Gset instances .}\label{tab:mc_gset_all}\\
	\toprule
	& \multicolumn{2}{c}{ShaPeakg} & \multicolumn{2}{c}{CRA} & \multicolumn{2}{c}{PIGNN} & \multicolumn{2}{c}{MCPG} \\
	\cmidrule(lr){2-3}\cmidrule(lr){4-5}\cmidrule(lr){6-7}\cmidrule(lr){8-9}
	Graphs & Gap (\%) & Time (s) & Gap (\%) & Time (s) & Gap (\%) & Time (s) & Gap (\%) & Time (s) \\
	\midrule
	\endfirsthead
	\toprule
	& \multicolumn{2}{c}{ShaPeakg} & \multicolumn{2}{c}{CRA} & \multicolumn{2}{c}{PIGNN} & \multicolumn{2}{c}{MCPG} \\
	\cmidrule(lr){2-3}\cmidrule(lr){4-5}\cmidrule(lr){6-7}\cmidrule(lr){8-9}
	Graphs & Gap (\%) & Time (s) & Gap (\%) & Time (s) & Gap (\%) & Time (s) & Gap (\%) & Time (s) \\
	\midrule
	\endhead
	\midrule
	\multicolumn{9}{r}{Continued on next page} \\
	\midrule
	\endfoot
	\bottomrule
	\endlastfoot
	G22 & 0.04 & 60.44 & 1.81 & 288.2 & 8.16 & 307.9 & 0.00 & 107.3 \\
	G23 & 0.06 & 67.31 & 1.76 & 298.8 & 7.91 & 317.1 & 0.06 & 108.8 \\
	G24 & 0.10 & 68.17 & 2.05 & 359.4 & 8.02 & 255.0 & 0.05 & 111.3 \\
	G25 & 0.10 & 55.55 & 1.63 & 348.0 & 7.50 & 277.3 & 0.06 & 110.8 \\
	G26 & 0.16 & 61.75 & 2.27 & 337.9 & 8.27 & 301.6 & 0.03 & 110.5 \\
	G27 & 0.45 & 45.50 & 17.78 & 453.2 & 21.31 & 836.0 & 0.03 & 111.5 \\
	G28 & 0.21 & 65.00 & 18.92 & 491.8 & 22.41 & 962.5 & 0.03 & 110.1 \\
	G29 & 0.44 & 64.36 & 19.44 & 437.8 & 19.85 & 995.3 & 0.35 & 110.6 \\
	G30 & 0.44 & 56.94 & 19.49 & 475.4 & 20.90 & 990.0 & 0.00 & 109.3 \\
	G31 & 0.54 & 57.19 & 20.04 & 485.9 & 21.09 & 973.8 & 0.09 & 110.4 \\
	G32 & 1.13 & 92.73 & 18.01 & 391.2 & 15.74 & 796.3 & 0.57 & 110.7 \\
	G33 & 1.01 & 63.76 & 18.96 & 388.0 & 18.23 & 932.3 & 0.72 & 109.3 \\
	G34 & 0.29 & 50.40 & 18.79 & 429.3 & 18.06 & 849.9 & 0.58 & 112.4 \\
	G35 & 0.53 & 70.49 & 3.05 & 337.9 & 8.64 & 231.3 & 0.21 & 113.3 \\
	G36 & 0.47 & 82.99 & 3.16 & 345.7 & 9.26 & 298.4 & 0.25 & 117.1 \\
	G37 & 0.78 & 79.02 & 2.82 & 358.4 & 8.23 & 309.0 & 0.27 & 117.6 \\
	G38 & 0.39 & 73.71 & 3.11 & 311.6 & 8.49 & 370.6 & 0.23 & 117.6 \\
	G39 & 1.04 & 56.91 & 12.17 & 409.6 & 18.52 & 677.2 & 0.00 & 116.8 \\
	G40 & 1.33 & 56.86 & 12.08 & 442.3 & 20.71 & 883.1 & 0.04 & 110.2 \\
	G41 & 0.96 & 64.08 & 12.02 & 443.4 & 19.67 & 778.7 & 0.00 & 110.9 \\
	G42 & 0.81 & 79.60 & 11.21 & 470.5 & 20.35 & 764.1 & 0.00 & 112.5 \\
	G43 & 0.36 & 31.04 & 1.83 & 352.0 & 6.82 & 218.0 & 0.00 & 61.08 \\
	G44 & 0.09 & 41.78 & 1.52 & 354.6 & 7.28 & 206.6 & 0.00 & 58.65 \\
	G45 & 0.26 & 42.07 & 1.98 & 361.7 & 7.50 & 199.7 & 0.00 & 58.48 \\
	G46 & 0.08 & 45.00 & 1.28 & 348.1 & 7.70 & 253.8 & 0.00 & 59.96 \\
	G47 & 0.12 & 40.58 & 2.27 & 371.2 & 7.96 & 173.4 & 0.00 & 57.63 \\
	G48 & 0.00 & 12.07 & 2.30 & 325.6 & 15.53 & 292.1 & 0.00 & 174.2 \\
	G49 & 0.00 & 11.87 & 1.73 & 315.4 & 17.00 & 284.5 & 0.00 & 163.7 \\
	G50 & 0.03 & 107.8 & 2.01 & 318.7 & 14.15 & 314.1 & 0.03 & 165.2 \\
	G51 & 0.34 & 67.36 & 3.01 & 357.1 & 11.28 & 357.0 & 0.08 & 56.83 \\
	G52 & 0.44 & 48.56 & 2.75 & 336.2 & 9.35 & 292.2 & 0.03 & 57.46 \\
	G53 & 0.47 & 52.56 & 2.91 & 387.2 & 8.73 & 277.8 & 0.05 & 56.82 \\
	G54 & 0.36 & 59.82 & 2.83 & 334.3 & 10.62 & 277.9 & 0.03 & 56.55 \\
	G55 & 0.38 & 61.67 & 2.80 & 322.4 & 12.02 & 311.5 & 0.42 & 279.0 \\
	G56 & 1.07 & 61.00 & 15.00 & 386.4 & 21.11 & 941.1 & 0.80 & 279.4 \\
	G57 & 1.60 & 51.20 & 15.75 & 413.9 & 19.30 & 933.7 & 1.72 & 275.2 \\
	G58 & 0.60 & 102.2 & 2.88 & 321.5 & 10.25 & 448.9 & 0.49 & 276.9 \\
	G59 & 1.22 & 87.57 & 8.14 & 423.1 & 21.88 & 976.8 & 0.67 & 281.2 \\
	G60 & 0.39 & 78.92 & 2.97 & 314.4 & 12.47 & 412.1 & 0.47 & 393.1 \\
	G61 & 1.12 & 90.20 & 14.49 & 359.6 & 21.07 & 908.5 & 1.28 & 391.9 \\
	G62 & 1.73 & 58.26 & 15.65 & 376.3 & 19.84 & 903.0 & 2.10 & 394.0 \\
	G63 & 3.47 & 52.22 & 3.00 & 288.2 & 11.18 & 568.0 & 0.53 & 399.2 \\
	G64 & 1.37 & 45.97 & 7.19 & 380.9 & 20.48 & 781.3 & 1.10 & 396.6 \\
	G65 & 1.69 & 39.74 & 13.31 & 359.0 & 20.15 & 845.4 & 2.16 & 446.9 \\
	G66 & 2.04 & 42.29 & 13.84 & 368.5 & 19.25 & 935.7 & 2.74 & 518.9 \\
	G67 & 1.73 & 40.29 & 14.47 & 343.1 & 19.94 & 991.0 & 2.48 & 570.3 \\
	G70 & 0.93 & 53.59 & 3.78 & 327.9 & 10.88 & 314.3 & 0.27 & 564.7 \\
	G72 & 1.91 & 60.26 & 13.52 & 393.1 & 19.69 & 868.5 & 2.54 & 576.7 \\
	G77 & 1.83 & 47.28 & 12.15 & 324.9 & 21.44 & 952.4 & 2.50 & 796.8 \\
	G81 & 2.10 & 54.87 & 11.38 & 361.1 & 20.96 & 187.0 & 2.82 & 1115.7 \\
\end{longtable}
}

\section*{Appendix C: Complete results for  QUBO with BiqBin instances}\label{com-BiqBin}

{\renewcommand{\arraystretch}{1.0}\addtolength{\tabcolsep}{3.5pt}
\begin{longtable}{lccccccc}
	\caption{Complete results for QUBO with all BiqBin instances.}
	\label{tab:biqbin_results}\\
	\toprule
	& & \multicolumn{2}{c}{ShaPeakg} & \multicolumn{2}{c}{ShaPeakh} & \multicolumn{2}{c}{BiqBin} \\
	\cmidrule(lr){3-4}\cmidrule(lr){5-6}\cmidrule(lr){7-8}
	Group & Case & Gap (\%) & Time (s) & Gap (\%) & Time (s) & Gap (\%) & Time (s) \\
	\midrule
	\endfirsthead
	
	\toprule
	& & \multicolumn{2}{c}{ShaPeakg} & \multicolumn{2}{c}{ShaPeakh} & \multicolumn{2}{c}{BiqBin} \\
	\cmidrule(lr){3-4}\cmidrule(lr){5-6}\cmidrule(lr){7-8}
	Group & Case & Gap (\%) & Time (s) & Gap (\%) & Time (s) & Gap (\%) & Time (s) \\
	\midrule
	\endhead
	
	\midrule
	\multicolumn{8}{r}{Continued on next page} \\
	\midrule
	\endfoot
	
	\bottomrule
	\endlastfoot
	
	be100 & 1 & 0.00 & 12.89 & 0.00 & 14.28 & 0.00 & 6.55 \\
	be100 & 2 & 0.00 & 10.40 & 0.00 & 11.52 & 0.00 & 3.90 \\
	be100 & 3 & 0.00 & 11.24 & 0.00 & 11.33 & 0.00 & 3.65 \\
	be100 & 4 & 0.00 & 10.61 & 0.00 & 11.90 & 0.00 & 7.71 \\
	be100 & 5 & 0.00 & 10.35 & 0.00 & 10.59 & 0.00 & 5.41 \\
	be100 & 6 & 0.00 & 11.83 & 0.00 & 12.24 & 0.00 & 3.16 \\
	be100 & 7 & 0.00 & 11.29 & 0.00 & 11.09 & 0.00 & 5.79 \\
	be100 & 8 & 0.00 & 10.91 & 0.00 & 11.66 & 0.00 & 6.69 \\
	be100 & 9 & 0.00 & 10.73 & 0.00 & 11.06 & 0.00 & 8.24 \\
	be100 & 10 & 0.00 & 11.22 & 0.00 & 11.88 & 0.00 & 4.71 \\\midrule
	
	be120.3 & 1 & 0.00 & 8.22 & 0.00 & 8.94 & 0.00 & 9.44 \\
	be120.3 & 2 & 0.00 & 8.65 & 0.00 & 9.39 & 0.00 & 3.39 \\
	be120.3 & 3 & 0.00 & 8.70 & 0.00 & 9.11 & 0.00 & 4.26 \\
	be120.3 & 4 & 0.00 & 9.53 & 0.00 & 9.43 & 0.00 & 5.66 \\
	be120.3 & 5 & 0.00 & 8.22 & 0.00 & 8.98 & 0.00 & 4.30 \\
	be120.3 & 6 & 0.00 & 9.67 & 0.00 & 9.87 & 0.00 & 3.02 \\
	be120.3 & 7 & 0.00 & 9.00 & 0.00 & 9.44 & 0.00 & 2.80 \\
	be120.3 & 8 & 0.00 & 8.52 & 0.00 & 9.21 & 0.00 & 2.16 \\
	be120.3 & 9 & 0.00 & 7.90 & 0.00 & 8.58 & 0.00 & 6.36 \\
	be120.3 & 10 & 0.00 & 8.04 & 0.00 & 8.70 & 0.00 & 6.93 \\\midrule
	
	be120.8 & 1 & 0.00 & 11.19 & 0.00 & 11.98 & 0.00 & 21.15 \\
	be120.8 & 2 & 0.00 & 11.09 & 0.00 & 10.94 & 0.00 & 8.34 \\
	be120.8 & 3 & 0.00 & 11.01 & 0.00 & 10.80 & 0.00 & 8.88 \\
	be120.8 & 4 & 0.00 & 11.63 & 0.00 & 12.40 & 0.00 & 16.18 \\
	be120.8 & 5 & 0.00 & 11.58 & 0.00 & 11.22 & 0.00 & 7.96 \\
	be120.8 & 6 & 0.00 & 11.00 & 0.00 & 10.70 & 0.00 & 9.74 \\
	be120.8 & 7 & 0.00 & 10.69 & 0.00 & 11.12 & 0.00 & 26.94 \\
	be120.8 & 8 & 0.00 & 12.40 & 0.00 & 12.26 & 0.00 & 25.12 \\
	be120.8 & 9 & 0.00 & 11.42 & 0.00 & 11.26 & 0.00 & 15.18 \\
	be120.8 & 10 & 0.00 & 11.18 & 0.00 & 11.95 & 0.00 & 10.57 \\\midrule
	
	be150.3 & 1 & 0.00 & 10.95 & 0.00 & 10.47 & 0.00 & 10.92 \\
	be150.3 & 2 & 0.00 & 11.62 & 0.00 & 10.31 & 0.00 & 13.22 \\
	be150.3 & 3 & 0.00 & 10.20 & 0.00 & 9.70 & 0.00 & 7.30 \\
	be150.3 & 4 & 0.00 & 10.55 & 0.00 & 9.64 & 0.00 & 18.51 \\
	be150.3 & 5 & 0.00 & 9.81 & 0.00 & 9.93 & 0.00 & 15.49 \\
	be150.3 & 6 & 0.00 & 10.58 & 0.00 & 11.34 & 0.00 & 42.69 \\
	be150.3 & 7 & 0.00 & 10.35 & 0.00 & 9.87 & 0.00 & 16.22 \\
	be150.3 & 8 & 0.00 & 10.81 & 0.00 & 10.35 & 0.00 & 56.98 \\
	be150.3 & 9 & 0.00 & 9.82 & 0.00 & 9.20 & 0.00 & 54.51 \\
	be150.3 & 10 & 0.00 & 10.44 & 0.00 & 10.15 & 0.00 & 51.30 \\\midrule
	
	be150.8 & 1 & 0.00 & 12.82 & 0.00 & 13.09 & 0.00 & 33.68 \\
	be150.8 & 2 & 0.00 & 13.13 & 0.00 & 12.46 & 0.00 & 106.3 \\
	be150.8 & 3 & 0.00 & 13.47 & 0.00 & 13.03 & 0.00 & 70.63 \\
	be150.8 & 4 & 0.00 & 12.84 & 0.00 & 13.94 & 0.00 & 17.61 \\
	be150.8 & 5 & 0.00 & 14.96 & 0.00 & 13.39 & 0.00 & 43.20 \\
	be150.8 & 6 & 0.00 & 13.71 & 0.00 & 13.02 & 0.00 & 71.78 \\
	be150.8 & 7 & 0.14 & 13.97 & 0.14 & 13.52 & 0.00 & 59.65 \\
	be150.8 & 8 & 0.00 & 13.28 & 0.00 & 13.54 & 0.00 & 36.62 \\
	be150.8 & 9 & 0.00 & 11.52 & 0.00 & 13.85 & 0.00 & 85.70 \\
	be150.8 & 10 & 0.00 & 14.17 & 0.00 & 13.23 & 0.00 & 32.76 \\\midrule
	
	be200.3 & 1 & 0.00 & 13.07 & 0.00 & 12.04 & 0.00 & 618.3 \\
	be200.3 & 2 & 0.00 & 11.91 & 0.00 & 10.99 & 0.00 & 448.7 \\
	be200.3 & 3 & 0.00 & 11.92 & 0.00 & 11.47 & 0.00 & 116.8 \\
	be200.3 & 4 & 0.00 & 12.70 & 0.00 & 13.04 & 0.00 & 413.4 \\
	be200.3 & 5 & 0.00 & 12.48 & 0.00 & 12.15 & 0.00 & 307.5 \\
	be200.3 & 6 & 0.00 & 12.49 & 0.00 & 11.99 & 0.00 & 136.1 \\
	be200.3 & 7 & 0.00 & 12.26 & 0.00 & 11.60 & 0.00 & 21.50 \\
	be200.3 & 8 & 0.00 & 12.36 & 0.00 & 11.52 & 0.00 & 291.9 \\
	be200.3 & 9 & 0.12 & 12.94 & 0.12 & 12.65 & 0.00 & 299.7 \\
	be200.3 & 10 & 0.00 & 11.22 & 0.00 & 10.69 & 0.00 & 621.5 \\\midrule
	
	be200.8 & 1 & 0.00 & 15.64 & 0.00 & 14.85 & 0.00 & 173.0 \\
	be200.8 & 2 & 0.12 & 15.56 & 0.12 & 14.96 & 0.00 & 619.1 \\
	be200.8 & 3 & 0.00 & 15.88 & 0.00 & 14.66 & 0.00 & 611.9 \\
	be200.8 & 4 & 0.00 & 17.04 & 0.00 & 16.41 & 0.00 & 423.2 \\
	be200.8 & 5 & 0.00 & 14.41 & 0.00 & 15.18 & 0.00 & 626.5 \\
	be200.8 & 6 & 0.00 & 15.30 & 0.00 & 15.46 & 0.00 & 64.91 \\
	be200.8 & 7 & 0.00 & 14.37 & 0.00 & 14.97 & 0.00 & 291.2 \\
	be200.8 & 8 & 0.00 & 14.79 & 0.00 & 15.46 & 0.00 & 641.2 \\
	be200.8 & 9 & 0.00 & 17.46 & 0.00 & 15.61 & 0.00 & 67.70 \\
	be200.8 & 10 & 0.00 & 14.85 & 0.00 & 15.13 & 0.00 & 455.8 \\\midrule
	
	be250 & 1 & 0.00 & 11.84 & 0.00 & 11.26 & 0.00 & 78.59 \\
	be250 & 2 & 0.00 & 10.48 & 0.00 & 10.17 & 0.00 & 62.09 \\
	be250 & 3 & 0.00 & 10.53 & 0.00 & 10.30 & 0.00 & 32.71 \\
	be250 & 4 & 0.00 & 10.90 & 0.00 & 11.49 & 0.00 & 69.92 \\
	be250 & 5 & 0.00 & 11.91 & 0.00 & 10.32 & 0.00 & 82.59 \\
	be250 & 6 & 0.00 & 10.03 & 0.00 & 10.10 & 0.00 & 51.24 \\
	be250 & 7 & 0.00 & 10.80 & 0.00 & 10.84 & 0.00 & 22.34 \\
	be250 & 8 & 0.00 & 11.00 & 0.00 & 10.48 & 0.00 & 69.28 \\
	be250 & 9 & 0.00 & 10.81 & 0.00 & 10.33 & 0.00 & 104.4 \\
	be250 & 10 & 0.00 & 10.96 & 0.00 & 10.57 & 0.00 & 90.12 \\\midrule
 
	bqp100 & 1 & 0.00 & 10.10 & 0.00 & 11.08 & 0.00 & 1.37 \\
	bqp100 & 2 & 0.00 & 9.97 & 0.00 & 9.55 & 0.00 & 1.83 \\
	bqp100 & 3 & 0.00 & 10.42 & 0.00 & 10.21 & 0.00 & 0.91 \\
	bqp100 & 4 & 0.00 & 10.26 & 0.00 & 9.96 & 0.00 & 0.73 \\
	bqp100 & 5 & 0.00 & 10.46 & 0.00 & 10.12 & 0.00 & 1.04 \\
	bqp100 & 6 & 0.00 & 11.20 & 0.00 & 10.76 & 0.00 & 4.59 \\
	bqp100 & 7 & 0.00 & 11.07 & 0.00 & 11.92 & 0.00 & 1.33 \\
	bqp100 & 8 & 0.00 & 12.73 & 0.00 & 10.86 & 0.00 & 1.25 \\
	bqp100 & 9 & 0.00 & 10.67 & 0.00 & 10.16 & 0.00 & 0.79 \\
	bqp100 & 10 & 0.00 & 11.64 & 0.00 & 10.07 & 0.00 & 1.78 \\\midrule
	
	bqp250 & 1 & 0.00 & 12.48 & 0.00 & 11.84 & 0.00 & 69.51 \\
	bqp250 & 2 & 0.00 & 11.82 & 0.00 & 11.34 & 0.00 & 36.95 \\
	bqp250 & 3 & 0.00 & 12.78 & 0.00 & 13.51 & 0.00 & 31.91 \\
	bqp250 & 4 & 0.00 & 14.22 & 0.00 & 12.40 & 0.00 & 62.38 \\
	bqp250 & 5 & 0.00 & 13.34 & 0.00 & 12.79 & 0.00 & 55.65 \\
	bqp250 & 6 & 0.00 & 12.76 & 0.00 & 11.68 & 0.00 & 111.8 \\
	bqp250 & 7 & 0.00 & 12.77 & 0.00 & 12.13 & 0.00 & 71.73 \\
	bqp250 & 8 & 0.17 & 11.97 & 0.17 & 11.63 & 0.00 & 612.8 \\
	bqp250 & 9 & 0.00 & 12.65 & 0.00 & 13.69 & 0.00 & 72.98 \\
	bqp250 & 10 & 0.00 & 11.55 & 0.00 & 11.74 & 0.00 & 57.72 \\
	
\end{longtable}}



\end{document}